\documentclass[twoside,
 12pt, a4paper]{article}
\usepackage{amsmath,amsthm,amssymb,amscd}
\usepackage{bm}
\usepackage{latexsym}
\usepackage{indentfirst}
\usepackage{graphics}
\usepackage{amsmath}
\usepackage{amssymb}
\usepackage{bm}
\usepackage{latexsym}
\usepackage[mathscr]{eucal}
\makeatletter \@addtoreset{equation}{section}

\makeatother

\DeclareMathOperator{\deck}{Deck}
\DeclareMathOperator{\gal}{Gal}
\DeclareMathOperator{\id}{id}

\newtheorem{theorem}{Theorem}[section]
\newtheorem*{theorem2.1'}{Theorem 2.1$'$}
\newtheorem*{theorem2.2'}{Theorem 2.2$'$}
\newtheorem*{theorem2.3'}{Theorem 2.3$'$}

\newtheorem*{corollary1}{Corollary 1}
\newtheorem*{corollary2}{Corollary 2}

\newtheorem*{corollary}{Corollary}

\newtheorem{proposition}[theorem]{Proposition}

\newtheorem{lemma}[theorem]{Lemma}

\theoremstyle{remark}
\newtheorem{remark}{Remark}

\begin{document}

\title{Riemann surfaces and the Galois correspondence
 }

\author{Junyang Yu \\[7pt]
\textit{\normalsize To the memory of my mother}}
%
%
\date{}
\maketitle

\begin{abstract}
  In this paper we introduce a space with some additional topologies using filter bases and renew the definition of Riemann surfaces of algebraic functions. We then present a Galois correspondence between these Riemann surfaces and their deck transformation groups. We also extend the monodromy theorem to the case that the global analytic function possesses singularities, which can be non-isolated.
\end{abstract}

\section{Introduction}

Let us recall two points lying in algebra and complex analysis respectively. At first, it is known that the Galois correspondence theorem, which is called the fundamental theorem of Galois theory, is one of the most important results in modern algebra (refer to \cite{E}, \cite{J}, \cite{Mo} and~\cite{R}, etc). The Galois correspondence and related issues have been extended to a number of cases, see e.g.\ \cite{BJ}, \cite{CH}, \cite{Kh2}, \cite{Ko}, \cite{MM}, \cite{PS}, \cite{Se}, \cite{Sz} and~\cite{V}. Then, we look at algebraic functions and their Riemann surfaces. Algebraic functions are studied in both function theory and algebraic geometry. To deal with the trouble of multivaluedness of functions Riemann designed the ``Riemann surface", which is a source of some modern mathematical branches. The theory of Riemann surfaces also provides a model for developments in many research areas in mathematics. There is a lot of literature in Riemann surfaces and algebraic functions, see 
\cite{AS}, \cite{Bl}, \cite{C}, \cite{Do}, \cite{Fo}, \cite{G}, \cite{I}, \cite{Kl}, \cite{Mi}, \cite{Na}, \cite{Ne}, \cite{Sp}, \cite{Sz}, \cite{V} and~\cite{We}, etc.

Originally, a Riemann surface may be regarded as a covering space (surface) of the (extended) complex plane (or a part of it). We may consider the Galois correspondence in the case of covering spaces. In fact, there is a correspondence in covering spaces similar to the Galois correspondence in the classical Galois theory, see \cite[13d]{Fu}, \cite{Kh1} and~\cite{Kh2} (in~\cite{Kh1} and~\cite{Kh2} finite ramified coverings over Riemann surfaces were also considered). If we observe \cite[Theorem~(8.12)]{Fo}, we may expect the Galois correspondence occurs between Riemann surfaces of algebraic functions and covering transformation groups, even in general infinite cases. However, branch points become a key problem.

In order to deal with branch points and multivaluedness of functions we employ filter bases. This idea is inspired by~\cite{Y1} and~\cite{Y2} and develops therefrom. About the notions of filters and filter bases, which were originated by H.~Cartan, we refer to \cite[\S6 and~\S7 of Chapter~I]{Bo} and \cite[Chapter~X]{Du}. Using filter bases we introduce a space with some additional topologies, which we call a \textit{universal topological space} (see Section~2). By using neighborhoods in this space we can deal with branch points in a natural manner, which enables us to carry out algebraic operations of functions and germs freely. To this end, we study presheaves on a universal topological space in Section~2. In addition, we introduce some more general notions for the extension of the monodromy theorem.

In Section~3, we present a new version of the definition of Riemann surfaces of algebraic functions, where the notions of \textit{harmonious equivalence} and \textit{up-harmonious equivalence} are introduced. We also consider analytic continuations in more extensive senses and extend the monodromy theorem to the case that the global analytic function possesses singularities, which can be non-isolated.

Finally, in Section~4 we present a Galois correspondence between Riemann surfaces of algebraic functions (\textit{algebraic Riemann surfaces}, see Subsection~3.4) and their deck transformation groups (Theorems~\ref{theorem4.1} and~\ref{theorem4.10}), which may be regarded as a geometric version of the classical Galois correspondence.

\section{A universal topological space}

\subsection{A perfect filterbase structure system and a universal topological space}

We recall that a nonempty family~$\mathcal B$ of subsets of a nonempty set~$X$ is a \textit{filter base} precisely if $\mathcal B$ does not contain the empty set and the intersection of any two sets of~$\mathcal B$ contains a set of~$\mathcal B$, see \cite[\S6.3 of Chapter~I]{Bo} and \cite[Definition~(2.1) in Chapter~X]{Du}. Suppose $\mathfrak B$ is a family of filter bases in the nonempty set~$X$. Two filter bases $\mathcal B_1$ and~$\mathcal B_2$ in~$\mathfrak B$ are said to be \textit{equivalent}, denoted $\mathcal B_1\thicksim\mathcal B_2$, if $\mathcal B_1\vdash\mathcal B_2$ and $\mathcal B_2\vdash\mathcal B_1$, where ``$\vdash$" means ``\textit{be subordinate to}" (i.e.\ $\mathcal B_1\vdash\mathcal B_2$ precisely if for each $B_2\in\mathcal B_2$ there exists $B_1\in\mathcal B_1$ such that $B_1\subseteq B_2$, see \cite[Definition~(2.4) in Chapter~X]{Du}). It is obvious that this really is an equivalence relation in~$\mathfrak B$. The equivalence class of~$\mathcal B$ is denoted~$\widetilde{\mathcal B}$.

Suppose $\mathfrak B_1$ and $\mathfrak B_2$ are two families of filter bases in a (nonempty) set~$X$. If for each $\mathcal B_2\in\mathfrak B_2$ there exists $\mathcal B_1\in\mathfrak B_1$ such that $\mathcal B_1\vdash\mathcal B_2$ (resp.\ $\mathcal B_1\thicksim\mathcal B_2$) then we say that $\mathfrak B_1$ is a (resp.\ an \textit{exact}) \textit{refinement} of~$\mathfrak B_2$. If $\mathfrak B_1$ and $\mathfrak B_2$ are (resp.\ exact) refinements of one another then we say that they are \textit{compatible} (resp.\ \textit{equivalent}), denoted $\mathfrak B_1\,\dot{\thicksim}\,\mathfrak B_2$ (resp.\ $\mathfrak B_1\thicksim\mathfrak B_2$). Obviously both the exact refinement relation and the refinement relation of filterbase families are preorders and they are also partial orders if ``$=$" means ``$\thicksim$" or~``$\dot{\thicksim}$". Both the compatibility and the  equivalence are equivalence relations and the equivalence implies the compatibility.

Suppose $(X,\mathscr T)$ is a (nonempty) topological space ($\mathscr T$ is the topology on~$X$) and $\mathfrak B$ is a family of filter bases consisting of open subsets of~$X$. Suppose for each $x\in X$ there is precisely one filter base $\mathcal B\in\mathfrak B$ such that $\mathcal B\to x$ (``$\to$" means ``converge to", i.e.\ $\,\mathcal B\to x$ precisely if for any neighborhood~$U$ of~$x$ there exists $B\in\mathcal B$ such that $B\subseteq U$, see \cite[Definition~(2.3) in Chapter~X]{Du}), or if there is another filter base $\mathcal B'\in\mathfrak B$ such that $\mathcal B'\to x$ then $\mathcal B'\thicksim\mathcal B$. If $\mathcal B\in\mathfrak B$ does not converge to any points in~$X$, then we consider that it converges to some ``ideal points". Generally, it may be allowed that some filter bases (in~$\mathfrak B$) converging in~$(X,\mathscr T)$ also converge to ideal points (of course, the Hausdorff condition will remove this case). Here we consider that $\mathcal B'$ and $\mathcal B$ converge to the same ideal point(s) precisely if $\mathcal B'\thicksim\mathcal B$. We also assume that if $\mathcal B\in\mathfrak B$, $\mathcal B\to x\in X$, then $x\notin B$ for any $B\in\mathcal B$. We call the filterbase family~$\mathfrak B$ satisfying the above conditions a \textit{perfect filterbase structure system} on~$(X,\mathscr T)$.

Suppose $\mathfrak B$ is a perfect  filterbase structure system on~$X$ and $\mathtt I$ denotes the set of all ideal points of~$\mathfrak B$. Let $\hat X:=X\cup\mathtt I$. Let
$$
\hat{\mathscr B}(x):=\{\hat B=B\cup\{x\}:\, B\in\mathcal B, \text{ where } \mathcal B\in\mathfrak B \text{ and } \mathcal B\to x\}
$$
for $x\in\hat X$ and
$$
\hat{\mathscr B}:=\bigcup_{x\in\hat X}\hat{\mathscr B}(x).
$$
Noticing every filter base $\mathcal B\in\mathfrak B$ consists of open subsets of~$X$, it is easy to verify that $\hat{\mathscr B}$ is a basis for some topology on~$\hat X$. So we obtain a topology on~$\hat X$ determined by~$\hat{\mathscr B}$. We call this new topology the \textit{filterbase topology} or \textit{partial topology} on~$\hat X$ determined by~$\mathfrak B$, denoted~$\hat{\mathscr T}$. Set $\hat B\in\hat{\mathscr B}(x)$ is called a \textit{basic $($open$)$ partial neighborhood} of~$x$, $\hat{\mathscr B}(x)$ a \textit{basic $($open$)$ partial neighborhood basis} at~$x$ ($x\in\hat X$) and $\hat{\mathscr B}$ a \textit{basic $($open$)$ partial neighborhood basis} on~$\hat X$.

Let $\check{\mathscr B}(x)$ be an open neighborhood basis at~$x$ in $(X,\mathscr T)$ for $x\in X$ and $\check{\mathscr B}(x):=\hat{\mathscr B}(x)$ (or an open neighborhood basis at~$x$ in $(\hat X,\hat{\mathscr T})$) for $x\in\mathtt I$. Let
$$
\check{\mathscr B}:=\bigcup_{x\in\hat X}\check{\mathscr B}(x).
$$
Then it is easy to verify that $\check{\mathscr B}$ is a basis for some topology on (set)~$\hat X$. Again we obtain a topology on~$\hat X$ determined by~$\check{\mathscr B}$, which is called the \textit{essential topology} on~$\hat X$ and denoted~$\check{\mathscr T}$. We call $\check{\mathscr B}$ (resp.\ $\check{\mathscr B}(x)$) a \textit{basic $($open$)$ essential neighborhood basis} (resp.\ at~$x$).

Generally, choose a subset~$A$ of~$\hat X$ containing the ideal point set~$\mathtt I$, which we call a \textit{partial point set} (of~$\hat X$). Let $\hat{\mathscr B}_A(x)$ be an open neighborhood basis at~$x$ in $(X,\mathscr T)$ for $x\in X\!\setminus\!A$ and $\hat{\mathscr B}_A(x):=\hat{\mathscr B}(x)$ for $x\in A$. Then
$$
\hat{\mathscr B}_A:=\bigcup_{x\in\hat X}\hat{\mathscr B}_A(x)
$$
is also a basis for some topology on (set)~$\hat X$. In this way we obtain a topology on~$\hat X$, determined by~$\hat{\mathscr B}_A$, which we call the \textit{mixed topology} on~$\hat X$ with the partial point set~$A$ or the \textit{$A$-mixed topology} on~$\hat X$, denoted~$\hat{\mathscr T}(A)$, and $\hat{\mathscr B}_A$ (resp.\ $\hat{\mathscr B}_A(x)$) is called an \textit{$A$-mixed  neighborhood basis} (resp.\ at~$x$).

The space~$\hat X$, equipped with a filterbase topology~$\hat{\mathscr T}$, an essential topology~$\check{\mathscr T}$ and some mixed topologies~$\hat{\mathscr T}(A)$, is called a \textit{universal topological space} on $(X,\mathscr T)$ determined by~$\mathfrak B$. The universal topological space~$\hat X$ on $(X,\mathscr T)$ is also denoted $(X,\mathscr T;\hat{X},\hat{\mathscr T})$. The set~$\mathtt I$ is also called the \textit{ideal point set} of~$\hat X$.

We remark here that if $A=\hat X$ then $\hat{\mathscr T}(A)=\hat{\mathscr T}$ and if $A=\mathtt I$ then $\hat{\mathscr T}(A)=\check{\mathscr T}$. We also remark that for the topological space $(X,\mathscr T)$ we can obtain a number of filterbase topologies by different perfect filterbase structure systems on~$X$ (see e.g.\ Subsection~2.5). Suppose $\hat{\mathscr T}_1$ and $\hat{\mathscr T}_2$ are filterbase topologies determined by two perfect filterbase structure systems~$\mathfrak B_1$ and~$\mathfrak B_2$ on~$(X,\mathscr T)$, respectively. If $\mathfrak B_1\,\dot{\thicksim}\,\mathfrak B_2$ then it follows that $\hat{\mathscr T}_1=\hat{\mathscr T}_2$ (under some assumption on ideal points) and generally the two mixed topologies $\hat{\mathscr T}_1(A)$ and~$\hat{\mathscr T}_2(A)$ are the same. Therefore, the universal topological spaces $(X,\mathscr T;\hat{X},\hat{\mathscr T}_1)$ and $(X,\mathscr T;\hat{X},\hat{\mathscr T}_2)$ we obtain here are the same.

If the topological space $(\hat X,\check{\mathscr T})$ is Hausdorff, then we say that the universal topological space~$\hat X$ is \textit{Hausdorff}. Thus that $(X,\mathscr T;\hat{X},\hat{\mathscr T})$ is Hausdorff implies that $(\hat{X},\hat{\mathscr T}(A))$ is Hausdorff and specially both $(X,\mathscr T)$ and $(\hat{X},\hat{\mathscr T})$ are also Hausdorff.

Let
$$
\mathscr B(x):=\{B:\, B\in\mathcal B, \text{ where } \mathcal B\in\mathfrak B \text{ and } \mathcal B\to x\}
$$
for $x\in\hat X$. Clearly $\mathscr B(x)$ is a filter base. We call $\mathscr B(x)$ a \textit{basic $($open$)$ punctured partial neighborhood basis} at~$x$ and set $B\in\mathscr B(x)$ a \textit{basic $($open$)$ punctured partial neighborhood} of~$x$ ($x\in\hat X$). In the paper, we may also use the terminology \textit{deleted} to replace the ``punctured". Let
$$
\mathscr B:=\bigcup_{x\in\hat X}\mathscr B(x),
$$
which is called a \textit{basic $($open$)$ punctured partial neighborhood basis} on~$\hat X$.

For a partial point set~$A$ ($\mathtt I\subseteq A\subseteq\hat X$), let $\mathscr B_A(x)$ be an open punctured neighborhood basis at~$x$ in $(X,\mathscr T)$ for $x\in X\!\setminus\!A$ and $\mathscr B_A(x):=\mathscr B(x)$ for $x\in A$. Let
$$
\mathscr B_A:=\bigcup_{x\in\hat X}\mathscr B_A(x),
$$
which is called an \textit{$A$-mixed punctured neighborhood basis} on~$\hat X$.

Suppose $(X,\mathscr T;\hat{X},\hat{\mathscr T})$ is a universal topological space determined by a perfect filterbase structure system~$\mathfrak B$ and $\hat Y$ is a subset of~$\hat X$. Suppose $\mathtt I_0\subseteq X$, $\mathtt I_1:=\mathtt I\cup\mathtt I_0$ ($\mathtt I$ is the ideal point set of~$\hat X$), $\mathtt I':=\mathtt I_1\cap\hat Y$ and $Y:=\hat Y\setminus\mathtt I'\ne\emptyset$. We call $(Y,\mathscr T';\hat{Y},\hat{\mathscr T}')$ a \textit{universal topological subspace} of $(X,\mathscr T;\hat{X},\hat{\mathscr T})$, where $\mathscr T':=\mathscr T\cap Y$ and the \textit{partial topology} $\hat{\mathscr T}':=\hat{\mathscr T}\cap\hat{Y}$ are induced topologies, the $A'$-\textit{mixed topologies}~$\hat{\mathscr T}'(A')$ of~$\hat Y$ are just the induced topologies $\hat{\mathscr T}(A_1)\cap\hat Y$, where $\mathtt I_1\subseteq A_1\subseteq\hat X$, $A'=A_1\cap\hat Y$ and $\hat{\mathscr T}(A_1)$ is the $A_1$-mixed topology of~$\hat X$, and specially the \textit{essential topology} $\check{\mathscr T}'$ of~$\hat Y$ is the induced topology $\hat{\mathscr T}(\mathtt I_1)\cap\hat{Y}$. The set~$\mathtt I'$ is called the \textit{ideal point set} of~$\hat Y$. If $\mathcal B\to x\in\hat X$ ($\mathcal B\in\mathfrak B$), then we denote $\mathcal B$ by~$\mathcal B(x)$. Let $\mathcal B_Y(x):=\mathcal B(x)\cap Y=\{B\cap Y:\,B\in\mathcal B(x)\}$ for $x\in\hat X$ and $\mathfrak B_Y:=\{\mathcal B_Y(y):\,y\in\hat Y\}$. Usually we assume that $B\cap Y\ne\emptyset$ for any $B\in\mathcal B(y)$, $y\in\hat Y$ (hence $\mathcal B_Y(y)$ is a filter base) and one of the following two assumptions holds: 
\medskip\par\noindent
(1)~$(\hat X,\check{\mathscr T})$ is Hausdorff; \smallskip\par\noindent
(2)~$\mathcal B_Y(y)\thicksim\mathcal B(y)$ for each $y\in\hat Y$. 
\medskip\par\noindent
Then $\mathfrak B_Y$ is a perfect filterbase structure system on $(Y,\mathscr T_Y)$ ($\mathscr T_Y:=\mathscr T\cap Y$ is the induced topology~$\mathscr T'$), called the \textit{induced perfect filterbase structure system} by~$\mathfrak B$ on~$Y$. Thus, we obtain a universal topological space $(Y,\mathscr T_Y;\hat{Y},\hat{\mathscr T}_{\hat Y})$ determined by~$\mathfrak B_Y$, which is just the universal topological subspace $(Y,\mathscr T';\hat{Y},\hat{\mathscr T}')$ (defined above) of $(X,\mathscr T;\hat{X},\hat{\mathscr T})$ (cf.\ the proof of Theorem~\ref{theorem2.1}(4) below). If $\hat Y$ is an open subset of $(\hat X,\hat{\mathscr T}(\mathtt I_1))$, then the subspace $(Y,\mathscr T';\hat{Y},\hat{\mathscr T}')$ is called an \textit{open set} in~$(X,\mathscr T;\hat{X},\hat{\mathscr T})$.

\subsection{Partial\, continuity, \ essential\, continuity\, and\, (exact)\, continuity}

Suppose $(X,\mathscr T;\hat{X},\hat{\mathscr T})$ and $(Y,\mathscr T';\hat{Y},\hat{\mathscr T}')$ are two universal topological spaces. Let $\hat f$ be a mapping from $\hat Y$ to~$\hat X$. The mapping~$\hat f$ is said to be \textit{partially continuous} (resp.\ at a point $y\in\hat Y$) if $\hat f:\,(\hat Y,\hat{\mathscr T}')\to(\hat X,\hat{\mathscr T})$ is continuous (resp.\ at~$y$). If $\hat f:\,(\hat Y,\check{\mathscr T}')\to(\hat X,\check{\mathscr T})$ is continuous (resp.\ at~$y\in\hat Y$), where $\check{\mathscr T}$ and~$\check{\mathscr T}'$ are the essential topologies of $\hat X$ and~$\hat Y$ respectively, then we say that $\hat f:\,\hat Y\to\hat X$ is \textit{essentially continuous} (resp.\ at~$y$). Now suppose $\hat f$ is partially continuous and $\hat f(Y)\subseteq X$. If $\hat f|_Y:\,(Y,\mathscr T')\to(X,\mathscr T)$ is also continuous 
then we say that $\hat f:\,\hat Y\to\hat X$ is (\textit{exactly}) \textit{continuous}. 
Evidently the (exact) continuity of~$\hat f$ implies the essential continuity of~$\hat f$.

Suppose $\hat f:\,\hat Y\to\hat X$ is a bijective. It is called an \textit{essential} (resp.\ \textit{a partial}) \textit{homeomorphism} if $\hat f:\,(\hat Y,\check{\mathscr T}')\to(\hat X,\check{\mathscr T})$ (resp.\ $\hat f:\,(\hat Y,\hat{\mathscr T}')\to(\hat X,\hat{\mathscr T})$) is a homeomorphism. It is called an (\textit{exact}) \textit{homeomorphism} if $\hat f(Y)=X$ and both $\hat f:\,(\hat Y,\hat{\mathscr T}')\to(\hat X,\hat{\mathscr T})$ and $\hat f|_Y:\,(Y,\mathscr T')\to(X,\mathscr T)$ are homeomorphisms. The (exact) homeomorphism obviously implies the essential homeomorphism.

As for local homeomorphisms, we may consider partial localness, essential localness and exact localness respectively. Then we may further locally consider partial, essential and exact homeomorphisms, respectively. In this paper we define \textit{essential local homeomorphisms} and (\textit{exact}) \textit{local homeomorphisms} as follows. Suppose $\hat f:\,\hat Y\to\hat X$ is a mapping. It is called an \textit{essential local homeomorphism} if $\hat f:\,(\hat Y,\check{\mathscr T}')\to(\hat X,\check{\mathscr T})$ is a local homeomorphism ($\check{\mathscr T}$ and $\check{\mathscr T}'$ are essential topologies on~$\hat X$ and~$\hat Y$ respectively). It is called an (\textit{exact}) \textit{local homeomorphism} if for each point $y\in\hat Y$ there exists an open essential neighborhood~$\check N$ of~$y$ (i.e.\ an open neighborhood of~$y$ in~$(\hat Y,\check{\mathscr T})$) such that $\hat f(\check N)$ is an open essential neighborhood of~$\hat f(y)$ and $\hat f|_{\check N}:\, \check N\to\hat f(\check N)$ is an exact (i.e.\ partial and essential) homeomorphism.
Obviously, the (exact) local homeomorphism implies the essential local homeomorphism.

Let $\hat X$, $\hat Y$ and $\hat Z$ be universal topological spaces. It is evident that if $\hat f:\,\hat Y\to\hat X$ and $\hat g:\,\hat Z\to\hat Y$ are (exactly) (resp.\ partially, essentially) continuous then $\hat f{\scriptstyle\,\circ\,}\hat g:\,\hat Z\to\hat X$ is also (exactly) (resp.\ partially, essentially) continuous.

\subsection{Covering maps and deck transformations}

Suppose $(X,\mathscr T;\hat{X},\hat{\mathscr T})$, $(Y,\mathscr T';\hat{Y},\hat{\mathscr T}')$ and $(Z,\mathscr T'';\hat{Z},\hat{\mathscr T}'')$ are universal topological spaces. Suppose $\hat p:\hat Y\to\hat X$ and $\hat q:\hat Z\to\hat X$ are essentially continuous mappings. For $x\in\hat X$, the set~$\hat p^{-1}(x)$ is called the \textit{fiber} of~$\hat p$ over~$x$. A mapping $\hat f:\hat Z\to\hat Y$ is called \textit{fiber-preserving} if $\hat p{\scriptstyle\,\,\circ\,}\hat f=\hat q$.

Suppose for every $x\in X\setminus\hat p(\hat Y\!\setminus\!Y)$ there exists $U\in\mathscr T(x)$, where $\mathscr T(x)$ denotes the set of all open neighborhoods (i.e.\ the open neighborhood system) of $x\in X$ in $(X,\mathscr T)$, and for every point $\hat x\in(\hat X\!\setminus\!X)\cup\hat p(\hat Y\!\setminus\!Y)$ there exists $\hat U\in\hat{\mathscr T}(\hat x)$ (the partial neighborhood system at~$\hat x$, i.e.\ the open neighborhood system at~$\hat x$ in $(\hat X,\hat{\mathscr T})$) such that
$$
\hat p^{-1}(U)=\bigcup_{j\in J}V_j\quad \text{and}\quad \hat p^{-1}(\hat U)=\bigcup_{k\in K}\hat V_k,
$$
where $V_j\in\mathscr T'$ ($j\in J$) are disjoint and $\hat V_k\in\hat{\mathscr T}'$ ($k\in K$) are disjoint. If all the mappings $\hat p|_{V_j}\!:V_j\to U$ ($j\in J$) and $\hat p|_{\hat V_k}\!:\hat V_k\to\hat U$ ($k\in K$) are essential homeomorphisms, then $\hat p:\hat Y\to\hat X$ is called an \textit{essential covering map}. If all the mappings $\hat p|_{V_j}$ ($j\in J$) and $\hat p|_{\hat V_k}$ ($k\in K$) are exact homeomorphisms, then $\hat p:\hat Y\to\hat X$ is called a (or an \textit{exact}) \textit{covering map}.

Let $\hat p:\hat Y\to\hat X$ be an exact (resp.\ essential) covering map. A fiber-preserving (exact) (resp.\ essential) homeomorphism $\hat f:\,\hat Y\to\hat Y$ is called a (or an \textit{exact}) (resp.\ an \textit{essential}) \textit{deck transformation} of $\hat p:\,\hat Y\to\hat X$. We denote the set of all (exact) deck transformations of $\hat p:\,\hat Y\to\hat X$ by $\deck(\hat Y\!\overset{\scriptscriptstyle\hat p}{\to}\!\hat X)$ or $\deck(\hat Y\!/\hat X)$. Then $\deck(\hat Y\!/\hat X)$ forms a group with operation the composition of mappings.

\subsection{A universal topological space derived by a presheaf}

Suppose $(X,\mathscr T;\hat{X},\hat{\mathscr T})$ is a universal topological space with a basic open (resp.\ punctured) partial neighborhood basis~$\hat{\mathscr B}$ (resp.~$\mathscr B$). Suppose $(\mathcal F,\rho)$, where $\mathcal F=(\mathcal F(U))_{U\in\mathscr T}$, is a presheaf of some algebraic system on~$(X,\mathscr T)$ (refer to \cite[\S6]{Fo}). Denote $f|_V:=\rho^U_V(f)$ for $U$, $V\in\mathscr T$, $V\subseteq U$ and $f\in\mathcal F(U)$.

For $\hat U\in\hat{\mathscr T}$, $\hat U\ne\emptyset$, let $U$ be the interior of $\hat U\!\setminus\!\mathtt I$ in $(X,\mathscr T)$ ($\mathtt I$ is the ideal point set). Such an open set~$U$ in $(X,\mathscr T)$ is not empty, called the \textit{body} of~$\hat U$ and denoted~$\hat U^\circ$. If $\hat U\in\hat{\mathscr T}(x)$ ($x\in\hat X$), then the interior of $\hat U\!\setminus\!\mathtt I\!\setminus\!\{x\}$ in $(X,\mathscr T)$ is also a nonempty open set, which we call the ($\mathscr T$-)\textit{open punctured partial neighborhood} of~$x$ corresponding to~$\hat U$. 
Denote the set of all $\mathscr T$-open punctured partial neighborhoods of~$x$ by~$\hat{\mathscr T}^\circ(x)$ and let
$$
\hat{\mathscr T}^\circ:=\bigcup_{x\in\hat X}\hat{\mathscr T}^\circ(x).
$$

Let
$$
\mathcal F(\hat{\mathscr T}^\circ(x)):=\bigcup_{V\in\hat{\mathscr T}^\circ(x)}\mathcal F(V)
$$
for $x\in\hat X$, which is a disjoint union. In $\mathcal F(\hat{\mathscr T}^\circ(x))$ two elements $g_1\in\mathcal F(V_1)$ and $g_2\in\mathcal F(V_2)$ ($V_1$, $V_2\in\hat{\mathscr T}^\circ(x)$) are said to be \textit{equivalent} at~$x$, denoted $g_1\,\underset{x}{\hat{\thicksim}}\,g_2$, if there exists $V\in\hat{\mathscr T}^\circ(x)$ with $V\subseteq V_1\cap V_2$ such that $g_1|_V=g_2|_V$. It is easy to see that this really is an equivalence relation. The set
$$
\hat{\mathcal F}_x:=\mathcal F(\hat{\mathscr T}^\circ(x))\Big/\underset{x}{\hat{\thicksim}}
$$
of all equivalence classes is called the \textit{punctured partial stalk} of~$\mathcal F$ at the point $x\in\hat X$ and the equivalence class of $g\in\mathcal F(\hat{\mathscr T}^\circ(x))$ is called the \textit{punctured partial germ} of~$g$ at~$x$, denoted~$\langle g\rangle_x$.

For $x\in X$ we denote the (usual) germ of $f\in\mathcal F(U)$ at~$x$ by~$[f]_x$. For $g\in\mathcal F(\hat{\mathscr T}^\circ(x))$, if there exist $U\in\mathscr T(x)$ (the open neighborhood system at the point~$x$ in $(X,\mathscr T)$), $f\in\mathcal F(U)$ and $V\in\hat{\mathscr T}^\circ(x)$ such that $f|_V=g|_V$, then $x$ is called a \textit{usual point} of~$\langle g\rangle_x$ or~$g$, $f$ a \textit{usual element} at~$x$ corresponding to~$g$, $\langle g\rangle_x$ a \textit{usual punctured partial germ} and $g$ \textit{usual} at~$x$. Denote $\langle f\rangle_x:=\langle f|_V\rangle_x$. Then $\langle f\rangle_x=\langle g\rangle_x$. In this case we say that $\langle g\rangle_x$ and~$[f]_x$ are \textit{equivalent}, denoted $\langle g\rangle_x\thicksim[f]_x$.

By the condition of a perfect filterbase structure system on~$X$ we know that any $U\in\mathscr T$ is not a singleton. Now assume $(X,\mathscr T)$ is a $T_1$ space. For $x\in X$ let $\overset{\scriptscriptstyle\circ}{\mathscr T}(x):=\{U\!\setminus\!\{x\}:\,U\in\mathscr T,\ x\in U\}$, which is the punctured open neighborhood system at the point~$x$ in $(X,\mathscr T)$, and
$$
\overset{\scriptscriptstyle\circ}{\mathscr T}:=\bigcup_{x\in\hat X}\overset{\scriptscriptstyle\circ}{\mathscr T}(x).
$$
Then $\overset{\scriptscriptstyle\circ}{\mathscr T}\subseteq\mathscr T$. Let
$$
\mathcal F(\overset{\scriptscriptstyle\circ}{\mathscr T}(x)):=\bigcup_{U\in\overset{\scriptscriptstyle\circ}{\mathscr T}(x)}\mathcal F(U).
$$
In $\mathcal F(\overset{\scriptscriptstyle\circ}{\mathscr T}(x))$, two elements $f_1\in\mathcal F(U_1)$ and $f_2\in\mathcal F(U_2)$ ($U_1$, $U_2\in\overset{\scriptscriptstyle\circ}{\mathscr T}(x)$) are said to be \textit{equivalent} at~$x$, denoted $f_1\,\underset{x}{\overset{\scriptscriptstyle\circ}{\thicksim}}\,f_2$, if there exists $U\in\overset{\scriptscriptstyle\circ}{\mathscr T}(x)$ with $U\subseteq U_1\cap U_2$ such that $f_1|_U=f_2|_U$. Easily we see that this is an equivalence relation. The set
$$
\overset{\scriptscriptstyle\,\,\circ}{\mathcal F}_x:=\mathcal F(\overset{\scriptscriptstyle\circ}{\mathscr T}(x))\Big/\overset{\scriptscriptstyle\circ}{\underset{\scriptscriptstyle x}{\thicksim}}
$$
of all equivalence classes is called the \textit{punctured stalk} of~$\mathcal F$ at $x\in\hat X$ and the equivalence class of $f\in\mathcal F(\overset{\scriptscriptstyle\circ}{\mathscr T}(x))$ is called the \textit{punctured germ} of~$f$ at~$x$, denoted~$[f]^\circ_x$.

For $f\in\mathcal F(\overset{\scriptscriptstyle\circ}{\mathscr T}(x))$ ($x\in X$), letting $f\in\mathcal F(U_0)$ ($U_0\in\overset{\scriptscriptstyle\circ}{\mathscr T}(x)$), if there exist $U\in\mathscr T(x)$, $h\in\mathcal F(U)$ and $V\in\overset{\scriptscriptstyle\circ}{\mathscr T}(x)$ such that $V\subseteq U_0\cap U$ and $h|_V=f|_V$, then $x$ is called a \textit{full point} of~$f$ or~$[f]^\circ_x$, $h$ a \textit{full element} at~$x$ corresponding to~$f$, $[f]^\circ_x$ \textit{full} and $f$  \textit{full} at~$x$. Denote $[h]^\circ_x:=[h|_V]^\circ_x$. Then $[h]^\circ_x=[f]^\circ_x$. In this case we say that $[f]^\circ_x$ and~$[h]_x$ are \textit{equivalent}, denoted $[f]^\circ_x\thicksim[h]_x$.

For $g\in\mathcal F(\hat{\mathscr T}^\circ(x))$ ($x\in X$), letting $g\in\mathcal F(V_0)$ ($V_0\in\hat{\mathscr T}^{\circ}(x)$), if there exist $U\in\overset{\scriptscriptstyle\circ}{\mathscr T}(x)$, $f\in\mathcal F(U)$ and $V\in\hat{\mathscr T}^\circ(x)$ such that $V\subseteq V_0\cap U$ and $f|_V=g|_V$, then $x$ is called an \textit{unbranched point} or a \textit{complete point} of~$\langle g\rangle_x$ or~$g$, $f$ a \textit{complete element} at~$x$ corresponding to~$g$, $\langle g\rangle_x$ \textit{unbranched} or \textit{complete} and $g$ \textit{unbranched} or \textit{complete} at~$x$. Denote $\langle f\rangle_x:=\langle f|_V\rangle_x$. Then $\langle f\rangle_x=\langle g\rangle_x$. In this case we say that $\langle g\rangle_x$ and~$[f]^\circ_x$ are \textit{equivalent}, denoted $\langle g\rangle_x\thicksim[f]^\circ_x$. Obviously, the usualness of~$\langle g\rangle_x$ implies its completeness.

\smallskip
\begin{remark}\label{remark1}
If $\hat{\mathscr B}(x)$ is a basis for~$\mathscr T(x)$ at $x\in X$, then the punctured partial germ at~$x$ and the punctured germ at~$x$ are just the same.
\end{remark}
\smallskip

If $\mathcal F$ is a presheaf of fields (resp.\ rings, vector spaces, etc), then the punctured partial stalk~$\hat{\mathcal F}_x$ ($x\in\hat X$) and the punctured stalk~$\overset{\,\circ}{\mathcal F}_x$ ($x\in\hat X$) with the operation defined on punctured partial germs and punctured germs respectively, by means of the operation defined on representatives, are both fields (resp.\ rings, vector spaces, etc).

Define
$$
\hat{\mathfrak F}:=\bigcup_{x\in\hat X}\hat{\mathcal F}_x
$$
and
$$
\overset{\scriptscriptstyle\circ}{\mathfrak F}:=\bigcup_{x\in X}\overset{\scriptscriptstyle\,\,\circ}{\mathcal F}_x,
$$
which are the disjoint unions of all the punctured partial stalks over~$\hat X$ and all the punctured stalks over~$X$, respectively. Let
$$
\hat p:\ \hat{\mathfrak F}\longrightarrow\hat X \quad\text{and}\quad \overset{\scriptscriptstyle\circ}p:\ \overset{\scriptscriptstyle\circ}{\mathfrak F}\longrightarrow X
$$
be the projections, i.e.\ $\hat p(\langle g\rangle_x)=x$ for $\langle g\rangle_x\in\hat{\mathfrak F}$ ($x\in\hat X$) and $\overset{\scriptscriptstyle\circ}p([f]^\circ_x)=x$ for $[f]^\circ_x\in\overset{\scriptscriptstyle\circ}{\mathfrak F}$ ($x\in X$).

Let $\mathcal F^\circ_x$ be the set of all the complete punctured partial germs in~$\hat{\mathcal F}_x$ ($x\in X$), which is called the \textit{complete} (or \textit{unbranched}) \textit{punctured partial stalk} of~$\mathcal F$ at~$x$. Let
$$
\mathfrak F^\circ:=\bigcup_{x\in X}\mathcal F^\circ_x,
$$
which is the disjoint union of all the complete punctured partial stalks over~$X$, and $p:=\hat p|_{\mathfrak F^\circ}$. Then
$$
p:\ \mathfrak F^\circ\longrightarrow X
$$
is also a projection.

For nonempty $\hat U\in\hat{\mathscr T}$ with body~$\hat U^\circ\in\mathscr T$ and $f\in\mathcal F(\hat U^\circ)$ we denote
$$
\langle\hat U,f\rangle:=\{\langle f\rangle_x:\,x\in\hat U\}.
$$
Define
$$
\hat{\mathscr N}(\mathcal F):=\{\langle\hat U,f\rangle:\,\emptyset\ne\hat U\in\hat{\mathscr T},\ f\in\mathcal F(\hat U^\circ)\}
$$
and
$$
\hat{\mathscr N}(\mathcal F(\hat{\mathscr B})):=\{\langle\hat B,f\rangle:\,\hat B\in\hat{\mathscr B},\ f\in\mathcal F(\hat B^\circ)\},
$$
where $\hat B^\circ\in\mathscr B$ is the nonempty body of~$\hat B$. For $f\in\mathcal F(B)$, where $B\in\mathscr B(x)$ ($x\in\hat X$), define
$$
\hat{\mathscr N}_f(\mathcal F(\hat{\mathscr B}))(\langle f\rangle_x):=\{\langle\hat U,f|_{\hat U\setminus\{x\}}\rangle:\,\hat U\in\hat{\mathscr B}(x),\ \hat U\subseteq\hat B\},
$$
where $\hat B=B\cup\{x\}\in\hat{\mathscr B}(x)$. Then it is easy to verify that $\hat{\mathscr N}_f(\mathcal F(\hat{\mathscr B}))(\langle f\rangle_x)$ is a filter base and
$$
\hat{\mathscr N}(\mathcal F(\hat{\mathscr B}))=\bigcup\{\hat{\mathscr N}_f(\mathcal F(\hat{\mathscr B}))(\langle f\rangle_x):\,f\in\mathcal F(B),\ B\in\mathscr B(x),\ x\in\hat X\}.
$$

For nonempty $U\in\mathscr T$, if there exists a point $a\in U$ such that $f\in\mathcal F(U\!\setminus\!\{a\})$, then we denote $f\in\overset{\scriptscriptstyle\,\,\circ}{\mathcal F}(U)$, or generally we may define
\begin{equation}\label{2.0}
\overset{\scriptscriptstyle\,\,\circ}{\mathcal F}(U):=\{f:\,f\in\mathcal F(U\!\setminus\!C), \text{ where } C \text{ is a discrete point set in~}U\}.
\end{equation}
For $f\in\overset{\scriptscriptstyle\,\,\circ}{\mathcal F}(U)$
we denote
$$
\langle U,f\rangle:=\{\langle f\rangle_x:\,x\in U\}.
$$
Define
$$
\mathscr N^\circ(\mathcal F):=\{\langle U,f\rangle:\, \emptyset\ne U\in\mathscr T,\ f\in\overset{\scriptscriptstyle\,\,\circ}{\mathcal F}(U)\}
$$
and
$$
\mathscr N(\mathcal F(\mathscr B)):=\{\langle B,f\rangle:\, B\in\mathscr B,\ f\in\mathcal F(B)\}.
$$
For $f\in\mathcal F(B)$, where $B\in\mathscr B(x)$ ($x\in\hat X$), define
$$
\mathscr N_f(\mathcal F(\mathscr B))(\langle f\rangle_x):=\{\langle U,f|_ U\rangle:\, U\in\mathscr B(x),\ U\subseteq B\}
$$
and
\begin{equation}\label{2.1}
\mathfrak N(\mathcal F(\mathscr B)):=\{\mathscr N_f(\mathcal F(\mathscr B))(\langle f\rangle_x):\,f\in\mathcal F(B),\ B\in\mathscr B(x),\ x\in\hat X\}.
\end{equation}
Then $\mathscr N_f(\mathcal F(\mathscr B))(\langle f\rangle_x)$ is a filter base and
$$
\mathscr N(\mathcal F(\mathscr B))=\bigcup\{\mathcal N:\,\mathcal N\in\mathfrak N(\mathcal F(\mathscr B))\}.
$$

If for a common complete point $x\in X$ of $g_1\in\mathcal F(V_1)$ and $g_2\in\mathcal F(V_2)$, where $V_1$, $V_2\in\hat{\mathscr T}^\circ(x)$, letting $f_j|_{V_j\cap U_j}=g_j|_{V_j\cap U_j}$, where $f_j\in\mathcal F(U_j)$, $U_j\in\overset{\scriptscriptstyle\circ}{\mathscr T}(x)$ ($j=1$, $2$), the equality $\langle g_1\rangle_x=\langle g_2\rangle_x$ always implies $[f_1]^\circ_x=[f_2]^\circ_x$, then the presheaf~$\mathcal F$ is called \textit{consistent} at~$x$ (on $(X,\mathscr T;\hat X,\hat{\mathscr T})$). The presheaf is called \textit{consistent} (on $(X,\mathscr T;\hat X,\hat{\mathscr T})$) if it is consistent at all the complete points.

For a partial point set~$A$ ($\mathtt I\subseteq A\subseteq\hat X$) let $\hat{\mathcal F}_x(A):=\hat{\mathcal F}_x$ (the punctured partial stalk at~$x$) for $x\in A$ and $\hat{\mathcal F}_x(A):=\mathcal F^\circ_x$ (the complete punctured partial stalk at~$x$) for $x\in X\!\setminus\! A$. Denote
$$
\hat{\mathfrak F}(A):=\bigcup_{x\in\hat X}\hat{\mathcal F}_x(A)
$$
and define
$$
\hat{\mathscr N}(\mathcal F(\hat{\mathscr B}_A)):=\{\langle V,f\rangle:\, V\in\hat{\mathscr B}_A(x),\ f\in\mathcal F(V\!\setminus\!\{x\}),\ x\in\hat X\},
$$
where $\hat{\mathscr B}_A(x)$ is an $A$-mixed neighborhood basis at $x\in\hat X$.

\begin{theorem}\label{theorem2.1}
Suppose $(X,\mathscr T;\hat{X},\hat{\mathscr T})$ is a universal topological space with a basic open $($resp.\ punctured\,$)$ partial neighborhood basis~$\hat{\mathscr B}$ $($resp.~$\mathscr B)$. Suppose $\mathcal F$ is a presheaf on $(X,\mathscr T)$. In~$(2)$, $(3)$ and\/~$(4)$ below we also suppose $(X,\mathscr T)$ is a $T_1$~space and $\mathcal F$ is consistent. Then
\par\medskip
\noindent
$(1)$ $\hat{\mathscr N}(\mathcal F(\hat{\mathscr B}))$ is a basis for a topology on~$\hat{\mathfrak F}$, and so is $\hat{\mathscr N}(\mathcal F)$ for the same topology, denoted~$\hat{\mathscr T}(\mathcal F)$, and the projection
$$
\hat p:\ (\hat{\mathfrak F},\hat{\mathscr T}(\mathcal F))\longrightarrow(\hat X,\hat{\mathscr T})
$$
is a local homeomorphism.
\par\smallskip\noindent
$(2)$ $\mathscr N^\circ(\mathcal F)$ is a basis for a topology on~$\mathfrak F^\circ$, denoted~$\mathscr T^\circ(\mathcal F)$, and the projection
$$
p:\ ({\mathfrak F}^\circ,{\mathscr T}^\circ(\mathcal F))\longrightarrow(X,{\mathscr T})
$$
is a local homeomorphism.
\par\smallskip\noindent
$(3)$ For a partial point set~$A$ $(\mathtt I\subseteq A\subseteq\hat X)$, $\hat{\mathscr N}(\mathcal F(\hat{\mathscr B}_A))$ is a basis for a topology on~$\hat{\mathfrak F}(A)$, denoted~$\hat{\mathscr T}_A(\mathcal F)$, and the projection
$$
\hat p_{_A}:\ (\hat{\mathfrak F}(A),\hat{\mathscr T}_A(\mathcal F))\longrightarrow(\hat X,\hat{\mathscr T}(A))
$$
is a local homeomorphism.
\par\smallskip\noindent
$(4)$ $\mathfrak N(\mathcal F(\mathscr B))$ is a perfect filterbase structure system on $({\mathfrak F}^\circ,{\mathscr T}^\circ(\mathcal F))$ and under some obvious assumption $({\mathfrak F}^\circ,{\mathscr T}^\circ(\mathcal F); \hat{\mathfrak F},\hat{\mathscr T}(\mathcal F))$ is a universal topological space determined by~$\mathfrak N(\mathcal F(\mathscr B))$ with the basic $($resp.\ punctured\/$)$ partial neighborhood basis $\hat{\mathscr N}(\mathcal F(\hat{\mathscr B}))$ $($resp.\ $\mathscr N(\mathcal F(\mathscr B)))$. Moreover, the projection
$$
\hat p:\ \hat{\mathfrak F}\longrightarrow\hat X
$$
is an exact local homeomorphism and hence an essential local homeomorphism.
\end{theorem}

\begin{proof}
(1) Obviously we have
$$
\hat{\mathfrak F}=\bigcup\{\langle\hat B,f\rangle:\,\langle\hat B,f\rangle\in\hat{\mathscr N}(\mathcal F(\hat{\mathscr B}))\}.
$$

For $\alpha_x\in\langle\hat B_1,f_1\rangle\cap\langle\hat B_2,f_2\rangle$, where $\alpha_x$ denotes a punctured partial germ at $x\in\hat X$ and $\langle\hat B_j,f_j\rangle\in\hat{\mathscr N}(\mathcal F(\hat{\mathscr B}))$ ($j=1$, $2$), there exists $\langle\hat B,f\rangle\in\hat{\mathscr N}(\mathcal F(\hat{\mathscr B}))$ such that $\alpha_x\in\langle\hat B,f\rangle\subseteq\langle\hat B_1,f_1\rangle\cap\langle\hat B_2,f_2\rangle$. This follows from that $f_1|_B=f_2|_B=:f$ for some basic punctured partial neighborhood $B\in\mathscr B(x)$ satisfying $\hat B=B\cup\{x\}\subseteq\hat B_1\cap\hat B_2$. Therefore, $\hat{\mathscr N}(\mathcal F(\hat{\mathscr B}))$ is a basis for a topology on~$\hat{\mathfrak F}$. Easily we can also show that $\hat{\mathscr N}(\mathcal F)$ is a basis for the same topology.

For $\alpha_x\in\hat{\mathfrak F}$ ($x\in\hat X$), there is $B\in\mathscr B(x)$ and $f\in\mathcal F(B)$ such that $\alpha_x=\langle f\rangle_x$. The mapping
$$
\hat p|_{\langle\hat B,f\rangle}:\ \langle\hat B,f\rangle\longrightarrow\hat B,
$$
where $\hat B=B\cup\{x\}$, is evidently a homeomorphism.

\medskip
(2) It is easy to see that
$$
\mathfrak F^\circ=\bigcup\{\langle U,f\rangle:\,\langle U,f\rangle\in\mathscr N^\circ(\mathcal F)\}.
$$

For $\alpha_x\in\langle U_1,f_1\rangle\cap\langle U_2,f_2\rangle$, where $\alpha_x$ denotes a punctured partial germ at $x\in X$ and $\langle U_j,f_j\rangle\in\mathscr N^\circ(\mathcal F)$ ($j=1$, $2$), we have $\langle f_1\rangle_x=\alpha_x=\langle f_2\rangle_x$. By the consistency of~$\mathcal F$ it follows $[f_1]^\circ_x=[f_2]^\circ_x$. Thus there is $U\in\mathscr T(x)$ satisfying $U\subseteq U_1\cap U_2$ and $f_1|_{U\setminus\{x\}}=f_2|_{U\setminus\{x\}}$. Let $f=f_1|_{U\setminus\{x\}}$. Then we have $\alpha_x=\langle f\rangle_x\in\langle U,f\rangle\subseteq\langle U_1,f_1\rangle\cap\langle U_2,f_2\rangle$. Therefore $\mathscr N^\circ(\mathcal F)$ is a basis for a topology on~$\mathfrak F^\circ$. For $\alpha\in{\mathfrak F}^\circ$ there is $f\in\overset{\scriptscriptstyle\,\,\circ}{\mathcal F}(U)$, where $U\in\mathscr T(x)$ ($x\in X$), such that $\alpha=\langle f\rangle_x$. The mapping
$$
p|_{\langle U,f\rangle}:\ \langle U,f\rangle\longrightarrow U
$$
is a homeomorphism.

\medskip
(3) It is evident that
$$
\hat{\mathfrak F}(A)=\bigcup\{\langle V,f\rangle:\,\langle V,f\rangle\in\hat{\mathscr N}(\mathcal F(\hat{\mathscr B}_A))\}.
$$

To prove $\hat{\mathscr N}(\mathcal F(\hat{\mathscr B}_A))$ is a topological basis we need to show this in three cases. We now consider the case that $\alpha_x\in\langle\hat B,f_1\rangle\cap\langle U,f_2\rangle$, where $\alpha_x$ denotes a punctured partial germ at $x\in X$, $x\in U$, $U\in\mathscr T$ and $\hat B\in\hat{\mathscr B}(x_1)$ ($x_1\in A$). If $x\in X\!\setminus\!A$ then $x\ne x_1$. Hence $x\in B\cap U$, where $B=\hat B\!\setminus\!\{x_1\}\in\mathscr B(x_1)$, and $\langle f_1\rangle_x=\alpha_x=\langle f_2\rangle_x$, which implies $[f_1]^\circ_x=[f_2]^\circ_x$ by the consistency of~$\mathcal F$. Consequently, $f_1|_{U_0\setminus\{x\}}=f_2|_{U_0\setminus\{x\}}:=f$ for some $U_0\in\mathscr T(x)$, $U_0\subseteq B\cap U$. We then have $\alpha_x\in\langle U_0,f\rangle\subseteq\langle\hat B,f_1\rangle\cap\langle U,f_2\rangle$. If $x\in A$ then $x\in\hat B\cap U$ and $\langle f_1\rangle_x=\langle f_2\rangle_x$, which implies there exists $B_0\in\mathscr B(x)$, $B_0\subseteq B\cap U$, such that $f_1|_{B_0}=f_2|_{B_0}=:f$. Therefore $\alpha_x\in\langle\hat B_0,f\rangle\subseteq\langle\hat B,f_1\rangle\cap\langle U,f_2\rangle$, where $\hat B_0=B_0\cup\{x\}\in\hat{\mathscr B}(x)$. Similarly we can show the other cases.

For $\alpha\in\hat{\mathfrak F}(A)$ there exists $\hat{\mathscr B}_A\in\hat{\mathscr B}_A(x)$ ($x\in\hat X$) and $f\in\mathcal F(B_A)$, where $B_A=\hat B_A\!\setminus\!\{x\}\in\mathscr B_A(x)$, such that $\alpha=\langle f\rangle_x$. The mapping
$$
p_{_A}|_{\langle\hat B_A,f\rangle}:\ \langle\hat B_A,f\rangle\longrightarrow\hat B_A
$$
is a homeomorphism.

\medskip
(4) For $\alpha\in\hat{\mathfrak F}$, there is $x\in\hat X$ and $f\in\mathcal F(B)$, where $B\in\mathscr B(x)$, such that $\alpha=\langle f\rangle_x$. Thus
$$
\mathscr N_f(\mathcal F(\mathscr B))(\langle f\rangle_x)\longrightarrow\alpha
$$
in topology~$\hat{\mathscr T}(\mathcal F)$ and if $\alpha\in\mathfrak F^\circ$ then the above limit also holds in topology~$\mathscr T^\circ(\mathcal F)$.

Now let $\alpha=\langle f\rangle_x\in\mathfrak F^\circ$ and assume
$$
\mathscr N_g(\mathcal F(\mathscr B))(\langle g\rangle_x)\longrightarrow\alpha
$$
in topology~$\mathscr T^\circ(\mathcal F)$, where $f\in\mathcal F(U)$, $U\in\overset{\scriptscriptstyle\circ}{\mathscr T}(x)$, $g\in\mathcal F(B)$, $B\in\mathscr B(x)$ and $x\in X$. Then there exists $V\in\mathscr B(x)$, $V\subseteq B$, such that
$$
\langle V,g|_V\rangle\subseteq\langle U,f\rangle.
$$
Hence $V\subseteq U$ and $\langle g\rangle_y=\langle f\rangle_y$ for each $y\in V$. So $\langle V,g|_V\rangle=\langle V,f|_V\rangle$. This implies that if both $\mathscr N_g(\mathcal F(\mathscr B))(\langle g\rangle_x)$ and $\mathscr N_h(\mathcal F(\mathscr B))(\langle h\rangle_x)$ converge to~$\alpha$ in topology~$\mathscr T^\circ(\mathcal F)$ then they are equivalent to one another.
Similarly we see that if both $\mathscr N_g(\mathcal F(\mathscr B))(\langle g\rangle_x)$ and $\mathscr N_h(\mathcal F(\mathscr B))(\langle h\rangle_x)$ ($x\in\hat X$) converge to $\alpha\in\hat{\mathfrak F}$ in topology~$\hat{\mathscr T}(\mathcal F)$ then they are also equivalent to one another. Assume the ideal points are just the incomplete punctured partial germs. Then $\mathfrak N(\mathcal F(\mathscr B))$ is a perfect filterbase structure system on $({\mathfrak F}^\circ,{\mathscr T}^\circ(\mathcal F))$ and $({\mathfrak F}^\circ,{\mathscr T}^\circ(\mathcal F); \hat{\mathfrak F},\hat{\mathscr T}(\mathcal F))$ is a universal topological space determined by~$\mathfrak N(\mathcal F(\mathscr B))$ with the basic (resp.\ punctured) partial neighborhood basis $\hat{\mathscr N}(\mathcal F(\hat{\mathscr B}))$ (resp.\ $\mathscr N(\mathcal F(\mathscr B))$).


Let $\alpha_a\in\mathfrak F^\circ$ ($a\in X$) and $\alpha_b\in\hat{\mathfrak F}$ ($b\in\hat X$). Then there exist $f\in\overset{\scriptscriptstyle\,\,\circ}{\mathcal F}(U)$ and $g\in\mathcal F(B)$, where $U\in{\mathscr T}(a)$ and $B\in\mathscr B(b)$, such that $\alpha_a=\langle f\rangle_a$ and $\alpha_b=\langle g\rangle_b$. It is easy to see that both
$$
\hat p|_{\langle U,f\rangle}:\ \langle U,f\rangle\longrightarrow U
$$
and
$$
\hat p|_{\langle\hat B,g\rangle}:\ \langle\hat B,g\rangle\longrightarrow\hat B
$$
are (exact) homeomorphisms, where $\hat B=B\cup\{b\}$.
\end{proof}

We call the space $({\mathfrak F}^\circ,{\mathscr T}^\circ(\mathcal F); \hat{\mathfrak F},\hat{\mathscr T}(\mathcal F))$ in the above theorem the \textit{derived universal topological space} over $(X,\mathscr T;\hat X,\hat{\mathscr T})$ by~$\mathcal F$.
\smallskip

A $T_1$~space~$(X,\mathscr T)$ is called \textit{strongly locally connected} if for any $x\in X$ and $U\in\mathscr T(x)$ there exists $V\in\mathscr T(x)$ such that $V\subseteq U$ and $V\!\setminus\!\{x\}$ is a domain (nonempty connected open set). A universal topological space $(X,\mathscr T;\hat X,\hat{\mathscr T})$ is called \textit{locally connected} if $(X,\mathscr T)$ is strongly locally connected and there exists a punctured partial neighborhood basis~$\mathscr B_1$ on~$\hat X$ satisfying every set~$B$ in~$\mathscr B_1$ is connected in~$(X,\mathscr T)$ (we say $\mathscr B_1$ is \textit{connected}). Here $\mathscr B_1$ may be different from the basic punctured partial neighborhood basis~$\mathscr B$ of~$\hat X$.

The \textit{uniqueness condition} on a presheaf~$\mathcal F$ on a universal topological space $(X,\mathscr T;\hat X,\hat{\mathscr T})$ means the following one: For every domain~$Y$ in~$(X,\mathscr T)$, given any $f$, $g\in\mathcal F(Y)$ and any $a\in\hat X$ satisfying there exists $B\in\mathscr B(a)$ with $B\subseteq Y$, the equality $\langle f\rangle_a=\langle g\rangle_a$ always implies $f=g$.

\begin{theorem}\label{theorem2.2}
Suppose $(X,\mathscr T;\hat X,\hat{\mathscr T})$ is locally connected Hausdorff universal topological space and $\mathcal F$ is a presheaf on~$(X,\mathscr T)$. If $\mathcal F$ satisfies the uniqueness condition on~$\hat X$, then both $({\mathfrak F}^\circ,{\mathscr T}^\circ(\mathcal F))$ and $(\hat{\mathfrak F},\hat{\mathscr T}(\mathcal F))$ are Hausdorff spaces, furthermore, $({\mathfrak F}^\circ,{\mathscr T}^\circ(\mathcal F); \hat{\mathfrak F},\hat{\mathscr T}(\mathcal F))$ is a Hausdorff universal topological space.
\end{theorem}

\noindent
\textit{Proof.} \,Since the local connectedness of the Hausdorff space $(X,\mathscr T;\hat X,\hat{\mathscr T})$ and the uniqueness condition on~$\mathcal F$ imply that
$\mathcal F$ is consistent on~$\hat X$, by Theorem~\ref{theorem2.1} we know that $({\mathfrak F}^\circ,{\mathscr T}^\circ(\mathcal F); \hat{\mathfrak F},\hat{\mathscr T}(\mathcal F))$ is a universal topological space. In the following we prove the spaces are Hausdorff.

At first, suppose $\langle f\rangle_x\ne\langle g\rangle_y$, where $x$, $y\in\hat X$, $f\in\mathcal F(B_1)$, $g\in\mathcal F(B_2)$, $B_1\in\mathscr B(x)$ and $B_2\in\mathscr B(y)$ ($\mathscr B$ is a punctured partial neighborhood basis). If $x\ne y$, then there exist $\hat U_1\in\hat{\mathscr B}(x)$ and $\hat U_2\in\hat{\mathscr B}(y)$ such that $\hat U_1\subseteq B_1\cup\{x\}$, $\hat U_2\subseteq B_2\cup\{y\}$ and $\hat U_1\cap\hat U_2=\emptyset$ ($\hat{\mathscr B}$ is the partial neighborhood basis corresponding to~$\mathscr B$). Clearly we have
$$
\langle\hat U_1,f\rangle\cap\langle\hat U_2,g\rangle=\emptyset.
$$
If $x=y$, then by the local connectedness of $(X,\mathscr T;\hat X,\hat{\mathscr T})$ we may assume $\mathscr B$ is connected and there exists $B\in\mathscr B(x)$ such that $B\subseteq B_1\cap B_2$. By the uniqueness condition we have
$$
\langle\hat B,f|_B\rangle\cap\langle\hat B,g|_B\rangle=\emptyset
$$
($\hat B=B\cup\{x\}$), since otherwise it follows that there exists $a\in\hat B$ such that $\langle f\rangle_a=\langle g\rangle_a$, which implies $f|_B=g|_B$, so $\langle f\rangle_x=\langle g\rangle_y$ ($x=y$), a contradiction. By the reasoning above we see that $(\hat{\mathfrak F},\hat{\mathscr T}(\mathcal F))$ is Hausdorff.

Next, suppose $\langle f\rangle_x\ne\langle g\rangle_y$, where $x$, $y\in X$, $f\in\mathcal F(U_1\!\setminus\!\{x\})$ ($U_1\in\mathscr T(x)$) and $g\in\mathcal F(U_2\!\setminus\!\{y\})$ ($U_2\in\mathscr T(y)$). If $x\ne y$, then there exist $V_1\in\mathscr T(x)$ and $V_2\in\mathscr T(y)$ such that $V_1\subseteq U_1$, $V_2\subseteq U_2$ and $V_1\cap V_2=\emptyset$. Therefore
$$
\langle V_1,f|_{V_1\setminus\{x\}}\rangle\cap\langle V_2,g|_{V_2\setminus\{y\}}\rangle=\emptyset.
$$
If $x=y$, then by the local connectedness of $(X,\mathscr T;\hat X,\hat{\mathscr T})$, there exists $V\in\mathscr T(x)$ such that $V\subseteq V_1\cap V_2$ and $V\!\setminus\!\{x\}\ne\emptyset$ is a domain in $(X,\mathscr T)$. By the uniqueness condition we have
$$
\langle V,f|_{V\setminus\{x\}}\rangle\cap\langle V,g|_{V\setminus\{x\}}\rangle=\emptyset.
$$
This follows from that otherwise there exists $a\in V$ such that $\langle f\rangle_a=\langle g\rangle_a$, which implies $f|_{V\setminus\{x\}}=g|_{V\setminus\{x\}}$, hence $\langle f\rangle_x=\langle g\rangle_y$ ($x=y$), a contradiction. Consequently, $({\mathfrak F}^\circ,{\mathscr T}^\circ(\mathcal F))$ is Hausdorff.

To prove $({\mathfrak F}^\circ,{\mathscr T}^\circ(\mathcal F); \hat{\mathfrak F},\hat{\mathscr T}(\mathcal F))$ is Hausdorff, we assume $x\in X$, $y\in\hat X$, $x\ne y$ and $\langle f\rangle_x\ne\langle g\rangle_y$, where $f\in\mathcal F(U)$ ($U\in\overset{\scriptscriptstyle\circ}{\mathscr T}(x)$) and $g\in\mathcal F(B)$ ($B\in\mathscr B(y)$, $\mathscr B$ is a connected punctured partial neighborhood basis). Then there exist $U_1\in{\mathscr T}(x)$ and $\hat B_1\in\hat{\mathscr B}(y)$ such that $U_1\subseteq U\cup\{x\}$, $\hat B_1\subseteq B\cup\{y\}$ and $U_1\cap\hat B_1=\emptyset$. Thus we have
$$
\langle U_1,f\rangle\cap\langle\hat B_1,g\rangle=\emptyset. \eqno \square
$$

\smallskip
Recall $\overset{\scriptscriptstyle\circ}{\mathfrak F}$ is the disjoint union of all the punctured stalks over~$X$. Denote
$$
[U,f]^\circ:=\{[f]^\circ_x:\,x\in U\}
$$
for $U\in\mathscr T$ and $f\in\overset{\scriptscriptstyle\,\,\circ}{\mathcal F}(U)$, and define
$$
\overset{\scriptscriptstyle\circ}{\mathscr N}(\mathcal F):=\{[U,f]^\circ:\,U\in\mathscr T,\ f\in\overset{\scriptscriptstyle\,\,\circ}{\mathcal F}(U)\}.
$$
As a special case of Theorems~\ref{theorem2.1} and~\ref{theorem2.2} (refer to Remark~\ref{remark1}) we have 

\begin{theorem}\label{theorem2.3}
Suppose $(X,\mathscr T)$ is a strongly locally connected Hausdorff topological space and $\mathcal F$ is a presheaf on~$(X,\mathscr T)$ which satisfies the uniqueness condition. Then $\overset{\scriptscriptstyle\circ}{\mathscr N}(\mathcal F)$ is a basis for a topology on~$\overset{\scriptscriptstyle\circ}{\mathfrak F}$, denoted~$\overset{\scriptscriptstyle\circ}{\mathscr T}(\mathcal F)$. Moreover, $(\overset{\scriptscriptstyle\circ}{\mathfrak F},\overset{\scriptscriptstyle\circ}{\mathscr T}(\mathcal F))$ is a Hausdorff space and the projection
$$
\overset{\scriptscriptstyle\circ}p:\ (\overset{\scriptscriptstyle\circ}{\mathfrak F},\overset{\scriptscriptstyle\circ}{\mathscr T}(\mathcal F))\longrightarrow(X,{\mathscr T})
$$
is a local homeomorphism.  \hfill $\square$
\end{theorem}

Here the \textit{uniqueness condition} on a presheaf~$\mathcal F$ on a $T_1$~space $(X,\mathscr T)$ means: For every domain~$Y$ in~$X$, the equality $[f]^\circ_a=[g]^\circ_a$, where $f$, $g\in\mathcal F(Y)$ and $a$ is any point in~$X$ satisfying $Y\!\setminus\!\{a\}$ is a punctured neighborhood of~$a$, always implies $f=g$.

We can also directly prove Theorem~\ref{theorem2.3} similarly to~\cite[Theorems~(6.8) and~(6.10)]{Fo}.

\subsection{A porous universal topological space}

Suppose $X$ is a topological space. Let $U$ be an open set and $E$ a set in~$X$. If for any nonempty domain $D\subseteq U$, the interior of $D\setminus E$ is a nonempty subdomain of~$D$ and for any open subset~$G$ of~$U$ we have $G\subseteq\overline{G\setminus\!E}$ (the closure of $G\setminus\!E$), then we say that $E$ is \textit{quasi-discrete} in~$U$ and $U\setminus E$ is \textit{porous} corresponding to~$U$. Here we also attach the following conditions: (1) if $E_1$, $E_2$ and~$E$ are quasi-discrete in~$U$ then $E_1\cup E_2$ and all subsets of~$E$ are quasi-discrete in~$U$; (2) every discrete subset of~$X$ is quasi-discrete.

If $(X,\mathscr T;\hat{X},\hat{\mathscr T})$ is a Hausdorff universal topological space (i.e.\ $(\hat{X},\check{\mathscr T})$ is Hausdorff) and $(X,\mathscr T)$ is $T_3$ (Hausdorff and regular), then we say $\hat X$ is $T_3$. Suppose $(X,\mathscr T;\hat{X},\hat{\mathscr T})$ is a $T_3$ universal topological space with a basic punctured partial neighborhood basis~$\mathscr B$ determined by a perfect filterbase structure system~$\mathfrak B$. 
For $\mathcal B\in\mathfrak B$, denote
$$
\mathcal B^{^{\mathrm{po}}}:=\{B\!\setminus\!E:\,B\in\mathcal B \text{ and } E \text{ is closed and quasi-discrete in } B\}
$$
and
$$
\mathfrak B^{^{\mathrm{po}}}:=\{\mathcal B^{^{\mathrm{po}}}:\,\mathcal B\in\mathfrak B\}.
$$
Then easily we see that $\mathcal B^{^{\mathrm{po}}}$ is a filter base, which we call a \textit{porous filter base} corresponding to~$\mathcal B$ or~$\mathfrak B$. Assume $\mathcal B^{^{\mathrm{po}}}\to x$ precisely if $\mathcal B\to x$ for an ideal point~$x$ in~$\hat X$. Then $\mathfrak B^{^{\mathrm{po}}}$ is a perfect filterbase structure system on~$(X,\mathscr T)$, which we call a \textit{perfect porous filterbase structure system} corresponding to~$\mathfrak B$. Let $\hat{\mathscr B}^{^{\mathrm{po}}}(x)$ (resp.\ $\mathscr B^{^{\mathrm{po}}}(x)$) be a basic (resp.\ punctured) partial neighborhood basis at~$x\in\hat X$ corresponding to~$\mathfrak B^{^{\mathrm{po}}}$. Let
$$
\hat{\mathscr B}^{^{\mathrm{po}}}:=\bigcup_{x\in\hat X}\hat{\mathscr B}^{^{\mathrm{po}}}(x)
$$
and
$$
{\mathscr B}^{^{\mathrm{po}}}:=\bigcup_{x\in\hat X}{\mathscr B}^{^{\mathrm{po}}}(x),
$$
which are called the \textit{basic porous partial neighborhood basis} and \textit{basic punctured porous partial neighborhood basis} for $(X,\mathscr T;\hat{X},\hat{\mathscr T})$, respectively.

Suppose $\hat{\mathscr T}^{^{\mathrm{po}}}$ is the filterbase topology determined by~$\mathfrak B^{^{\mathrm{po}}}$. We call $\hat{\mathscr T}^{^{\mathrm{po}}}$ the \textit{porous filterbase topology} or \textit{porous partial topology} on~$\hat X$ determined by~$\mathfrak B$ and $(\hat X,\hat{\mathscr T}^{^{\mathrm{po}}})$ the \textit{porous partial toplogical space} corresponding to $(\hat X,\hat{\mathscr T})$. Let $\hat{\mathscr T}^{^{\circ\mathrm{p}}}(x)$ denote the set of all punctured open partial neighborhoods of~$x\in\hat X$ in~$(\hat X,\hat{\mathscr T}^{^{\mathrm{po}}})$ and let
$$
\hat{\mathscr T}^{^{\circ\mathrm{p}}}:=\bigcup_{x\in\hat X}\hat{\mathscr T}^{^{\circ\mathrm{p}}}(x).
$$
We call $\hat{\mathscr T}^{^{\circ\mathrm{p}}}$ (resp.\ $\hat{\mathscr T}^{^{\circ\mathrm{p}}}(x)$) the \textit{punctured porous open partial neighborhood system} (resp.\ at~$x$) on~$\hat{X}$.

Similarly by the perfect porous filterbase structure system
\begin{equation}\label{2.2}
\mathfrak T^{^{\mathrm{po}}}:=\{\mathcal B^{^{\mathrm{po}}}:\,\mathcal B\in\overset{\scriptscriptstyle\circ}{\mathfrak T}\}
\end{equation}
in~$X$, where $\overset{\scriptscriptstyle\circ}{\mathfrak T}:=\{\overset{\scriptscriptstyle\circ}{\mathscr T}(x):\,x\in X\}$ ($\overset{\scriptscriptstyle\circ}{\mathscr T}(x)$ is the set of all punctured open neighborhoods of $x\in X$ in $(X,\mathscr T)$), we obtain a filterbase topology, denoted~$\overset{\,\scriptscriptstyle\mathrm{po}}{\mathscr T}$, which is called the \textit{porous topology} on~$X$. The space $(X,\overset{\,\scriptscriptstyle\mathrm{po}}{\mathscr T})$ is called the \textit{porous topological space} corresponding to $(X,\mathscr T)$. Let $\overset{\scriptscriptstyle\circ\mathrm{p}}{\mathscr T}(x)$ denote the set of all punctured open neighborhoods of $x\in X$ in~$(X,\overset{\,\scriptscriptstyle\mathrm{po}}{\mathscr T})$ and let
$$
\overset{\scriptscriptstyle\circ\mathrm{p}}{\mathscr T}:=\bigcup_{x\in\hat X}\overset{\scriptscriptstyle\circ\mathrm{p}}{\mathscr T}(x).
$$
We call $\overset{\scriptscriptstyle\circ\mathrm{p}}{\mathscr T}$ (resp.\ $\overset{\scriptscriptstyle\circ\mathrm{p}}{\mathscr T}(x)$) the \textit{punctured porous open neighborhood system} (resp.\ at~$x$) on $(X,\mathscr T)$.

For a partial point set~$A$ ($\mathtt I\subseteq A\subseteq\hat X$), let $\hat{\mathscr B}^{^{\mathrm{po}}}_A(x)$ (resp.\ $\mathscr B^{^{\mathrm{po}}}_A(x)$) be a (resp.\ punctured) porous open neighborhood basis at~$x$ in $(X,\mathscr T)$ for $x\in X\!\setminus\!A$ and $\hat{\mathscr B}^{^{\mathrm{po}}}_A(x):=\hat{\mathscr B}^{^{\mathrm{po}}}(x)$ (resp.\ $\mathscr B^{^{\mathrm{po}}}_A(x):=\mathscr B^{^{\mathrm{po}}}(x)$) for $x\in A$. Then $\hat{\mathscr B}^{^{\mathrm{po}}}_A(x)$ and~$\mathscr B^{^{\mathrm{po}}}_A(x)$ are filter bases. Let
$$
\hat{\mathscr B}^{^{\mathrm{po}}}_A:=\bigcup_{x\in\hat X}\hat{\mathscr B}^{^{\mathrm{po}}}_A(x) \quad \text{and} \quad \mathscr B^{^{\mathrm{po}}}_A:=\bigcup_{x\in\hat X}\mathscr B^{^{\mathrm{po}}}_A(x).
$$
Then $\hat{\mathscr B}^{^{\mathrm{po}}}_A$ is a basis for some topology on (set)~$\hat X$. The topology on~$\hat X$ determined by~$\hat{\mathscr B}^{^{\mathrm{po}}}_A$ is called the \textit{mixed porous topology} on~$\hat X$ with the partial point set~$A$ or the \textit{$A$-mixed porous topology} on~$\hat X$ determined by~$\hat{\mathscr B}_A$, denoted~$\hat{\mathscr T}^{^{\mathrm{po}}}(A)$. $\mathscr B^{^{\mathrm{po}}}_A$ (resp.\ $\mathscr B^{^{\mathrm{po}}}_A(x)$) is called an \textit{$A$-mixed punctured porous neighborhood basis} (resp.\ at~$x$) on~$\hat X$. Denote $\check{\mathscr T}^{^{\mathrm{po}}}:=\hat{\mathscr T}^{^{\mathrm{po}}}(\mathtt I)$, which is called the \textit{porous essential topology} on~$\hat X$. We can also define $\check{\mathscr B}^{^{\mathrm{po}}}(x)$ ($x\in\hat X$) and $\check{\mathscr B}^{^{\mathrm{po}}}$ in an obvious way to get the \textit{porous essential topology}~$\check{\mathscr T}^{^{\mathrm{po}}}$.

Since $\hat{X}$ is assumed to be~$T_3$, we easily see that $\mathfrak B^{^{\mathrm{po}}}$ is a perfect filterbase structure system on~$(X,\overset{\,\scriptscriptstyle\mathrm{po}}{\mathscr T})$ and $(X,\overset{\,\scriptscriptstyle\mathrm{po}}{\mathscr T};\hat{X},\hat{\mathscr T}^{^{\mathrm{po}}})$ is a universal topological space determined by~$\mathfrak B^{^{\mathrm{po}}}$ (with a basic (resp.\ punctured) open neighborhood basis~$\hat{\mathscr B}^{^{\mathrm{po}}}$ (resp.~${\mathscr B}^{^{\mathrm{po}}}$)), which we call the \textit{porous universal topological space} corresponding to $(X,\mathscr T;\hat{X},\hat{\mathscr T})$. Clearly, $\hat{\mathscr T}^{^{\mathrm{po}}}(A)$ is a mixed topology of $(X,\overset{\,\scriptscriptstyle\mathrm{po}}{\mathscr T};\hat{X},\hat{\mathscr T}^{^{\mathrm{po}}})$. Specially, $\check{\mathscr T}^{^{\mathrm{po}}}$ is the essential topology of $(X,\overset{\,\scriptscriptstyle\mathrm{po}}{\mathscr T};\hat{X},\hat{\mathscr T}^{^{\mathrm{po}}})$.

Suppose $\mathcal F$ is a presheaf of some algebraic system on~$(X,\mathscr T)$. For $U\in\mathscr T$ let
$$
\overset{\scriptscriptstyle\,\,\mathrm{po}}{\mathcal F}(U):=\{f\in\mathcal F(U\!\setminus\!E):\, E \text{ is a quasi-discrete closed subset of~}U\}.
$$
Let $\hat{\mathcal F}^{^{\mathrm{po}}}_x$ be the punctured partial stalk of~$\mathcal F$ at $x\in\hat X$ in $(\hat{X},\hat{\mathscr T}^{^{\mathrm{po}}})$, called the \textit{punctured porous partial stalk} of~$\mathcal F$ at $x\in\hat X$ in $(\hat{X},\hat{\mathscr T})$. Define
$$
\hat{\mathfrak F}^{^{\mathrm{po}}}:=\bigcup_{x\in\hat X}\hat{\mathcal F}^{^{\mathrm{po}}}_x.
$$

Let $g\in\overset{\scriptscriptstyle\,\,\mathrm{po}}{\mathcal F}(V)$ ($V\in\hat{\mathscr T}^\circ(x)$, $x\in\hat X$) (i.e.\ $\,g\in{\mathcal F}(\hat{\mathscr T}^{^{\circ\mathrm{p}}}(x))$, $\hat{\mathscr T}^{^{\circ\mathrm{p}}}(x)$ is the punctured porous open partial neighborhood system at~$x$). We call the punctured partial germ of~$g$ at~$x$ in $(\hat{X},\hat{\mathscr T}^{^{\mathrm{po}}})$ the \textit{punctured porous partial germ} of~$g$ at~$x$ in $(\hat{X},\hat{\mathscr T})$, denoted $\langle g\rangle^{^{\mathrm{po}}}_x$. For $U\in\mathscr T$, $x\in U$ and $f\in\overset{\scriptscriptstyle\,\,\mathrm{po}}{\mathcal F}(U)$ denote $\langle f\rangle^{^{\mathrm{po}}}_x:=\langle f|_V\rangle^{^{\mathrm{po}}}_x$, where $V\in\hat{\mathscr T}^{^{\circ\mathrm{p}}}(x)$, $V\subseteq U\setminus E$ for some quasi-discrete closed subset~$E$ of~$U$ and $f\in\mathcal F(U\!\setminus\! E)$, and denote
$$
\langle U,f\rangle^{^{\mathrm{po}}}:=\{\langle f\rangle^{^{\mathrm{po}}}_x:\,x\in U\}.
$$
For nonempty $\hat{U}\in\hat{\mathscr T}$ with body~$\hat{U}^\circ$ and $f\in\overset{\scriptscriptstyle\,\,\mathrm{po}}{\mathcal F}(\hat{U}^\circ)$ denote
$$
\langle\hat{U},f\rangle^{^{\mathrm{po}}}:=\{\langle f\rangle^{^{\mathrm{po}}}_x:\,x\in\hat{U}\}.
$$

Define
$$
\hat{\mathscr N}^{^{\mathrm{po}}}(\mathcal F):=\{\langle\hat U,f\rangle^{^{\mathrm{po}}}:\,\emptyset\ne\hat U\in\hat{\mathscr T},\ f\in\overset{\scriptscriptstyle\,\,\mathrm{po}}{\mathcal F}(\hat U^\circ)\}
$$
and
$$
\hat{\mathscr N}^{^{\mathrm{po}}}(\mathcal F(\hat{\mathscr B})):=\{\langle\hat B,f\rangle^{^{\mathrm{po}}}:\,\hat B\in\hat{\mathscr B},\ f\in\overset{\scriptscriptstyle\,\,\mathrm{po}}{\mathcal F}(\hat B^\circ)\},
$$
where $\hat U^\circ\in\mathscr T$ and $\hat B^\circ\in\mathscr B$ are the bodies of~$\hat U$ and~$\hat B$, respectively. For $f\in\overset{\scriptscriptstyle\,\,\mathrm{po}}{\mathcal F}(B)$, where $B\in\mathscr B(x)$ ($x\in\hat X$), define
$$
\hat{\mathscr N}^{^{\mathrm{po}}}_f(\mathcal F(\hat{\mathscr B}))(\langle f\rangle^{^{\mathrm{po}}}_x):=\{\langle\hat U,f|_{\hat U\setminus\{x\}}\rangle^{^{\mathrm{po}}}:\,\hat U\in\hat{\mathscr B}(x),\ \hat U\subseteq\hat B\},
$$
where $\hat B=B\cup\{x\}\in\hat{\mathscr B}(x)$. Then $\hat{\mathscr N}^{^{\mathrm{po}}}_f(\mathcal F(\hat{\mathscr B}))(\langle f\rangle^{^{\mathrm{po}}}_x)$ is a filter base and
$$
\hat{\mathscr N}^{^{\mathrm{po}}}(\mathcal F(\hat{\mathscr B}))=\bigcup\{\hat{\mathscr N}^{^{\mathrm{po}}}_f(\mathcal F(\hat{\mathscr B}))(\langle f\rangle^{^{\mathrm{po}}}_x):\,f\in\overset{\scriptscriptstyle\,\,\mathrm{po}}{\mathcal F}(B),\ B\in\mathscr B(x),\ x\in\hat X\}.
$$
Define
$$
\mathscr N^{^{\mathrm{po}}}(\mathcal F):=\{\langle U,f\rangle^{^{\mathrm{po}}}:\,\emptyset\ne U\in\mathscr T,\ f\in\overset{\scriptscriptstyle\,\,\mathrm{po}}{\mathcal F}(U)\}
$$
and
$$
\mathscr N^{^{\mathrm{po}}}(\mathcal F(\mathscr B)):=\{\langle B,f\rangle^{^{\mathrm{po}}}:\, B\in\mathscr B,\ f\in\overset{\scriptscriptstyle\,\,\mathrm{po}}{\mathcal F}(B)\}.
$$
For $f\in\overset{\scriptscriptstyle\,\,\mathrm{po}}{\mathcal F}(B)$, where $B\in\mathscr B(x)$ ($x\in\hat X$), define
$$
\mathscr N^{^{\mathrm{po}}}_f(\mathcal F(\mathscr B))(\langle f\rangle^{^{\mathrm{po}}}_x):=\{\langle U,f|_ U\rangle^{^{\mathrm{po}}}:\, U\in\mathscr B(x),\ U\subseteq B\}
$$
and
$$
\mathfrak N^{^{\mathrm{po}}}(\mathcal F(\mathscr B)):=\{\mathscr N^{^{\mathrm{po}}}_f(\mathcal F(\mathscr B))(\langle f\rangle^{^{\mathrm{po}}}_x):\,f\in\overset{\scriptscriptstyle\,\,\mathrm{po}}{\mathcal F}(B),\ B\in\mathscr B(x),\ x\in\hat X\}.
$$
Then $\mathscr N^{^{\mathrm{po}}}_f(\mathcal F(\mathscr B))(\langle f\rangle^{^{\mathrm{po}}}_x)$ is a filter base and
$$
\mathscr N^{^{\mathrm{po}}}(\mathcal F(\mathscr B))=\bigcup\{\mathcal N:\,\mathcal N\in\mathfrak N^{^{\mathrm{po}}}(\mathcal F(\mathscr B))\}.
$$

Suppose $\mathcal F$ is a presheaf on a $T_1$ space $(X,\mathscr T)$. Similarly to the punctured porous partial stalk and the punctured porous partial germ we can define the \textit{punctured porous stalk} of~$\mathcal F$ at $x\in X$ in $(X,\mathscr T)$, denoted~$\overset{\,\,\scriptscriptstyle\circ\mathrm{p}}{\mathcal F}_x$, and the \textit{punctured porous germ} of~$f\in\overset{\scriptscriptstyle\,\,\mathrm{po}}{\mathcal F}(U)$ ($U\in\mathscr T$) at $x\in X$, denoted~$[f]^{^{\circ\mathrm{p}}}$, which are the punctured porous partial stalk and the punctured porous partial germ corresponding to the perfect porous filterbase structure system~$\mathfrak T^{^{\mathrm{po}}}$ (see~(\ref{2.2})), respectively. Let
$$
\overset{\scriptscriptstyle\circ\mathrm{p}}{\mathfrak F}:=\bigcup_{x\in X} \overset{\,\scriptscriptstyle\circ\mathrm{p}}{\mathcal F}_x.
$$
If for $g\in\mathcal F(\hat{\mathscr T}^{^{\mathrm{po}}}(x))$ ($x\in X$) there exists $f\in\mathcal F(\overset{\scriptscriptstyle\circ\mathrm{p}}{\mathscr T}(x))$ such that $\langle f\rangle^{^{\mathrm{po}}}_x=\langle g\rangle^{^{\mathrm{po}}}_x$ ($\overset{\scriptscriptstyle\circ\mathrm{p}}{\mathscr T}(x)$ is the punctured porous open neighborhood system at~$x$), then $x$ is called a \textit{porously complete point} of~$\langle g\rangle^{^{\mathrm{po}}}_x$ or~$g$, $f$ a \textit{porously complete element} corresponding to~$g$ at~$x$, $\langle g\rangle^{^{\mathrm{po}}}_x$ \textit{porously complete} and $g$ \textit{porously complete} at~$x$ (here we may also use the terminology \textit{unbranched} to replace ``complete''). In this case, we say that $\langle g\rangle^{^{\mathrm{po}}}_x$ and $[f]^{^{\circ\mathrm{p}}}_x$ are \textit{equivalent}, denoted $\langle g\rangle^{^{\mathrm{po}}}_x\thicksim[f]^{^{\circ\mathrm{p}}}_x$.

Suppose $g_1$, $g_2\in\mathcal F(\hat{\mathscr T}^{^{\mathrm{po}}}(x))$ and $x\in X$ is a common porously complete point of~$g_1$ and~$g_2$. Let $f_1$ and~$f_2$ be porously complete elements corresponding to $g_1$ and~$g_2$ at~$x$, respectively. If $\langle g_1\rangle^{^{\mathrm{po}}}_x=\langle g_2\rangle^{^{\mathrm{po}}}_x$ always implies $[f_1]^{^{\circ\mathrm{p}}}_x=[f_2]^{^{\circ\mathrm{p}}}_x$, then the presheaf~$\mathcal F$ is called \textit{porously consistent} at~$x$ (on~$\hat X$). If $\mathcal F$ is porously consistent at all the porously complete points then we say that $\mathcal F$ is \textit{porously consistent} (on~$\hat X$).

Let
$$
\hat p:\ \hat{\mathfrak F}^{^{\mathrm{po}}}\longrightarrow\hat X \quad\text{and}\quad \overset{\scriptscriptstyle\circ}p:\ \overset{\scriptscriptstyle\circ\mathrm{p}}{\mathfrak F}\longrightarrow X
$$
be the projections, i.e.\ $\hat p(\langle g\rangle^{^{\mathrm{po}}}_x)=x$ for $\langle g\rangle^{^{\mathrm{po}}}_x\in\hat{\mathfrak F}^{^{\mathrm{po}}}$ ($x\in\hat X$) and $\overset{\scriptscriptstyle\circ}p([f]^{^{\circ\mathrm{p}}}_x)=x$ for $[f]^{^{\circ\mathrm{p}}}_x\in\overset{\scriptscriptstyle\circ\mathrm{p}}{\mathfrak F}$ ($x\in X$). Let $\mathcal F^{^{\circ\mathrm{p}}}_x$ be the set of all complete punctured porous partial germs in~$\hat{\mathcal F}^{^{\mathrm{po}}}_x$ ($x\in X$), which is called the \textit{complete} (or \textit{unbranched}) \textit{punctured porous partial stalk} of~$\mathcal F$ at~$x$. Let
$$
\mathfrak F^{^{\circ\mathrm{p}}}:=\bigcup_{x\in X}\mathcal F^{^{\circ\mathrm{p}}}_x,
$$
and $p:=\hat p|_{\mathfrak F^{^{\circ\mathrm{p}}}}$. Then
$$
p:\ \mathfrak F^{^{\circ\mathrm{p}}}\longrightarrow X
$$
is also a projection.

For a partial point set~$A$ let $\hat{\mathcal F}^{^{\mathrm{po}}}_x(A):={\mathcal F}^{^{\circ\mathrm{p}}}_x$ for $x\in X\!\setminus\!A$ and $\hat{\mathcal F}^{^{\mathrm{po}}}_x(A):=\hat{\mathcal F}^{^{\mathrm{po}}}_x$ for $x\in A$. Define
$$
\hat{\mathfrak F}^{^{\mathrm{po}}}(A):=\bigcup_{x\in\hat X}\hat{\mathcal F}^{^{\mathrm{po}}}_x(A)
$$
and
$$
\hat{\mathscr N}^{^{\mathrm{po}}}(\mathcal F(\hat{\mathscr B}_A)):=\bigcup_{x\in X}\{\langle V,f\rangle^{^{\mathrm{po}}}:\,V\in\hat{\mathscr B}_A(x),\ f\in\overset{\,\,\scriptscriptstyle\mathrm{po}}{\mathcal F}(V\!\setminus\!\{x\})\}.
$$

We can obtain the following results corresponding to Theorems~\ref{theorem2.1}, \ref{theorem2.2} and~\ref{theorem2.3} by similar reasoning to the proofs of the theorems, respectively.

\begin{theorem2.1'}\label{theorem2.1'}
Suppose $(X,\mathscr T;\hat{X},\hat{\mathscr T})$ is a $T_3$ universal topological space with a basic open $($resp.\ punctured\,$)$ partial neighborhood basis~$\hat{\mathscr B}$ $($resp.\ $\mathscr B)$. Suppose $\mathcal F$ is a presheaf on $(X,\mathscr T)$. In~$(2)$, $(3)$ and\/~$(4)$ below, we also suppose $\mathcal F$ is porously consistent. Then
\par\medskip
\noindent
$(1)$ $\hat{\mathscr N}^{^{\mathrm{po}}}(\mathcal F(\hat{\mathscr B}))$ is a basis for a topology on~$\hat{\mathfrak F}^{^{\mathrm{po}}}$, and so is $\hat{\mathscr N}^{^{\mathrm{po}}}(\mathcal F)$ for the same topology, denoted~$\hat{\mathscr T}^{^{\mathrm{po}}}(\mathcal F)$, and the projection
$$
\hat p:\ (\hat{\mathfrak F}^{^{\mathrm{po}}},\hat{\mathscr T}^{^{\mathrm{po}}}(\mathcal F))\longrightarrow(\hat X,\hat{\mathscr T})
$$
is a local homeomorphism.
\par\smallskip\noindent
$(2)$ $\mathscr N^{^{\mathrm{po}}}(\mathcal F)$ is a basis for a topology on~$\mathfrak F^{^{\circ\mathrm{p}}}$, denoted~$\mathscr T^{^{\circ\mathrm{p}}}(\mathcal F)$, and the projection
$$
p:\ ({\mathfrak F}^{^{\circ\mathrm{p}}},{\mathscr T}^{^{\circ\mathrm{p}}}(\mathcal F))\longrightarrow(X,{\mathscr T})
$$
is a local homeomorphism.
\par\smallskip\noindent
$(3)$ For a partial point set~$A$ $(\mathtt I\subseteq A\subseteq\hat X)$, $\hat{\mathscr N}^{^{\mathrm{po}}}(\mathcal F(\hat{\mathscr B}_A))$ is a basis for a topology on~$\hat{\mathfrak F}^{^{\mathrm{po}}}(A)$, denoted~$\hat{\mathscr T}^{^{\mathrm{po}}}_A(\mathcal F)$, and the projection
$$
\hat p_{_A}:\ (\hat{\mathfrak F}^{^{\mathrm{po}}}(A),\hat{\mathscr T}^{^{\mathrm{po}}}_A(\mathcal F))\longrightarrow(\hat X,\hat{\mathscr T}(A))
$$
is a local homeomorphism.
\par\smallskip\noindent
$(4)$ $\mathfrak N^{^{\mathrm{po}}}(\mathcal F(\mathscr B))$ is a perfect filterbase structure system on $({\mathfrak F}^{^{\circ\mathrm{p}}},{\mathscr T}^{^{\circ\mathrm{p}}}(\mathcal F))$ and under some obvious assumption $({\mathfrak F}^{^{\circ\mathrm{p}}},{\mathscr T}^{^{\circ\mathrm{p}}}(\mathcal F); \hat{\mathfrak F}^{^{\mathrm{po}}},\hat{\mathscr T}^{^{\mathrm{po}}}(\mathcal F))$ is a universal topological space determined by $\mathfrak N^{^{\mathrm{po}}}(\mathcal F(\mathscr B))$ with the basic $($resp.\ punctured\/$)$ partial neighborhood basis $\hat{\mathscr N}^{^{\mathrm{po}}}(\mathcal F(\hat{\mathscr B}))$ $($resp.\ $\mathscr N^{^{\mathrm{po}}}(\mathcal F(\mathscr B)))$. Moreover, the projection
$$
\hat p:\ \hat{\mathfrak F}^{^{\mathrm{po}}}\longrightarrow\hat X
$$
is an exact local homeomorphism and hence an essential local homeomorphism.  \hfill $\square$
\end{theorem2.1'}

We call the space $({\mathfrak F}^{^{\circ\mathrm{p}}},{\mathscr T}^{^{\circ\mathrm{p}}}(\mathcal F); \hat{\mathfrak F}^{^{\mathrm{po}}},\hat{\mathscr T}^{^{\mathrm{po}}}(\mathcal F))$ in the above theorem the \textit{derived porous universal topological space} over $(X,\mathscr T;\hat X,\hat{\mathscr T})$ by~$\mathcal F$.
\smallskip

Suppose $\mathcal F$ is a presheaf on $(X,\mathscr T)$ and $Y$ is an open set in~$X$. Let $f$, $g\in\overset{\scriptscriptstyle\,\,\mathrm{po}}{\mathcal F}(Y)$. If there exists a quasi-discrete closed subset~$E$ of~$Y$ such that $f|_{Y\setminus E}=g|_{Y\setminus E}$, then we say that $f$ is \textit{porously equal} to~$g$, denoted $f\overset{\scriptscriptstyle\mathrm{po}}=g$ (if $f$, $g\in\overset{\scriptscriptstyle\,\,\circ}{\mathcal F}(Y)$ (see~(\ref{2.0})) and $f|_{Y\setminus E}=g|_{Y\setminus E}$, where $E$ is a discrete subset of~$Y$, then we say that $f$ is \textit{permissibly equal} to~$g$, denoted $f\overset{\scriptscriptstyle\circ}=g$). The presheaf~$\mathcal F$ on $(X,\mathscr T)$ is said to satisfy the \textit{porous uniqueness condition} on a universal topological space $(X,\mathscr T;\hat X,\hat{\mathscr T})$ if for every domain~$Y$ in~$(X,\mathscr T)$, given any $f$, $g\in\overset{\scriptscriptstyle\,\,\mathrm{po}}{\mathcal F}(Y)$ and any $a\in\hat X$ satisfying there exists $B\in\mathscr B(a)$ with $B\subseteq Y$, the equality $\langle f\rangle^{^{\mathrm{po}}}_a=\langle g\rangle^{^{\mathrm{po}}}_a$ always implies $f\overset{\scriptscriptstyle\mathrm{po}}=g$.

\begin{theorem2.2'}\label{2.2'}
Suppose $(X,\mathscr T;\hat X,\hat{\mathscr T})$ is a locally connected $T_3$ universal topological space and $\mathcal F$ is a presheaf on~$(X,\mathscr T)$. If $\mathcal F$ satisfies the porous uniqueness condition on~$\hat X$, then both $({\mathfrak F}^{^{\circ\mathrm{p}}},{\mathscr T}^{^{\circ\mathrm{p}}}(\mathcal F))$ and $(\hat{\mathfrak F}^{^{\mathrm{po}}},\hat{\mathscr T}^{^{\mathrm{po}}}(\mathcal F))$ are Hausdorff spaces, furthermore, $({\mathfrak F}^{^{\circ\mathrm{p}}},{\mathscr T}^{^{\circ\mathrm{p}}}(\mathcal F); \hat{\mathfrak F}^{^{\mathrm{po}}},\hat{\mathscr T}^{^{\mathrm{po}}}(\mathcal F))$ is a $T_3$ universal topological space.
\hfill $\square$
\end{theorem2.2'}

For nonempty $U\in\mathscr T$ and $f\in\overset{\scriptscriptstyle\,\,\mathrm{po}}{\mathcal F}(U)$ we denote
$$
[U,f]^{^{\circ\mathrm{p}}}:=\{[f]^{^{\circ\mathrm{p}}}_x:\,x\in U\}
$$
and define
$$
\overset{\scriptscriptstyle\circ\mathrm{p}}{\mathscr N}(\mathcal F):=\{[U,f]^{^{\circ\mathrm{p}}}:\,\emptyset\ne U\in\mathscr T,\ f\in\overset{\scriptscriptstyle\,\,\mathrm{po}}{\mathcal F}(U)\}.
$$

\begin{theorem2.3'}\label{2.3'}
Suppose $(X,\mathscr T)$ is a strongly locally connected $T_3$ topological space and $\mathcal F$ is a presheaf on~$(X,\mathscr T)$ which satisfies the porous uniqueness condition. Then \,$\overset{\scriptscriptstyle\circ\mathrm{p}}{\mathscr N}(\mathcal F)$ is a basis for a topology on~$\overset{\scriptscriptstyle\circ\mathrm{p}}{\mathfrak F}$, denoted~$\overset{\scriptscriptstyle\circ\mathrm{p}}{\mathscr T}(\mathcal F)$. Moreover, $(\overset{\scriptscriptstyle\circ\mathrm{p}}{\mathfrak F},\overset{\scriptscriptstyle\circ\mathrm{p}}{\mathscr T}(\mathcal F))$ is a $T_3$ space and the projection
$$
\overset{\scriptscriptstyle\circ}p:\ (\overset{\scriptscriptstyle\circ\mathrm{p}}{\mathfrak F},\overset{\scriptscriptstyle\circ\mathrm{p}}{\mathscr T}(\mathcal F))\longrightarrow(X,{\mathscr T})
$$
is a local homeomorphism.  \hfill $\square$
\end{theorem2.3'}

The \textit{porous uniqueness condition} in Theorem~2.3$'$ means: For every domain~$Y$ in $(X,\mathscr T)$, the equality $[f]^{^{\circ\mathrm{p}}}_a=[g]^{^{\circ\mathrm{p}}}_a$, where $f$, $g\in\overset{\scriptscriptstyle\,\,\mathrm{po}}{\mathcal F}(Y)$ and $a$ is any point in~$X$ satisfying $Y\!\setminus\!\{a\}$ is a punctured neighborhood of~$a$, always implies $f\overset{\scriptscriptstyle\mathrm{po}}=g$.

\medskip
In the following we present a general case. Suppose $(X,\mathscr T;\hat X,\hat{\mathscr T})$ is a $T_3$ universal topological space and $\mathsf{P}$ is a subset of~$\hat X$. Let $U$ be an open set in $(X,\mathscr T)$ and $E$ a set in~$X$. If $E$ is a quasi-discrete in~$U$ and for $x\in\hat X\!\setminus\!\mathsf{P}$ there exists a punctured neighborhood~$\overset{\scriptscriptstyle\,\circ}N(x)$ of~$x$ in $(\hat{X},\check{\mathscr T})$ such that $\overset{\scriptscriptstyle\,\circ}N(x)\cap U\subseteq U\!\setminus\!E$, then we say that $E$ is \textit{${\mathsf{P}}$-quasi-discrete} in~$U$ and $U\!\setminus\!E$ is \textit{${\mathsf{P}}$-porous} corresponding to~$U$. We call $\mathsf P$ a \textit{porous point set}.

If we use ${\mathsf{P}}$-quasi-discrete (resp.\ ${\mathsf{P}}$-porous) sets to replace quasi-discrete (resp.\ porous) sets in the preceding part of this subsection, then we can get corresponding notions and results. What we need to do is just to replace ``quasi", ``porous(ly)", ``$\mathrm{po}$" and ``$\circ\mathrm{p}$" by ``${\mathsf{P}}$-quasi", ``${\mathsf{P}}$-porous(ly)", ``$\mathsf{P}$" and ``$\circ\mathsf{P}$", respectively. For instance, we have a \textit{${\mathsf{P}}$-porous filter base}
$$
\mathcal B^{^{\mathsf{P}}}:=\{B\!\setminus\!E:\,B\in\mathcal B \text{ and } E \text{ is closed and ${\mathsf{P}}$-quasi-discrete in } B\},
$$
a \textit{perfect ${\mathsf{P}}$-porous filterbase structure system} (in $X$ corresponding to a perfect filterbase structure system~$\mathfrak B$)
$$
\mathfrak B^{\!^{\mathsf{P}}}:=\{\mathcal B^{^{\mathsf{P}}}:\,\mathcal B\in\mathfrak B\},
$$
the \textit{basic} (resp.\ \textit{punctured}) \textit{${\mathsf{P}}$-porous partial neighborhood basis}~$\hat{\mathscr B}^{^{\mathsf{P}}}$ (resp.~$\mathscr B^{^{\mathsf{P}}}$) (on $(X,\mathscr T;\hat X,\hat{\mathscr T})$),
the \textit{${\mathsf{P}}$-porous filterbase topology}~$\hat{\mathscr T}^{^{\mathsf{P}}}$, the \textit{${\mathsf{P}}$-porous universal topological space} $(X,\overset{\scriptscriptstyle\,\mathsf{P}}{\mathscr T};\hat X,\hat{\mathscr T}^{^{\mathsf{P}}})$ (corresponding to $(X,\mathscr T;\hat X,\hat{\mathscr T})$), and so on. We list some other notations as follows:
$\hat{\mathscr B}^{^{\mathsf{P}}}(x)$,
$\mathscr B^{^{\mathsf{P}}}(x)$,
$\hat{\mathscr T}^{^{\circ\mathsf{P}}}(x)$,
$\hat{\mathscr T}^{^{\circ\mathsf{P}}}$,
$\overset{\scriptscriptstyle\circ\mathsf{P}}{\mathscr T}(x)$,
$\overset{\scriptscriptstyle\circ\mathsf{P}}{\mathscr T}$,
$\hat{\mathscr B}^{^{\mathsf{P}}}_A(x)$,
$\mathscr B^{^{\mathsf{P}}}_A(x)$,
$\hat{\mathscr B}^{^{\mathsf{P}}}_A$, $\mathscr B^{^{\mathsf{P}}}_A$,
$\hat{\mathscr T}^{^{\mathsf{P}}}(A)$,
$\check{\mathscr B}^{^{\mathsf{P}}}(x)$,
$\check{\mathscr B}^{^{\mathsf{P}}}$,
$\check{\mathscr T}^{^{\mathsf{P}}}$,
$\overset{\scriptscriptstyle\,\,\mathsf{P}}{\mathcal F}(U)$,
$\hat{\mathcal F}^{^{\mathsf{P}}}_x$,
$\hat{\mathfrak F}^{^{\mathsf{P}}}$,
$\langle g\rangle^{\!^{\mathsf{P}}}_x$
(here $\langle g\rangle^{\!^{\mathsf{P}}}_x=\langle g\rangle_x$ for $x\in\hat X\!\setminus\!\mathsf{P}$ and $\langle g\rangle^{\!^{\mathsf{P}}}_x\subseteq\langle g\rangle^{\!^{\mathrm{po}}}_x$ for $x\in\mathsf P$; we may assume $\langle g\rangle^{\!^{\mathsf{P}}}_x=\langle g\rangle^{\!^{\mathrm{po}}}_x$ for $x\in\mathsf P$),
$\langle U,f\rangle^{\!^{\mathsf{P}}}$, $\langle\hat{U},f\rangle^{\!^{\mathsf{P}}}$, $\hat{\mathscr N}^{^{\mathsf{P}}}(\mathcal F)$, $\hat{\mathscr N}^{^{\mathsf{P}}}(\mathcal F(\hat{\mathscr B}))$,
$\hat{\mathscr N}^{^{\mathsf{P}}}_f(\mathcal F(\hat{\mathscr B}))(\langle f\rangle^{\!^{\mathsf{P}}}_x)$,
${\mathscr N}^{^{\mathsf{P}}}(\mathcal F)$,
${\mathscr N}^{^{\mathsf{P}}}(\mathcal F(\mathscr B))$,
${\mathscr N}^{^{\mathsf{P}}}_f(\mathcal F(\mathscr B))(\langle f\rangle^{\!^{\mathsf{P}}}_x)$,
${\mathfrak N}^{^{\mathsf{P}}}(\mathcal F(\mathscr B))$,
$\overset{\scriptscriptstyle\,\,\circ\mathsf{P}}{\mathcal F}_x$,
$[f]^{^{\circ\mathsf{P}}}_x$
(here $[f]^{^{\circ\mathsf{P}}}_x=[f]^{\circ}_x$ for $x\in\hat X\!\setminus\!\mathsf{P}$ and $[f]^{^{\circ\mathsf{P}}}_x$ is assumed to be $[f]^{^{\circ\mathrm{p}}}_x$ for $x\in\mathsf P$),
$\overset{\scriptscriptstyle\circ\mathsf{P}}{\mathfrak F}$,
$\mathcal F^{^{\circ\mathsf{P}}}_x$,
$\mathfrak F^{^{\circ\mathsf{P}}}$,
$\hat{\mathcal F}^{^{\mathsf{P}}}_x(A)$,
$\hat{\mathfrak F}^{^{\mathsf{P}}}(A)$,
$\hat{\mathscr N}^{^{\mathsf{P}}}(\mathcal F(\hat{\mathscr B}_A))$,
$\hat{\mathscr T}^{^{\mathsf{P}}}(\mathcal F)$,
${\mathscr T}^{^{\circ\mathsf{P}}}(\mathcal F)$,
$\hat{\mathscr T}^{^{\mathsf{P}}}_A(\mathcal F)$,
$\overset{\scriptscriptstyle\mathsf{P}}=$,
$[U,f]^{^{\circ\mathsf{P}}}$,
$\overset{\scriptscriptstyle\,\circ\mathsf{P}}{\mathscr N}(\mathcal F)$,
$\overset{\scriptscriptstyle\circ\mathsf{P}}{\mathscr T}(\mathcal F)$
and
$({\mathfrak F}^{^{\circ\mathsf{P}}},{\mathscr T}^{^{\circ\mathsf{P}}}(\mathcal F); \hat{\mathfrak F}^{^{\mathsf{P}}},\hat{\mathscr T}^{^{\mathsf{P}}}(\mathcal F))$.

\smallskip
\begin{remark}\label{remark0}
In the $\mathsf P$-porous case, the results corresponding to Theorems~2.1$'$, 2.2$'$ and~2.3$'$ are true by similar reasoning to Theorems~\ref{theorem2.1}, \ref{theorem2.2} and~\ref{theorem2.3}, which are denoted by Theorems~2.1$''$, 2.2$''$ and~2.3$''$, respectively.
\end{remark}
\smallskip

At the end of this section, we remark that in natural ways we may define the following equivalences: $\varphi_x\thicksim\psi_x$, where $\varphi_x$ and~$\psi_x$ are chosen from $\{[f]_x,\ [f]^\circ_x,\ [f]^{^{\circ\mathrm{p}}}_x,\ \langle f\rangle_x,\ \langle f\rangle^{^{\!\mathrm{po}}}_x\}$ and $\{[g]_x,\ [g]^\circ_x,\ [g]^{^{\circ\mathrm{p}}}_x,\ \langle g\rangle_x,\ \langle g\rangle^{^{\!\mathrm{po}}}_x\}$, respectively.
Specially if $\varphi_x$ and~$\psi_x$ are the same kind of germs then $\varphi_x\thicksim\psi_x$ means $\varphi_x=\psi_x$. Some of the equivalences have been defined in the preceding paragraphs and here as another example we define $[f]^{^{\circ\mathrm{p}}}_x\thicksim\langle g\rangle_x$ as follows: Suppose $f\in\mathcal F(\overset{\scriptscriptstyle\circ\mathrm{p}}{\mathscr T}(x))$ and $g\in\mathcal F(\hat{\mathscr T}^\circ(x))$. If there exist $V\in\hat{\mathscr T}^\circ(x)$ and a quasi-discrete closed subset~$E$ of~$V$ such that $f|_{V\setminus E}=g|_{V\setminus E}$ and thus $\langle f\rangle^{^{\!\mathrm{po}}}_x$\:($:=\langle f|_{V\setminus E}\rangle^{^{\!\mathrm{po}}}_x$)$\;=\langle g\rangle^{^{\!\mathrm{po}}}_x$, then we say that $[f]^{^{\circ\mathrm{p}}}_x$ and~$\langle g\rangle_x$ are \textit{equivalent}. If necessary, we may regard two equivalent germs as the same.

\section{ Algebraic functions and Riemann surfaces}

\subsection{A basic Riemann surface}

Suppose $X$ is a Riemann surface (in the usual sense, see e.g.\ \cite{Do}, \cite{Fo} and~\cite{We}, which we will call a \textit{traditional Riemann surface} later) and $\mathscr T$ is its topology. Now we choose a perfect filterbase structure system~$\mathfrak B$, whose elements consist of domains (usually simply connected ones), and then obtain a universal topological space $(X,\mathscr T;\hat X,\hat{\mathscr T})$, which we call a \textit{basic Riemann surface}. Here we also assume $\hat X$ is Hausdorff (i.e.\ $(\hat X,\check{\mathscr T})$ is Hausdorff).

We recall the notion of a universal topological subspace defined in the end of Subsection~2.1.  Suppose $\hat Y\subseteq\hat X$ and $\mathtt I_0\subseteq X\cap\hat Y$ satisfy that $\mathtt I_0$ is closed and discrete in $(X,\mathscr T)$, $\hat Y$ is an open subset of $(\hat X,{\mathscr T}(\mathtt I\cup\mathtt I_0))$ ($\mathtt I$ is the ideal point set of the basic Riemann surface~$\hat X$ and ${\mathscr T}(\mathtt I\cup\mathtt I_0)$ is the $(\mathtt I\cup\mathtt I_0)$-mixed topology of~$\hat X$) and $Y:=(\hat Y\cap X)\!\setminus\!\mathtt I_0\ne\emptyset$. Then $Y$ is an open subset of~$(X,\mathscr T)$ and the space $(\hat Y,{\mathscr T}(\mathtt I\cup\mathtt I_0)|_{\hat Y})$ is Hausdorff (${\mathscr T}(\mathtt I\cup\mathtt I_0)|_{\hat Y}$ is the induced topology).  If $Y$ is connected then $Y$ is a traditional Riemann surface. Let $\mathfrak B_Y=\{\mathcal B_Y(y):\,y\in\hat Y\}$ be the induced perfect filterbase structure system by~$\mathfrak B$ on~$Y$. Then $\mathcal B_Y(y)\thicksim\mathcal B(y)$ for each $y\in\hat Y$ and the universal topological subspace $(Y,\mathscr T';\hat Y,\hat{\mathscr T}')$ is the universal topological space determined by~$\mathfrak B_Y$. If $Y$ is connected then $(Y,\mathscr T';\hat Y,\hat{\mathscr T}')$  is also a basic Riemann surface, which we call a \textit{basic} (\textit{Riemann}) \textit{subsurface} of the basic Riemann surface $(X,\mathscr T;\hat X,\hat{\mathscr T})$. This kind of subsurface is similar to a domain in a traditional Riemann surface. So generally the propositions which hold on a basic Riemann surface are also true on ``domains" in a basic Riemann surface.

\subsection{Analytic continuation}

Suppose $(X,\mathscr T;\hat X,\hat{\mathscr T})$ is a basic Riemann surface and $\hat{\mathscr T}(A)$ is a mixed topology of~$\hat X$, where $\mathtt{I}\subseteq A\subseteq\hat X$ ($\mathtt{I}$ is the ideal point set of~$\hat X$). Let $(\mathcal H,\rho)$ denote the sheaf of holomorphic functions on $(X,\mathscr T)$. Suppose $u:[0,1]\to\hat X$ is a curve in $(\hat X, \hat{\mathscr T}(A))$ (i.e.\ $\,u:[0,1]\to(\hat X, \hat{\mathscr T}(A))$ is continuous), which is called an \textit{$A$-curve} in~$\hat X$, and $a=u(0)$, $b=u(1)$. If $\hat{\mathscr T}(A)=\check{\mathscr T}$ (the essential topology of~$\hat X$) then $u$ is called an \textit{essential curve} in~$\hat X$. Obviously, $A$-curves are essential curves. 

Let $\widetilde{\mathcal H}_x$ denote the set of all kinds of germs of~$\mathcal H$ at~$x\in\hat X$, i.e.\ $\widetilde{\mathcal H}_x:=\mathcal H_x\cup\overset{\scriptscriptstyle\,\,\circ}{\mathcal H}_x\cup\hat{\mathcal H}_x\cup\hat{\mathcal H}^{^{\mathrm{po}}}_x\cup\overset{\scriptscriptstyle\,\circ\mathrm{p}}{\mathcal H}_x$ for $x\in X$ and $\widetilde{\mathcal H}_x:=\hat{\mathcal H}_x\cup\hat{\mathcal H}^{^{\mathrm{po}}}_x$ for $x\in\mathtt I$.
Suppose $\mathsf P\subseteq\hat X$.
Let $\hat{\mathcal H}^{^{\mathsf P}}_x(A):={\mathcal H}^{^{\circ\mathsf{P}}}_x$ for $x\in X\!\setminus\!A$ and $\hat{\mathcal H}^{^{\mathsf P}}_x(A):=\hat{\mathcal H}^{^{\mathsf P}}_x$ for $x\in A$,\, where ${\mathcal H}^{^{\circ\mathsf{P}}}_x\!:={\mathcal H}^{\circ}_x$ for $x\in X\!\setminus\!\mathsf P$, ${\mathcal H}^{^{\circ\mathsf{P}}}_x:={\mathcal H}^{^{\circ\mathrm{p}}}_x$ for $x\in X\cap\mathsf P$, $\hat{\mathcal H}^{^{\mathsf P}}_x:=\hat{\mathcal H}_x$ for $x\in\hat X\!\setminus\!\mathsf P$ and $\hat{\mathcal H}^{^{\mathsf P}}_x:=\hat{\mathcal H}^{^{\mathrm{po}}}_x$ for $x\in\mathsf P$. Suppose $\varphi_{_{\scriptstyle a}}\in\widetilde{\mathcal H}_a$ and $\varphi_{_{\scriptstyle b}}\in\widetilde{\mathcal H}_b$. If there exists a family $\{\psi_{_{\scriptstyle t}}\in\hat{\mathcal H}^{^{\mathsf P}}_{u(t)}(A):\,t\in[0,1]\}$ such that $\psi_{_{\scriptstyle a}}\thicksim\varphi_{_{\scriptstyle a}}$, $\psi_{_{\scriptstyle b}}\thicksim\varphi_{_{\scriptstyle b}}$ and for every $s\in[0,1]$, there exists a neighborhood $T\subseteq[0,1]$ of~$s$, a domain~$U$ in $(\hat X,\hat{\mathscr T}(A))$ with $u(T)\subseteq U$ and $f\in\overset{\scriptscriptstyle\,\,\mathsf P}{\mathcal H}(U)$ ($\overset{\scriptscriptstyle\,\,\mathsf P}{\mathcal H}(U):=\{f\in\mathcal H(U\!\setminus\!E):\,E \text{ is a $\mathsf{P}$-quasi-discrete closed subset of~}U\}$) satisfying $\langle f\rangle^{^{\mathsf P}}_{u(t)}=\psi_{_{\scriptstyle t}}$ for every $t\in T$, where $\langle f\rangle^{^{\mathsf P}}_x:=\langle f\rangle^{^{\mathrm{po}}}_x$ for $x\in\mathsf P$ and $\langle f\rangle^{^{\mathsf P}}_x:=\langle f\rangle_x$ for $x\in\hat X\setminus\mathsf P$, then we say that $\varphi_{_{\scriptstyle b}}$ is a \textit{$\mathsf P$-porous} (\textit{analytic}) \textit{continuation} of~$\varphi_{_{\scriptstyle a}}$ along~$u$ in $(\hat X,\hat{\mathscr T}(A))$ or that $\varphi_{_{\scriptstyle b}}$ is a \textit{$\mathsf P$-porous $A$-analytic continuation} (or \textit{$\mathsf P$-porous $A$-continuation}) of~$\varphi_{_{\scriptstyle a}}$ along~$u$. Corresponding to $\mathsf P=\emptyset$ (resp.\ $\mathsf P=\hat X$) the continuation is called an (resp.\ a \textit{porous}) \textit{$A$-analytic continuation} (or \textit{$A$-continuation}). The (resp.\ $\mathsf P$-porous, porous) $\mathtt I$-continuation is also called an (resp.\ a \textit{$\mathsf P$-porous}, a \textit{porous}) \textit{essential} (\textit{analytic}) \textit{continuation}. If the patial point set $\mathtt I=\emptyset$ then essential curves and essential analytic continuations are \textit{curves} and \textit{analytic continuations} in the usual sense, respectively. It is obvious that the ($\mathsf P$-porous) $A$-continuation implies the porous $A$-continuation.

Easily we see that $\varphi_{_{\scriptstyle b}}$ is a $\mathsf P$-porous $A$-continuation of~$\varphi_{_{\scriptstyle a}}$ along~$u$ if and only if the following holds: There exist a partition $0=t_0<t_1<\cdots<t_n=1$ of $[0,1]$, domains~$U_j$ in $(\hat X,\hat{\mathscr T}(A))$ with $u([t_{j-1},t_j])\subseteq U_j$ and $f_j\in{\mathcal H}(U_j\!\setminus\! E_j)$ ($E_j$ is a $\mathsf P$-quasi-discrete and closed in~$U_j$, $j=1$, $2$, $\dots$,~$n$) such that $\langle f_1\rangle^{^{\mathsf P}}_a\thicksim\varphi_{_{\scriptstyle a}}$, $\langle f_n\rangle^{^{\mathsf P}}_b\thicksim\varphi_{_{\scriptstyle b}}$ and $f_j|_{V_j\setminus(E_j\cup E_{j+1})}=f_{j+1}|_{V_j\setminus(E_j\cup E_{j+1})}$ ($j=1$, $2$, $\dots$, $n\!-\!1$), where $V_j$ is the connected component of $U_j\cap U_{j+1}$ containing~$u(t_j)$.

Let
$$
\hat{\mathfrak H}^{^{\mathsf{P}}}(A):=\bigcup_{x\in\hat X}\hat{\mathcal H}^{^{\mathsf{P}}}_x(A)
$$
and
$$
\hat{\mathscr N}^{^{\mathsf{P}}}(\mathcal H(\hat{\mathscr B}_A)):=\bigcup_{x\in\hat X}\{\langle V,f\rangle^{^{\mathsf{P}}}:\,V\in\hat{\mathscr B}_A(x),\ f\in\overset{\scriptscriptstyle\,\,\mathsf{P}}{\mathcal H}(V\!\setminus\!\{x\})\},
$$
where $\langle V,f\rangle^{^{\mathsf{P}}}:=\{\langle f\rangle^{^{\mathsf P}}_x:x\in V\}$.
Then by Theorem~2.1$''$(3) (see Remark~\ref{remark0}), $\hat{\mathscr N}^{^{\mathsf{P}}}(\mathcal H(\hat{\mathscr B}_A))$ is a basis for a topology~$\hat{\mathscr T}^{^{\mathsf{P}}}_A(\mathcal H)$ on~$\hat{\mathfrak H}^{^{\mathsf{P}}}(A)$. Similarly to \cite[Lemma~(7.2)]{Fo}, by the definition of $\mathsf P$-porous $A$-continuation and Theorem~2.1$''$(3) we have

\begin{lemma}\label{lemma3.1}
Suppose $(X,\mathscr T;\hat X,\hat{\mathscr T})$ is a basic Riemann surface and $A$ is a partial point set $(\mathtt{I}\subseteq A\subseteq\hat X)$. Suppose $u:[0,1]\to\hat X$ is an $A$-curve with $a=u(0)$ and $b=u(1)$. Suppose $\mathsf P\subseteq\hat X$, $\varphi_{_{\scriptstyle a}}\in\widetilde{\mathcal H}_a$ and $\varphi_{_{\scriptstyle b}}\in\widetilde{\mathcal H}_b$. Then $\varphi_{_{\scriptstyle b}}$ is a $\mathsf P$-porous $A$-continuation of~$\varphi_{_{\scriptstyle a}}$ along~$u$ if and only if there exists a lifting $\hat u:[0,1]\to(\hat{\mathfrak H}^{^{\mathsf{P}}}(A),\hat{\mathscr T}^{^{\mathsf{P}}}_A(\mathcal H))$ of the $A$-curve~$u$ $($with respect to the projection~$\hat p_{_A})$ such that $\hat u(0)\thicksim\varphi_{_{\scriptstyle a}}$ and $\hat u(1)\thicksim\varphi_{_{\scriptstyle b}}$.
\hfill$\square$
\end{lemma}

By Theorems~2.1$'$ and~2.2$'$, Lemma~\ref{lemma3.1} and \cite[Theorem~(4.10)]{Fo} we obtain

\begin{theorem}[Monodromy Theorem]\label{theorem3.2}
Suppose $(X,\mathscr T;\hat X,\hat{\mathscr T})$ is a basic Riemann surface and $\hat{\mathscr T}(A)$ is a mixed topology of~$\hat X$ $(\mathtt{I}\subseteq A\subseteq\hat X)$. Let $a$, $b\in\hat X$. Suppose $u_0$ and~$u_1$ are homotopic $A$-curves from $a$ to~$b$ and $u_s$ $(0\leqslant\!s\!\leqslant1)$ is a deformation from $u_0$ to~$u_1$ in $(\hat X, \hat{\mathscr T}(A))$. Let $\varphi_{_{\scriptstyle a}}\in\widetilde{\mathcal H}_a$ and let $\varphi^{\scriptscriptstyle(0)}_{_{\scriptstyle b}}$, $\varphi^{\scriptscriptstyle(1)}_{_{\scriptstyle b}}\in\widetilde{\mathcal H}_b$ be the porous $A$-continuations of~$\varphi_{_{\scriptstyle a}}$ along~$u_0$ and~$u_1$, respectively. If $\varphi_{_{\scriptstyle a}}$ admits a porous $A$-continuation along every curve~$u_s$, then $\varphi^{\scriptscriptstyle(0)}_{_{\scriptstyle b}}\thicksim\varphi^{\scriptscriptstyle(1)}_{_{\scriptstyle b}}$.
\hfill$\square$
\end{theorem}

We consider a traditional Riemann surface $(X,\mathscr T)$ as a basic Riemann surface $(X,\mathscr T;X,\mathscr T)$ (i.e.\ the ideal point set~$\mathtt I=\emptyset$). Then Theorem~\ref{theorem3.2} implies

\begin{corollary1}\label{corollary3.2-1}
Suppose $(X,\mathscr T)$ is a $($traditional\,$)$ Riemann surface. Let $a$, $b\in X$. Suppose $u_0$, $u_1$ are homotopic curves from $a$ to~$b$ and $u_s$ $(0\leqslant\!s\!\leqslant1)$ is a deformation from $u_0$ to~$u_1$. Let $\varphi_{_{\scriptstyle a}}\in\overline{\mathcal H}_a:=\mathcal H_a\cup\overset{\scriptscriptstyle\,\,\circ}{\mathcal H}_a\cup\overset{\scriptscriptstyle\,\circ\mathrm{p}}{\mathcal H}_a$ and let $\varphi^{\scriptscriptstyle(0)}_{_{\scriptstyle b}}$, $\varphi^{\scriptscriptstyle(1)}_{_{\scriptstyle b}}\in\overline{\mathcal H}_b$ be the porous continuations of~$\varphi_{_{\scriptstyle a}}$ along~$u_0$ and~$u_1$, respectively. If $\varphi_{_{\scriptstyle a}}$ admits a porous continuation along every curve~$u_s$, then $\varphi^{\scriptscriptstyle(0)}_{_{\scriptstyle b}}\thicksim\varphi^{\scriptscriptstyle(1)}_{_{\scriptstyle b}}$.
\hfill$\square$
\end{corollary1}

For a traditional Riemann surface $(X,\mathscr T)$, suppose $\mathsf S$ is a quasi-discrete closed subset of $(X,\mathscr T)$. We call $F\in\mathcal H(X\!\setminus\!\mathsf S)$ an \textit{analytic function} in~$X$ with the \textit{singularities}~$\mathsf S$. If $a\in\mathsf S$ is a removable singularity of~$F$ then we also say $F$ is \textit{analytic} at~$a$. By Corollary~1 we have

\begin{corollary2}\label{corollary3.2-2}
Suppose $(X,\mathscr T)$ is a simply connected $($traditional\,$)$ Riemann surface. Suppose $\mathsf S$ is a quasi-discrete closed subset of~$X$. Let $a\in X$ and $\varphi_{_{\scriptstyle a}}\in\overline{\mathcal H}_a$ $(=\mathcal H_a\cup\overset{\scriptscriptstyle\,\,\circ}{\mathcal H}_a\cup\overset{\scriptscriptstyle\,\circ\mathrm{p}}{\mathcal H}_a)$. If $\varphi_{_{\scriptstyle a}}$ admits a porous continuation along every curve~$u$ starting at~$a$ to some $\varphi_{_{\scriptstyle b}}\in\overline{\mathcal H}_b$ $(b\in X$ is the end point of~$u)$ and $\varphi_{_{\scriptstyle b}}$ is equivalent to a usual germ for every $b\in X\!\setminus\!\mathsf S$, then there exists a unique analytic function~$F$ in~$X$ with the singularities~$\mathsf S$ such that $[F]^{^{\circ\mathrm{p}}}_a\thicksim\varphi_{_{\scriptstyle a}}$.
\end{corollary2}

\begin{proof}
Define $F(x):=\varphi_{_{\scriptstyle x}}(x)$ for $x\in X\!\setminus\!\mathsf S$, where $\varphi_{_{\scriptstyle x}}(x):=f(x)$ if $\varphi_{_{\scriptstyle x}}\thicksim[f]_x$ for some usual element~$f$ at~$x\in X$.
\end{proof}

\subsection{Harmonious equivalences and up-harmonious equivalences}

Let $X$ and~$Y$ be sets and let $\lambda:Y\to X$, $y\mapsto x$ be a surjection. We consider pairs $(X,x)$ and $(Y,y)$, where $x\in X$ (resp.\ $y\in Y$) is a variable which traverses all elements in~$X$ (resp.~$Y$).  We say that $(X,x)$ is \textit{up-harmoniously equivalent} to $(Y,y)$ \textit{modulo}~$\lambda$, and $\lambda$ is called an \textit{up-harmonious mapping}.  If $\lambda$ is bijective then $(X,x)$ and $(Y,y)$ are said to be \textit{harmoniously equivalent} (with one another) \textit{modulo}~$\lambda$, where $\lambda$ is called a \textit{harmonious mapping}. Later, we simply use the terminology \textit{$($up-$)$harmonious} to replace ``(up-)harmoniously equivalent".

Suppose $X$ and~$Y$ are topological spaces. If further the surjection $\lambda:Y\to X$ is continuous then we say that $(X,x)$ is \textit{continuously up-harmonious} with $(Y,y)$ \textit{modulo}~$\lambda$, where $\lambda$ is called a \textit{continuously up-harmonious mapping} and if $\lambda$ is a homeomorphism then $(X,x)$ and $(Y,y)$ are said to be \textit{continuously harmonious} (with one another) \textit{modulo}~$\lambda$, where $\lambda$ is called a \textit{continuously harmonious mapping}.

In the case that $X$ and~$Y$ are traditional Riemann surfaces and the surjection $\lambda:Y\to X$ is analytic, we say that $(X,x)$ is \textit{analytically up-harmonious} with $(Y,y)$ \textit{modulo}~$\lambda$, where $\lambda$ is called an \textit{analytically up-harmonious mapping}. If $\lambda$ is biholomorphic then $(X,x)$ and $(Y,y)$ are said to be \textit{analytically harmonious} (with one another) \textit{modulo}~$\lambda$, where $\lambda$ is called an \textit{analytically harmonious mapping}.

Now suppose $(X,\mathscr T;\hat X,\hat{\mathscr T})$ and $(Y,\mathscr T';\hat Y,\hat{\mathscr T}')$ are two basic Riemann surfaces and $\hat\lambda:\hat Y\to\hat X$ is a surjection satisfying $\hat\lambda(Y)\subseteq X$ and $(\hat\lambda(Y),\mathscr T|_{\hat\lambda(Y)};\hat X,\hat{\mathscr T})$ is a subsurface of $(X,\mathscr T;\hat X,\hat{\mathscr T})$. We say that $(X,\mathscr T;\hat X,\hat{\mathscr T};x)$ (simply denoted $(\hat X,x)$, $x\in\hat X$) is \textit{analytically up-harmonious} with $(Y,\mathscr T';\hat Y,\hat{\mathscr T}';y)$ ($y\in\hat Y$) \textit{modulo}~$\hat\lambda$, denoted $(\hat X,x)\overset{\scriptscriptstyle\,\hat\lambda}{\looparrowright}(\hat Y,y)$, if $(\hat\lambda(Y),x)$ is analytically up-harmonious with $(Y, y)$ modulo~$\hat\lambda|_Y$ and $((\hat X,\hat{\mathscr T}),x)$ is continuously up-harmonious with $((\hat Y,\hat{\mathscr T}'),y)$ modulo~$\hat\lambda$, where $\hat\lambda$ is called an \textit{analytically up-harmonious mapping}. Here we use the notation ``$(\hat X,x)\overset{\scriptscriptstyle\,\hat\lambda}{\looparrowright}(\hat Y,y)$" (or ``$(\hat Y,y)\overset{\scriptscriptstyle\!\hat\lambda}{\looparrowleft}(\hat X,x)$") rather than ``$(\hat Y,y)\overset{\scriptscriptstyle\!\hat\lambda}{\looparrowright}(\hat X,x)$" because we consider that $(\hat X,x)$ may be ``pasted into" $(\hat Y,y)$ by~$\hat\lambda$. Usually we assume $\hat\lambda|_Y:\,Y\to\hat\lambda(Y)$ is an unbranched covering. If moreover $\hat\lambda|_{Y}:Y\to X$ is biholomorphic and $\hat\lambda:\hat Y\to\hat X$ is homeomorphic, then we say that $(\hat X,x)$ and $(\hat Y,y)$ are \textit{analytically harmonious} (with one another) \textit{modulo}~$\hat\lambda$, denoted $(\hat X,x)\overset{\scriptscriptstyle\hat\lambda}{\leftrightarrow}(\hat Y,y)$, where $\hat\lambda$ is called an \textit{analytically harmonious mapping}.

Consider a family~$\mathfrak X$ of pairs $(\hat X,x)$, where $\hat X$ is a basic Riemann surface and $x\in\hat X$. Suppose $\hat\varLambda$ is a set consisting of some analytic up-harmonious mappings, which satisfies that $\mathrm{id}_{\hat X}\in\hat\varLambda$ for all pairs $(\hat X,x)\in\mathfrak X$ (here $\mathrm{id}_A$ denotes the identity mapping from set~$A$ to itself) and that if $(\hat X,x)\overset{\scriptscriptstyle\hat\lambda_1}{\looparrowright}(\hat Y,y)$, $(\hat Y,y)\overset{\scriptscriptstyle\hat\lambda_2}{\looparrowright}(\hat Z,z)$ and $\hat\lambda_1$, $\hat\lambda_2\in\hat\varLambda$ then $\hat\lambda_1{\circ}\hat\lambda_2\in\hat\varLambda$. Then we call $\hat\varLambda$ an \textit{analytically up-harmonious relation} in~$\mathfrak X$. If for two pairs $(\hat X,x)$, $(\hat Y,y)\in\mathfrak X$ there is $\hat\lambda\in\hat\varLambda$ such that $(\hat X,x)\overset{\scriptscriptstyle\hat\lambda}{\looparrowright}(\hat Y,y)$, then we say that $(\hat X,x)$ is \textit{analytically up-harmonious} with $(\hat Y,y)$ \textit{modulo}~$\hat\varLambda$, denoted $(\hat X,x)\overset{\scriptscriptstyle\hat\varLambda}{\looparrowright}(\hat Y,y)$. Let
\begin{eqnarray*}
\hat\varLambda_0&:=&\{\hat\lambda\in\hat\varLambda:\,\hat\lambda:\hat Y\to\hat X \text{ is homeomorphic, }\\
&&\qquad\qquad \hat\lambda|_Y\!:Y\to X \text{ is biholomorphic and } \hat\lambda^{-1}\in\hat\varLambda\}.
\end{eqnarray*}
If $(\hat X,x)\overset{\scriptscriptstyle\hat\lambda}{\leftrightarrow}(\hat Y,y)$ for some $\hat\lambda\in\hat\varLambda_0$ then we say that $(\hat X,x)$ and $(\hat Y,y)$ are \textit{analytically harmonious} (with one another) \textit{modulo}~$\hat\varLambda_0$, denoted $(\hat X,x)\overset{\scriptscriptstyle\hat\varLambda_0}{\leftrightarrow}(\hat Y,y)$, where $\hat\varLambda_0$ is called an \textit{analytically harmonious} (\textit{equivalence}) \textit{relation}. We may attach additional conditions to the analytically (up-)harmonious relation (in the next subsection we add the \textit{base-preserving} condition).

For $\mathbf{Y}\subseteq\mathfrak X$ we define the \textit{analytically up-harmonious class} (of~$\mathbf Y$)
$$
\widetilde{\mathbf{Y}}:=\{(\hat X,x)\in\mathfrak X:\,\text{there exists } (\hat Y,y)\in\mathbf Y \text { such that } (\hat X,x)\overset{\scriptscriptstyle\hat\varLambda}{\looparrowright}(\hat Y,y)\}.
$$
If there is an element $(\hat Y,y)\in\widetilde{\mathbf{Y}}$ such that for all $(\hat X,x)\in\widetilde{\mathbf{Y}}$ we have $(\hat X,x)\overset{\scriptscriptstyle\hat\varLambda}{\looparrowright}(\hat Y,y)$, then we call $(\hat Y,y)$ a \textit{holographic element} of~$\widetilde{\mathbf{Y}}$.

When we emphasize the surface~$\hat X$ in the pair $(\hat X,x)$, we write $\hat X(x)$ instead of $(\hat X,x)$. We can define similar notions to the above for sets, topological spaces and traditional Riemann surfaces, respectively.

\subsection{Algebraic Riemann surfaces}

Suppose $(X,\mathscr T;\hat X,\hat{\mathscr T})$ is a universal topological space and $\mathtt I$ is the set of all ideal points. Suppose there is a partition $\mathtt I=\bigcup_{j\in J}\mathtt I_j$ of~$\mathtt I$, where $\mathtt I_j=\{\hat x_{jk}:\, k=1,\,\dots,\, k_j\}$ ($k_j$ ($j\in J$) are positive integers), and a topological space $(\bar{X},\bar{\mathscr T})$ such that $\bar{X}=X\cup\bar{\mathtt I}$, where $X\cap\bar{\mathtt I}=\emptyset$, $\bar{\mathtt I}=\{\bar{x}_j:\,j\in J\}$, $\bar{\mathscr T}|_X=\mathscr T$ and for every neighborhood~$N(\bar{x}_j)$ of~$\bar{x}_j$ in $(\bar{X},\bar{\mathscr T})$ ($j\in J$) there exist punctured partial neighborhoods $\hat N^{\circ}(\hat x_{jk})$ of~$\hat x_{jk}$ such that $\hat N^{\circ}(\hat x_{jk})\subseteq N(\bar{x}_j)$ for $k=1$, $\dots$, $k_j$. Then we say that $(\bar{X},\bar{\mathscr T})$ is a \textit{tied space} corresponding to $(X,\mathscr T;\hat X,\hat{\mathscr T})$ and $(X,\mathscr T;\hat X,\hat{\mathscr T})$ is an \textit{untied space} corresponding to $(\bar{X},\bar{\mathscr T})$. In this case we may write $(X,\mathscr T;\hat X,\hat{\mathscr T};\bar{X},\bar{\mathscr T})$, or simply $(\hat X,\bar X)$. Denote $(\hat X,\bar X)$ by~$\dot X$, which we call a \textit{universal topological space with a tied space}.

Suppose $(X,\mathscr T;\hat X,\hat{\mathscr T})$ is a basic Riemann surface and a mapping $\hat f:(\hat X,\check{\mathscr T})\to\hat{\mathbb C}$ is continuous, where $\check{\mathscr T}$ is the essential topology of~$\hat X$ and $\hat{\mathbb C}=\mathbb C\cup\{\infty\}$ denotes the extended complex plane. If $\hat f|_X$ is meromorphic then we call $\hat f$ (\textit{essentially}) \textit{para-meromorphic} on~$\hat X$. If $\hat f(\hat X)\subseteq\mathbb C$ and $\hat f|_X$ is holomorphic then we call $\hat f$ (\textit{essentially}) \textit{para-holomorphic} on~$\hat X$. If there exists a traditional Riemann surface $(\bar{X},\bar{\mathscr T})$, which is a tied space corresponding to $\hat X$ (we always assume $\bar X$ has a uniform complex structure with~$X$), and a meromorphic (resp.\ holomorphic) function~$\bar f$ on~$\bar X$ such that $\hat f|_X=\bar f|_X$, then we call $\dot f:=(\hat f,\bar f)$ or~$\hat f$ (\textit{essentially}) \textit{meromorphic} (resp.\ (\textit{essentially}) \textit{holomorphic}) on $\dot X=(\hat X,\bar X)$ or~$\hat X$. Define $\dot f|_{\hat X}:=\hat f$ and $\dot f|_{\bar X}:=\bar f$. We denote the set of all para-meromorphic (resp.\ para-holomorphic) functions on~$\hat X$ by~$\breve{\mathcal M}(\hat X)$ (resp.\ $\breve{\mathcal H}(\hat X)$) and the set of all meromorphic (resp.\ holomorphic) functions on $\dot X$ by $\mathcal M(\dot X)$ (resp.\ $\mathcal H(\dot X)$). Denote $\mathcal M(\hat X):=\{\dot f|_{\hat X}:\dot f\in\mathcal M(\dot X)\}$ and $\mathcal H(\hat X):=\{\dot f|_{\hat X}:\dot f\in\mathcal H(\dot X)\}$. For $\dot f=(\hat f,\bar f)$, $\dot g=(\hat g,\bar g)\in\mathcal M(\dot X)$ we define $\dot f+\dot g:=(\hat f+\hat g,\bar f+\bar g)$ and $\dot f\cdot\dot  g:=(\hat f\cdot\hat g,\bar f\cdot\bar g)$.

Suppose $(X,\mathscr T;\hat X,\hat{\mathscr T})$ and $(Y,\mathscr T';\hat Y,\hat{\mathscr T}')$ are basic Riemann surfaces and a mapping $\hat f:(\hat Y,\check{\mathscr T}')\to(\hat X,\check{\mathscr T})$ is continuous, where $\check{\mathscr T}$ and~$\check{\mathscr T}'$ are the essential topologies of~$\hat X$ and~$\hat Y$ respectively. If $\hat f(Y)\subseteq X$ and $\hat f|_Y:Y\to X$ is holomorphic then we call $\hat f:\hat Y\to\hat X$ \textit{essentially para-holomorphic}. If $\hat f:\hat Y\to\hat X$ is essentially para-holomorphic and partially continuous, then we call $\hat f$ (\textit{exactly}) \textit{para-holomorphic}. A mapping $\hat f:\hat Y\to\hat X$ is called (\textit{exactly}) (resp.\ \textit{essentially}) \textit{para-biholomorphic} if it is bijective and both $\hat f:\hat Y\to\hat X$ and $\hat f^{-1}:\hat X\to\hat Y$ are (exactly) (resp.\ essentially) para-holomorphic.

Suppose $\hat f:\hat Y\to\hat X$ is an essentially para-holomorphic mapping. If there exist Riemann surfaces $(\bar{X},\bar{\mathscr T})$ and $(\bar{Y},\bar{\mathscr T}')$, which are tied spaces corresponding to $\hat X$ and~$\hat Y$ respectively, and a holomorphic mapping~$\bar f:\bar Y\to\bar X$ such that $\hat f|_Y=\bar f|_Y$, then we say that $\dot f=(\hat f,\bar f)$ is \textit{essentially holomorphic} from $\dot Y$ to $\dot X$ and $\hat f:\hat Y\to\hat X$ is \textit{essentially holomorphic}. Define $\dot f|_{\hat Y}:=\hat f$ and $\dot f|_{\bar Y}:=\bar f$. If $\hat f:\hat Y\to\hat X$ is essentially holomorphic and partially continuous, then we call $\hat f$ and~$\dot f$ (\textit{exactly}) \textit{holomorphic}. A mapping $\dot f=(\hat f,\bar f):\dot Y\to\dot X$ is called (\textit{exactly}) (resp.\ \textit{essentially}) \textit{biholomorphic} if it is bijective (i.e.\ both $\hat f$ and~$\bar f$ are bijective) and both $\dot f:\dot Y\to\dot X$ and $\dot f^{-1}:=(\hat f^{-1},\bar f^{-1}):\dot X\to\dot Y$ are (exactly) (resp.\ essentially) holomorphic. If $\dot f:\dot Y\to\dot X$ is (exactly) (resp.\ essentially) biholomorphic then we also say $\hat f:\hat Y\to\hat X$ is (\textit{exactly}) (resp.\ \textit{essentially}) \textit{biholomorphic}.

Denote the set of all (resp.\ essentially) para-holomorphic mappings from $\hat Y$ to~$\hat X$ by $\hat{\mathcal H}(\hat Y\!\to\!\hat X)$ (resp.\ $\breve{\mathcal H}(\hat Y\!\to\!\hat X)$). Denote the set of all (resp.\ essentially) holomorphic mappings from $\dot Y$ to $\dot X$ by $\mathcal H(\dot Y\!\to\!\dot X)$ (resp.\  $\check{\mathcal H}(\dot Y\!\to\!\dot X)$). Denote $\mathcal H(\hat Y\!\to\!\hat X):=\{\dot f|_{\hat Y}:\dot f\in\mathcal H(\dot Y\!\to\!\dot X)\}$ and $\check{\mathcal H}(\hat Y\!\to\!\hat X):=\{\dot f|_{\hat Y}:\dot f\in\check{\mathcal H}(\dot Y\!\to\!\dot X)\}$. It is easy to see that $\hat f\in\breve{\mathcal H}(\hat Y\!\to\!\hat X)$ and $\hat g\in\breve{\mathcal H}(\hat Z\!\to\!\hat Y)$ imply $\hat f\circ\hat g\in\breve{\mathcal H}(\hat Z\!\to\!\hat X)$, that $\hat f\in\hat{\mathcal H}(\hat Y\!\to\!\hat X)$ and $\hat g\in\hat{\mathcal H}(\hat Z\!\to\!\hat Y)$ imply $\hat f\circ\hat g\in\hat{\mathcal H}(\hat Z\!\to\!\hat X)$, that $\dot f\in\check{\mathcal H}(\dot Y\!\to\!\dot X)$ and $\dot g\in\check{\mathcal H}(\dot Z\!\to\!\dot Y)$ imply $\dot f\circ\dot g\in\check{\mathcal H}(\dot Z\!\to\!\dot X)$ and that $\dot f\in\mathcal H(\dot Y\!\to\!\dot X)$ and $\dot g\in\mathcal H(\dot Z\!\to\!\dot Y)$ imply $\dot f\circ\dot  g\in\mathcal H(\dot Z\!\to\!\dot X)$ ($\dot f\circ\dot  g:=(\hat f\circ\hat g,\bar f\circ\bar g)$). Specially, we know that $\hat f\in\breve{\mathcal M}(\hat X)$ (resp.\ $\breve{\mathcal H}(\hat X)$) and $\hat g\in\breve{\mathcal H}(\hat Y\!\to\!\hat X)$ imply $\hat f\circ\hat g\in\breve{\mathcal M}(\hat Y)$ (resp.\ $\breve{\mathcal H}(\hat Y)$) and that $\dot f\in\mathcal M(\dot X)$ (resp.\ ${\mathcal H}(\dot X)$) and $\dot g\in\check{\mathcal H}(\dot Y\!\to\!\dot X)$ imply $\dot f\circ\dot  g\in\mathcal M(\dot Y)$ (resp.\ ${\mathcal H}(\dot Y)$).

\smallskip
\begin{remark}\label{remark2}
$\breve{\mathcal H}(\hat X)$, ${\mathcal H}(\hat X)$ and ${\mathcal H}(\dot X)$ are rings; ${\mathcal M}(\hat X)$ and ${\mathcal M}(\dot X)$ are fields.
\end{remark}
\smallskip

Suppose $(X,\mathscr T;\hat X,\hat{\mathscr T})$ and $(Y,\mathscr T';\hat Y,\hat{\mathscr T}')$ are basic Riemann surfaces and $\hat f:\hat Y\to\hat X$ is an essentially para-holomorphic mapping. Let
$$
\hat f^*:\,\breve{\mathcal M}(\hat X)\longrightarrow\breve{\mathcal M}(\hat Y)
$$
be defined by
$$
\hat f^*(\hat\varphi):=\hat\varphi\circ\hat f
$$
for $\hat\varphi\in\breve{\mathcal M}(\hat X)$. Then
$$
\hat f^*|_{\breve{\mathcal H}(\hat X)}:\,\breve{\mathcal H}(\hat X)\longrightarrow\breve{\mathcal H}(\hat Y)
$$
is a ring homomorphism. If $\dot f\in\check{\mathcal H}(\dot Y\!\to\!\dot X)$ then
$$
\dot f^*:\,\mathcal M(\dot X)\longrightarrow\mathcal M(\dot Y),
$$
defined by $\dot f^*(\dot\varphi):=\dot\varphi\circ\dot  f$,
and
$$
\dot f^*|_{\mathcal H(\dot X)}:\,\mathcal H(\dot X)\longrightarrow\mathcal H(\dot Y)
$$
are ring homomorphisms.

Suppose $(R,\mathscr T_0;\hat R,\hat{\mathscr T}_0)$ is a basic Riemann surface with $\hat R=R$ $($equal as sets, i.e.\ the ideal point set~$\mathtt I=\emptyset$ and so $\bar R=R)$. Suppose $(X,\mathscr T;\hat X,\hat{\mathscr T})$ is a basic Riemann surface and $(\bar{X},\bar{\mathscr T})$ is a traditional Riemann surface which is a tied space corresponding to~$\hat X$. Suppose $\dot\lambda=(\hat\lambda,\bar\lambda):\dot X\to\dot R$ is a holomorphic mapping such that $(\hat R,r)\overset{\scriptscriptstyle\,\hat\lambda}{\looparrowright}(\hat X,x)$ ($\dot R=(\hat R,\bar R)$ and $\dot X=(\hat X,\bar X)$). Then
$$
\dot\lambda^*:\,\mathcal M(\dot R)\longrightarrow\mathcal M(\dot X)
$$
is a ring monomorphism. Define
$$
\dot h\cdot\dot g:=(\dot\lambda^*\dot h)\cdot\dot g
$$
for $\dot h\in{\mathcal M}(\dot R)$ and $\dot g\in\mathcal M(\dot X)$. Then $\mathcal M(\dot X)$ is a vector space over ${\mathcal M}(\dot R)$. We also consider $\dot h\in\mathcal M(\dot R)$ as $\dot\lambda^*\dot h=\dot h\circ\dot\lambda$ and then consider $\mathcal M(\dot R)$ as a subfield of the field $\mathcal M(\dot X)$. So $\mathcal M(R)$ is a subfield of~$\mathcal M(\bar X)$ and $\mathcal M(\hat X)$. Consequently and similarly, the punctured partial stalk~${\mathcal M}_{\hat X,\,x}$ of the sheaf~${\mathcal M}$ of meromorphic functions on~$\hat X$ at $x\in\hat X$ is also a vector space over ${\mathcal M}(R)$ by defining
$$
h\cdot\langle\hat g\rangle_x:=\langle(h\circ\hat\lambda)\cdot\hat g\rangle_x
$$
for $h\in{\mathcal M}(R)$ and $\langle\hat g\rangle_x\in\mathcal M_{\hat X,\,x}$, and we may also consider ${\mathcal M}(R)$ as a subfield of~${\mathcal M}_{\hat X,\,x}$.

Let $\hat f:\hat Y\to\hat X$ be para-holomorphic. Suppose for $y\in\hat Y$ with $\hat f(y)=x\in\hat X$ there exists a punctured partial neighborhood~$\hat N^{\circ}(y)$ of~$y$ such that $x\notin\hat f(\hat N^{\circ}(y))$. Then the mapping
$$
\hat f^*:\,\breve{\mathcal M}_{\hat X,\,x}\longrightarrow\breve{\mathcal M}_{\hat Y\!,\,y}\, ,
$$
defined by $\hat f^*(\langle\hat g\rangle_x):=\langle\hat g\circ\hat f\rangle_y$, is a ring monomorphism. If further for every partial neighborhood~$\hat N(y)$ of~$y$ which is open in $(\hat Y,\check{\mathscr T}')$ there exists a punctured partial neighborhood~$\hat N_1(x)$ of~$x$ which is open in $(\hat X,\check{\mathscr T})$ such that $\hat f|_{\hat N(y)}:\,\hat N(y)\to\hat N_1(x)$ is para-biholomorphic, then $\hat f^*$ is an isomorphism. Let
$$
\hat f_*:\,\breve{\mathcal M}_{\hat Y\!,\,y}\longrightarrow\breve{\mathcal M}_{\hat X,\,x}
$$
be the \textit{inverse} of~$\hat f^*$.

Suppose $\hat p\in\breve{\mathcal H}(\hat Y\!\to\!\hat X)$, $\bar p\in{\mathcal H}(\bar Y\!\to\!\bar X)$ (the set of all holomorphic mappings from $\bar Y$ to~$\bar X$), $\hat P(t)=\hat c_0t^n+\hat c_1t^{n-1}+\cdots+\hat c_n\in\breve{\mathcal M}(\hat X)[t]$ and $\bar P(t)=\bar c_0t^n+\bar c_1t^{n-1}+\cdots+\bar c_n\in\mathcal M(\bar X)[t]$, where $\mathcal M(\bar X)$ denotes the set of all meromorphic functions on~$\bar X$. Define
$$
(\hat p^*\hat P)(t):=(\hat p^*\hat c_0)t^n+(\hat p^*\hat c_1)t^{n-1}+\cdots+(\hat p^*\hat c_n),
$$
and
$$
(\bar p^*\bar P)(t):=(\bar p^*\bar c_0)t^n+(\bar p^*\bar c_1)t^{n-1}+\cdots+(\bar p^*\bar c_n),
$$
which are in $\breve{\mathcal M}(\hat Y)[t]$ and ${\mathcal M}(\bar Y)[t]$, respectively. Suppose $\dot p=(\hat p,\bar p)\in{\mathcal H}(\dot Y\!\to\!\dot X)$ and $\dot P(t)=\dot c_0t^n+\dot c_1t^{n-1}+\cdots+\dot c_n\in\mathcal M(\dot X)[t]$, where we may write $\dot P(t)=(\hat P(t),\bar P(t))$ and denote $\dot P|_{\hat Y}:=\hat P$ and $\dot P|_{\bar Y}:=\bar P$. Define
$$
(\dot p^*\dot P)(t):=(\dot p^*\dot c_0)t^n+(\dot p^*\dot c_1)t^{n-1}+\cdots+(\dot p^*\dot c_n),
$$
i.e.
$$
(\dot p^*\dot P)(t):=((\hat p^*\hat P)(t),(\bar p^*\bar P)(t)),
$$
which is in ${\mathcal M}(\dot Y)[t]$. If $\hat p:\hat Y\to\hat X$ is a (resp.\ an essentially) holomorphic ($n$-sheeted) covering map and $\bar p:\bar Y\to\bar X$ is a branched holomorphic ($n$-sheeted) covering map, then we say that $\dot p:\dot Y\to\dot X$ is a (or an \textit{exactly}) (resp.\ an  \textit{essentially}) \textit{holomorphic} (\textit{$n$-sheeted}) \textit{covering map}.

\smallskip
\begin{remark}\label{remark3}
Suppose $\dot f:=(\hat f,\bar f)$ is meromorphic on $\dot Y=(\hat Y,\bar Y)$. Then $(\hat p^*\hat P)(\hat f)=0$ if and only if $(\bar p^*\bar P)(\bar f)=0$. In this case, $(\dot p^*\dot P)(\dot f)=((\hat p^*\hat P)(\hat f),(\bar p^*\bar P)(\bar f))=0$.
\end{remark}

\begin{theorem}\label{theorem3.3}
Suppose $(R,\mathscr T_0;\hat R,\hat{\mathscr T}_0)$ is a basic Riemann surface determined by a perfect filterbase structure system~$\mathfrak B$ whose elements consist of simply connected domains with $\hat R=R$ and $P(t)\in{\mathcal M}(R)[t]$ is an irreducible monic polynomial of degree~$n$ $($here $\bar P(t)=\hat P(t)=P(t))$. Then there exists a basic Riemann surface $\dot S=(S,\mathscr T;\hat S,\hat{\mathscr T};\bar S,\bar{\mathscr T})$, a holomorphic $n$-sheeted covering map $\dot p:\dot S\to\dot R$ and a meromorphic function $\dot F\in\mathcal M(\dot S)$ such that $(\dot p^*\dot P)(\dot F)=0$. We call $\dot F$ a \textit{basic algebraic function} over~$\dot R$ $($or~$\hat R$ or~$R)$ with domain $\dot S=(\hat S,\bar S)$, denoted by $(\dot S,\dot p,\dot F)$. If $(\dot Z,\dot q,\dot G)$
has the corresponding properties,
then there exists exactly one fiber-preserving biholomorphic mapping $\dot\sigma:\dot Z\to\dot S$ $($i.e.\ $\dot p\circ\dot\sigma=\dot q)$ such that $\dot G=\dot\sigma^*\dot F$.
\end{theorem}

\begin{proof}
Let $\mathcal H$ be the sheaf of holomorphic functions on~$R$. Then by Theorems~\ref{theorem2.1} and~\ref{theorem2.2} we obtain a Hausdorff universal topological space $(\mathfrak H^{\circ},\mathscr T^{\circ}(\mathcal H);\hat{\mathfrak H},\hat{\mathscr T}(\mathcal H))$. Let
$$
\hat S:=\{\hat\varphi\in\hat{\mathfrak H}:\,P(\hat\varphi)=0\}
$$
and
$$
S:=\hat S\cap\mathfrak H^{\circ}.
$$
Let
$$
\mathscr T:=\mathscr T^{\circ}(\mathcal H)|_S:=\mathscr T^{\circ}(\mathcal H)\cap S
$$
and
$$
\hat{\mathscr T}:=\hat{\mathscr T}(\mathcal H)|_{\hat S}:=\hat{\mathscr T}(\mathcal H)\cap{\hat S}.
$$
Then $(S,\mathscr T;\hat S,\hat{\mathscr T})$ is a Hausdorff universal topological space determined by the perfect filterbase structure system
$$
\mathfrak N(\mathcal H(\mathscr B))|_S:=\{\mathcal N:\,\mathcal N\in\mathfrak N(\mathcal H(\mathscr B)) \text{ and every }N\in\mathcal N \text{ is a subset of~}S\}
$$
induced by $\mathfrak N(\mathcal H(\mathscr B))$ (see~(\ref{2.1})) on~$S$. Let
$$
\hat p:\hat S\to\hat R
$$
be the projection. Then evidently we see that $\hat p$ is an (exact) $n$-sheeted covering map. It is also evident to see that $\hat p\in\hat{\mathcal H}(\hat S\!\to\!\hat R)$. Define
$$
\hat F(\hat\varphi):=\hat\varphi(r):=\lim_{\gamma{\to}r \textrm{ in }{\scriptscriptstyle\hat{\mathscr T}}}\hat f(\gamma)
$$
for $\hat\varphi=\langle\hat f\rangle_r\in\hat S$. Then $\hat F\in\breve{\mathcal M}(\hat S)$ and $(\hat p^*\hat P)(\hat F)=0$.

By reasoning similar to \cite[Theorem~(8.9)]{Fo} we can obtain a traditional Riemann surface~$\bar S$, which is a tied space corresponding to~$\hat S$, and $\bar F\in\mathcal M(\bar S)$ with $(\bar p^*\bar P)(\bar F)=0$ and $\bar F|_S=\hat F|_S$, where the projection $\bar p:\bar S\to R$ is a branched holomorphic $n$-sheeted covering map with $\bar p|_S=\hat p|_S$. Hence $\dot F:=(\hat F,\bar F)\in\mathcal M(\dot S)$ and $\dot p:=(\hat p,\bar p)\in\mathcal H(\dot S\!\to\!\dot R)$, where $\dot S=(\hat S,\bar S)$.

For $z\in\hat Z$ let $\hat q(z)=r$ and $\hat\varphi:=\hat q_*\hat G_z$, where $\hat G=\dot G|_{\hat Z}$ and $\hat G_z$ denotes the partial germ of~$\hat G$ at~$z$. Then $\hat P(\hat\varphi)=0$. Hence $\hat\varphi\in\hat S$ and $\hat p(\hat\varphi)=r$. Define $\hat\sigma:\hat Z\to\hat Y$ by $\hat\sigma(z)=\hat\varphi$ ($z\in\hat Z$). Then $\hat\sigma$ is fiber-preserving and $\hat G=\hat\sigma^*\hat F$. Easily we see that $\hat\sigma$ is continuous. According to the reasoning in \cite[Theorem~(8.9)]{Fo}, $\hat\sigma|_Z$ can be extended to a fiber-preserving biholomorphic mapping $\bar\sigma:\bar Z\to\bar Y$ such that $\bar\sigma|_Z=\hat\sigma|_Z$ and $\bar G=\bar\sigma^*\bar F$. Let $\dot\sigma=(\hat\sigma,\bar\sigma)$. It is also easy to see that $\hat\sigma$ is bijective and $\hat\sigma^{-1}$ is continuous. Therefore, $\dot\sigma$ is biholomorphic. Obviously we have $\dot G=\dot\sigma^*\dot F$ and the mapping~$\dot\sigma$ is uniquely determined by this relation.
\end{proof}

We call $\dot Z=(\hat Z,\bar Z)$ (or~$\hat Z$) in Theorem~\ref{theorem3.3} a \textit{basic algebraic Riemann surface} over~$\hat R$ (or~$R$) and $\dot S=(\hat S,\bar S)$ (or~$\hat S$) in the proof of Theorem~\ref{theorem3.3} the \textit{original basic algebraic Riemann surface} over~$\hat R$ (or~$R$) (determined by $\hat P(t)$). We call $\dot G$ (or~$\hat G$) (resp.~$\dot F$ (or~$\hat F$)) in (resp.\ the proof of) Theorem~\ref{theorem3.3} a (resp.\ the \textit{original}) \textit{basic function} on~$\dot Z$ (or~$\hat Z$) (resp.\ on~$\dot S$ (or~$\hat S$)). We call $\hat R$ a \textit{base surface}. The holomorphic covering maps~$\dot q$ and~$\dot p$ in Theorem~\ref{theorem3.3} and its proof are called \textit{canonical} or \textit{natural} \textit{projections}.

Suppose $\dot Z_1$ and~$\dot Z_2$ are basic algebraic Riemann surfaces over a base surface~$\hat R$ and suppose $\dot p_1$ and~$\dot p_2$ are canonical projections from $\dot Z_1$ and~$\dot Z_2$ to $\dot R=(\hat R,\bar R)$, respectively. Then a mapping $\dot\lambda=(\hat\lambda,\bar\lambda):\dot Z_2\to\dot Z_1$ (resp.\ $\hat\lambda$, $\bar\lambda$) satisfying $\dot p_1\circ\dot\lambda=\dot p_2$ (resp.\ $\hat p_1\circ\hat\lambda=\hat p_2$, \,$\bar p_1\circ\bar\lambda=\bar p_2$) is said to be \textit{base-preserving}.

\smallskip
\begin{remark}\label{remark4}
The topological space $(\hat Z,\check{\mathscr T}')$ in Theorem~\ref{theorem3.3}, where $\check{\mathscr T}'$ is the essential topology of~$\hat Z$, is connected and path-connected. Therefore, by Lemma~\ref{lemma3.1} we see that any two punctured partial germs in the original basic algebraic Riemann surface~$\hat S$ are analytic continuations along some curve in~$R$ from one another.
\end{remark}
\smallskip

Suppose $(R,\mathscr T_0;\hat R,\hat{\mathscr T}_0)$ is a base surface with $\hat R=R$. Let $\hat\varphi\in\hat{\mathfrak H}$ be a punctured partial germ satisfying $P(\hat\varphi)=0$, where $P(t)\in\mathcal M(R)[t]$ is an irreducible monic polynomial of degree~$n$ (called the \textit{minimal polynomial} of~$\hat\varphi$ in $\mathcal M(R)$ or in~$\hat R$). Here $\hat\varphi$ is called an \textit{algebraic punctured partial germ} of degree~$n$ on~$\hat R$. We call basic algebraic Riemann surfaces and the original basic algebraic Riemann surface determined by $P(t)$ \textit{basic algebraic Riemann surfaces} and the \textit{original basic algebraic Riemann surface} determined by $\hat\varphi$, respectively.


By \cite[Theorem~(8.12)]{Fo} (refer to its proof) we have

\begin{lemma}\label{lemma3.4}
Suppose $\dot S=(\hat S,\bar S)$ is an original basic algebraic Riemann surface over a base surface~$\hat R$ and $\dot F$ is the original basic function on~$\dot S$. If $\dot f$ is a meromorphic function on~$\dot S$, then there exists a polynomial $\dot Q(t)\in\mathcal M(\dot R)[t]$ $(\dot R=(\hat R,\bar R))$ such that $\dot f=(\dot p^*\dot Q)(\dot F)$, where $\dot p:\dot S\to\dot R$ is the canonical projection, i.e.\ $\dot f=\dot Q(\dot F)$.
\hfill$\square$
\end{lemma}

Let $\mathcal A$ (resp.~$\mathcal A^0$) denote the set of all (resp.\ original) basic algebraic Riemann surfaces over~$\hat R$. We consider a pair $(\dot Z,\zeta)$, where $\dot Z=(\hat Z,\bar Z)\in\mathcal A$ and $\zeta$ is a variable in~$\hat Z$. Let
$$
\mathcal A(\zeta):=\{(\dot Z,\zeta):\,\dot Z=(\hat Z,\bar Z)\in\mathcal A \text{ and } \zeta \text{ is a variable in } \hat Z\}
$$
and
$$
\mathcal A^0(\xi):=\{(\dot S,\xi):\,\dot S=(\hat S,\bar S)\in\mathcal A^0 \text{ and } \xi \text{ is a variable in } \hat S\}.
$$

Let $(\dot S_1,\xi_1)$, $(\dot S_2,\xi_2)\in\mathcal A^0(\xi)$. We say that $(\dot S_1,\xi_1)$ is \textit{directly up-harmonious} with $(\dot S_2,\xi_2)$ if there is a polynomial $Q(t)\in\mathcal M(R)[t]$ such that $\xi_1=Q(\xi_2)$ and $\hat S_1=Q(\hat S_2)$, denoted $(\dot S_1,\xi_1)\overset{\scriptscriptstyle\,Q}{\looparrowright}(\dot S_2,\xi_2)$. If $(\dot S_1,\xi_1)$ and $(\dot S_2,\xi_2)$ are directly up-harmonious with one another then we say that they are \textit{directly harmonious} (with one another), denoted $(\dot S_1,\xi_1){\leftrightarrow}(\dot S_2,\xi_2)$. Let  $(\dot Z_1,\zeta_1)$, $(\dot Z_2,\zeta_2)\in\mathcal A(\zeta)$. If there exists a holomorphic base-preserving surjection $\dot\lambda:\dot Z_2\to\dot Z_1$ with $\hat\lambda(\zeta_2)=\zeta_1$ then we say that $(\dot Z_1,\zeta_1)$ is \textit{analytically up-harmonious} with $(\dot Z_2,\zeta_2)$ modulo~$\dot\lambda$ (or~$\hat\lambda$), denoted $(\dot Z_1,\zeta_1)\overset{\scriptscriptstyle\,\dot\lambda}{\looparrowright}(\dot Z_2,\zeta_2)$, where $\dot\lambda$ and~$\hat\lambda$ are called \textit{analytically up-harmonious mappings}. If further $\dot\lambda$ is biholomorphic, then we say that $(\dot Z_1,\zeta_1)$ and $(\dot Z_2,\zeta_2)$ are \textit{analytically harmonious} (with one another) modulo~$\dot\lambda$ (or~$\hat\lambda$), denoted $(\dot Z_1,\zeta_1)\overset{\scriptscriptstyle\dot\lambda}{\leftrightarrow}(\dot Z_2,\zeta_2)$, where $\dot\lambda$ and~$\hat\lambda$ are called \textit{analytically harmonious mappings}.

\begin{proposition}\label{proposition3.5}
Let $(\dot S_1,\xi_1)$, $(\dot S_2,\xi_2)\in\mathcal A^0(\xi)$. Then $(\dot S_1,\xi_1)$ is analytically up-harmonious with $(\dot S_2,\xi_2)$ modulo~$\dot\lambda$ if and only if $(\dot S_1,\xi_1)$ is directly up-harmonious with $(\dot S_2,\xi_2)$. In this case, $\dot\lambda:\dot S_2\to\dot S_1$ is a holomorphic covering map $($i.e.\ $\hat\lambda:\hat S_2\to\hat S_1$ is an exact covering map and $\bar\lambda:\bar S_2\to\bar S_1$ is a holomorphic branched covering map$)$.
\end{proposition}

\begin{proof}
Suppose $\dot\lambda=(\hat\lambda,\bar\lambda)$ is the analytically up-harmonious mapping from $\dot S_2=(\hat S_2,\bar S_2)$ to~$\dot S_1=(\hat S_1,\bar S_1)$. Suppose $\xi_1\in\hat S_1$, $\xi_2\in\hat S_2$ and $\xi_1=\hat\lambda(\xi_2)$. Then there exists a punctured partial neighborhood~$V$ of $r=\hat p_1(\xi_1)$ in~$\hat R$ (if necessary we will shrink~$V$), where $\dot p_1=(\hat p_1,\bar p_1)$ is the canonical projection from $\dot S_1$ to~$\dot R$, and $f_1$, $f_2\in\mathcal H(V)$ such that $\xi_1=\langle f_1\rangle_r$ and $\xi_2=\langle f_2\rangle_r$. Suppose $\dot F_1=(\hat F_1,\bar F_1)$ and~$\dot F_2=(\hat F_2,\bar F_2)$ are the original basic functions on~$\dot S_1$ and~$\dot S_2$, respectively. Then by Lemma~\ref{lemma3.4}, there exists a polynomial $Q(t)\in\mathcal M(R)[t]$ such that $\bar F_1\circ\bar\lambda=((\bar p_1\circ\bar\lambda)^*Q)(\bar F_2)$. Since $\hat\lambda$ is base-preserving and partially continuous, we have $\bar\lambda(\langle f_2\rangle_{r'})=\langle f_1\rangle_{r'}$ for $r'\in V$. Therefore,  it follows that
$$
f_1(r')=\bar F_1(\langle f_1\rangle_{r'})=(\bar F_1\circ\bar\lambda)(\langle f_2\rangle_{r'})=(((\bar p_1\circ\bar\lambda)^*Q)(\bar F_2))(\langle f_2\rangle_{r'})=(Q(f_2))(r').
$$
Then we get
$$
\xi_1=\langle f_1\rangle_r=\langle Q(f_2)\rangle_r=Q(\xi_2).
$$

Suppose there is a polynomial $Q(t)\in\mathcal M(\hat R)[t]$ such that $\xi_1=Q(\xi_2)$, where $\xi_1$ and~$\xi_2$ travel around the whole $\hat S_1$ and~$\hat S_2$, respectively and correspondingly. Suppose $\dot\lambda:\dot S_2\to\dot S_1$ is defined by $\hat\lambda(\xi_2):=Q(\xi_2)$ for $\xi_2\in\hat S_2$, $\bar\lambda(\xi_2):=Q(\xi_2)$ for $\xi_2\in S_2=\bar S_2\cap\hat S_2$ and for $\xi_2\in\bar S_2\!\setminus\! S_2$ we continuously continue~$\bar\lambda$ on~$\bar S_2$. Then $\hat\lambda:\hat S_2\to\hat S_1$ is an exact covering map and $\bar\lambda:\bar S_2\to\bar S_1$ is a proper holomorphic map (cf.\ \cite[Theorems~(8.4) and~(8.9)]{Fo}).
\end{proof}

From Proposition~\ref{proposition3.5} it follows
\begin{corollary}
Let $(\dot S_1,\xi_1)$, $(\dot S_2,\xi_2)\in\mathcal A^0(\xi)$. Then $(\dot S_1,\xi_1)$ and $(\dot S_2,\xi_2)$ are analytically harmonious if and only if they are directly harmonious.
\hfill$\square$
\end{corollary}

Suppose $\hat R$ is a base surface. Fix a point~$r_0$ in~$\hat R$, which we call a \textit{base point} in~$\hat R$. Let $\hat\varphi\in\hat{\mathcal H}_{r_0}$ be an algebraic punctured partial germ at~$r_0$ on~$\hat R$ and let $\dot S$ be the original basic algebraic Riemann surface determined by~$\hat\varphi$. In order to make the difference, we put a ``label"~$\hat\varphi$ on~$\dot S$ and call $(\hat\varphi;\dot S)$ an \textit{original basic algebraic Riemann surface with label} (or \textit{tag})~$\hat\varphi$, denoted by~$\ddot S$. We say $\hat\varphi$ is the (resp.\ a) \textit{natural label} of~$\ddot S$ (resp.~$\dot S$). The original basic function~$\dot F$ (or~$\hat F$) on~$\dot S$ is also considered as the \textit{original basic function} on~$\ddot S$. Let $\ddot S_1=(\hat\varphi_1;\dot S_1)$ and $\ddot S_2=(\hat\varphi_2;\dot S_2)$ be original basic algebraic Riemann surfaces with natural labels. Then we consider that $(\ddot S_1,\xi_1))\overset{\scriptscriptstyle\,\hat\lambda}{\looparrowright}(\ddot S_2,\xi_2)$ precisely if $\hat\varphi_1=\hat\lambda(\hat\varphi_2)$ (refer to Lemma~\ref{lemma3.6} below).

\smallskip
\begin{remark}\label{remark5}
By Lemma~\ref{lemma3.1} (refer to Remark~\ref{remark4}) we can continue the label~$\hat\varphi$ along curves in $(\hat R,\mathscr T(A))$, where $A$ is the set of branch points of the minimal polynomial of~$\hat\varphi$ in $\mathcal M(\hat R)[t]$, to get~$\hat S$.
\end{remark}
\smallskip

In order to give a label to $\dot Z\in\mathcal A$, we now introduce a punctured partial germ in a system of ``equivalent presheaves". Suppose $\mathfrak X$ is a family of basic Riemann surfaces that are analytically harmonious with one another and $\hat\Lambda_0$ is the analytically harmonious relation. Suppose $(X,\mathscr T;\hat X,\hat{\mathscr T})\in\mathfrak X$ and $Y$ is a traditional Riemann surface. Denote by $\mathcal H(U\!\to\!Y)$ the set of all holomorphic mappings from $U$ to~$Y$, where $U$ is an open set in $(X,\mathscr T)$. Let $\mathcal H_{X,Y}=(\mathcal H(U\!\to\!Y))_{U\in\mathscr T}$ be the family consisting of all holomorphic mappings from $U$ to~$Y$ for all $U\in\mathscr T$. It is a sheaf of sets from $X$ to~$Y$. Denote by $\mathcal H_{\mathfrak X,Y}$ the system of all sheaves $\mathcal H_{X,Y}$ for $(X,\mathscr T;\hat X,\hat{\mathscr T})\in\mathfrak X$.

Let $a_0\in\hat X_0$, where $\hat X_0\in\mathfrak X$. Denote
$$
\tilde{a}:=\{\hat\lambda(a_0):\, \hat\lambda\in\hat\Lambda_0 \text{ is an analytically harmonious mapping from } \hat X_0 \text{ to } \hat X\in\mathfrak X\}.
$$
Let
$$
\mathcal H_{\mathfrak X(\tilde a),Y}:=\biguplus_{\hat\lambda\in\hat\Lambda_0}\,\biguplus_{U\in\hat{\mathscr T}^{\circ}(\hat\lambda(a_0))} \mathcal H(U\!\to\!Y),
$$
which is a disjoint union (we may also use $\mathcal H(\tilde{\hat{\mathscr T}}^{\circ}(\tilde a)\!\to\!Y)$ to represent $\mathcal H_{\mathfrak X(\tilde a),Y}$). In $\mathcal H_{\mathfrak X(\tilde a),Y}$, two elements (mappings) $f_1\in \mathcal H(U_1\!\to\!Y)$ and $f_2\in \mathcal H(U_2\!\to\!Y)$ ( $U_1\in\hat{\mathscr T}_1^{\circ}(\hat\lambda_1(a_0))$ and $U_2\in\hat{\mathscr T}_2^{\circ}(\hat\lambda_2(a_0))$, $\hat\lambda_1$, $\hat\lambda_2\in\hat\Lambda_0$) are said to be \textit{equivalent}, denoted $f_1\,\underset{\tilde a}{\hat\thicksim}\,f_2$, if there exists $U\in\hat{\mathscr T}^{\circ}(\hat\lambda(a_0))$ ($\hat\lambda\in\hat\Lambda_0$) with $\hat\lambda^{-1}(U)\subseteq\hat\lambda_1^{-1}(U_1)\cap\hat\lambda_2^{-1}(U_2)$ such that $f_1\circ\hat\lambda_1|_{\hat\lambda^{-1}(U)}=f_2\circ\hat\lambda_2|_{\hat\lambda^{-1}(U)}$. It is easy to see that this really is an equivalence relation. Denote
$$
\mathcal H_{\tilde a}:=\mathcal H_{\mathfrak X(\tilde a),Y}\Big/\underset{\tilde a}{\hat\thicksim}\, ,
$$
which is the set of all equivalence classes and is called the \textit{punctured partial stalk} of the sheaf system $\mathcal H_{\mathfrak X,Y}$ at~$\tilde a$ or~$\hat\lambda(a_0)$. Suppose $f\in\mathcal H(V\!\to\!Y)$, where $V\supseteq U$, $V\in\mathscr T$ and $U\in\hat{\mathscr T}^{\circ}(\hat\lambda(a_0))$. The equivalence class of  $f|_U\in \mathcal H(U\!\to\!Y)$ modulo~$\underset{\tilde a}{\hat\thicksim}$ is called the \textit{punctured partial germ} of~$f$ at~$\tilde a$ or~$\hat\lambda(a_0)$, denoted~$\langle f\rangle_{\tilde a}$ or~$\langle f\rangle_{\hat\lambda(a_0)}$.

Suppose $\mathbf A$ is the set consisting of all algebraic punctured partial germs at the base point~$r_0$ on~$\hat R$. Then $\mathbf A$ is a field. Let $\ddot{\mathcal A}^0$ denote the set of all original basic algebraic Riemann surfaces with labels over~$\hat R$ determined by germs in~$\mathbf A$. Given a subset~$\mathtt S$ of~$\mathbf A$, let $\mathbf S$ be the subfield $\mathcal M_{r_0}(\mathtt S)$ of~$\mathbf A$ generated by~$\mathtt S$ and $\mathcal M_{r_0}=\{\langle f\rangle_{r_0}:f\in\mathcal M(\hat R)\}$, and then let $\ddot{\mathcal S}^0$ be the set of all original basic algebraic Riemann surfaces with labels determined by germs in~$\mathbf S$. Let $\tilde\Lambda$ and $\tilde\Lambda^0$ denote the \textit{analytically harmonious relation} and the \textit{directly harmonious relation}, respectively. Suppose
$$
\ddot{\mathcal S}^0=\biguplus_{j\in J} L^0_j
$$
is the partition of~$\ddot{\mathcal S}^0$ by~$\tilde\Lambda^0$. We call $L^0_j$ ($j\in J$) \textit{original level surfaces} or \textit{original levels}. Suppose $\tilde\Lambda^0_j$ is the directly harmonious relation in~$L^0_j$ ($j\in J$). Suppose $\dot Z=(\hat Z,\bar Z)$ is a basic algebraic Riemann surface that is analytically harmonious with $\ddot S\in L^0_j$ and $\dot\sigma=(\hat\sigma,\bar\sigma)$ is the analytically harmonious mapping from $\ddot S=(\hat\varphi;\dot S)$ to~$\dot Z$. Then $\bar\sigma\in\mathcal H(\bar S\!\to\!\bar Z)$. Let $\mathfrak X=L^0_j$ and let $\langle\dot\sigma\rangle_{\hat\varphi}$ denote the punctured partial germ~$\langle\bar\sigma\rangle_{\hat\varphi}$. We now put label~$\langle\dot\sigma\rangle_{\hat\varphi}$ on~$\dot Z$ and call $(\langle\dot\sigma\rangle_{\hat\varphi};\dot Z)$ a \textit{basic algebraic Riemann surface with label} (or \textit{tag})~$\langle\dot\sigma\rangle_{\hat\varphi}$, denoted by~$\ddot Z$. We say $\langle\dot\sigma\rangle_{\hat\varphi}$ is the (resp.\ a) \textit{given label} of~$\ddot Z$ (resp.~$\dot Z$). It is worth noting that if $\ddot Z$ is an original basic algebraic Riemann surface with natural label~$\hat\varphi$ then the given label of~$\ddot Z$ is just $\langle\mathrm{id}\rangle_{\hat\varphi}$, which is uniform with its natural label~$\hat\varphi$.

Let $\ddot{\mathcal A}$ denote the set of all basic algebraic Riemann surfaces with labels over~$\hat R$. Let $\dot\varLambda=(\hat\varLambda,\bar\varLambda)$ denote the \textit{analytically up-harmonious relation} in~$\ddot{\mathcal A}$ determined by the direct up-harmonious relation~$\dot\varLambda^0=(\hat\varLambda^0,\bar\varLambda^0)$ in~$\ddot{\mathcal A}^0$, which is defined as follows: For $\ddot Z_1$, $\ddot Z_2\in\ddot{\mathcal A}$, $\ddot Z_1(\zeta_1)\overset{\scriptscriptstyle\,\dot\lambda}{\looparrowright}\ddot Z_2(\zeta_2)$ ($\dot\lambda\in\dot\varLambda$) if and only if there exist  $\ddot S_1$, $\ddot S_2\in\ddot{\mathcal A}^0$ such that $\displaystyle\ddot Z_1(\zeta_1)\overset{\scriptscriptstyle\dot\lambda_1}{\leftrightarrow}\ddot S_1(\xi_1)$, $\displaystyle\ddot Z_2(\zeta_2)\overset{\scriptscriptstyle\dot\lambda_2}{\leftrightarrow}\ddot S_2(\xi_2)$, $\displaystyle\ddot S_1(\xi_1)\overset{\scriptscriptstyle\,\dot\lambda_0}{\looparrowright}\ddot S_2(\xi_2)$ ($\dot\lambda_1$, $\dot\lambda_2\in\tilde\Lambda$, $\dot\lambda_0\in\dot\varLambda^0$ and $\ddot Z(\zeta)$ denotes $(\ddot Z,\zeta)$)\, and $\displaystyle\dot\lambda=\dot\lambda_1\circ\dot\lambda_0\circ\dot\lambda_2^{-1}$, where $\dot\lambda$ is called an \textit{analytically up-harmonious mapping}. We may also denote $\ddot Z_1(\zeta_1)\overset{\scriptscriptstyle\,\dot\lambda}{\looparrowright}\ddot Z_2(\zeta_2)$ simply by $\ddot Z_1\overset{\scriptscriptstyle\,\dot\lambda}{\looparrowright}\ddot Z_2$ or $\ddot Z_1\looparrowright\ddot Z_2$. Meanwhile, $\ddot Z_1\overset{\scriptscriptstyle\,\dot\lambda}{\looparrowright}\ddot Z_2$ also means that the label~$\langle\dot\lambda_2\rangle_{\hat\varphi_2}$ of~$\ddot Z_2$ is mapped to the label~$\langle\dot\lambda_1\rangle_{\hat\varphi_1}$ of~$\ddot Z_1$ by~$\hat\lambda$ ($\dot\lambda=(\hat\lambda,\bar\lambda)$) ($\hat\varphi_1$ and~$\hat\varphi_2$ are the natural labels of $\ddot S_1$ and~$\ddot S_2$ respectively), where we define
$
\hat\lambda(\langle\dot\mu_2\rangle_{\hat\varphi_2}):=\langle\dot\mu_1\rangle_{\hat\varphi_1}
$
for holomorphic mappings $\dot\mu_j:\ddot S_j\to\ddot Z_j$ ($j=1$, $2$), $\displaystyle\ddot S_1\overset{\scriptscriptstyle\,\dot\lambda_0}{\looparrowright}\ddot S_2$ ($\ddot S_1$ and~$\ddot S_2$ are original algebraic Riemann surfaces with the natural labels) and $\displaystyle\dot\mu_1\circ\dot\lambda_0=\dot\lambda\circ\dot\mu_2$, which can easily be shown to be well defined. By Proposition~\ref{proposition3.5} we know that $\dot\varLambda$ is really an analytically up-harmonious relation. If $\ddot Z_1\overset{\scriptscriptstyle\,\dot\lambda}{\looparrowright}\ddot Z_2$, then we say $\ddot Z_2$ is \textit{over}~$\ddot Z_1$ or $\ddot Z_1$ is \textit{under}~$\ddot Z_2$, also denoted by $\ddot Z_2\geqslant\ddot Z_1$ or $\ddot Z_1\leqslant\ddot Z_2$. Here if $\dot\lambda:\ddot Z_2\to\ddot Z_1$ is not biholomorphic, then we say $\ddot Z_2$ is \textit{strictly over}~$\ddot Z_1$ or $\ddot Z_1$ is \textit{strictly under}~$\ddot Z_2$. If $\dot\lambda:\ddot Z_2\to\ddot Z_1$ is biholomorphic, then we say $\ddot Z_2$ is \textit{equivalent} to~$\ddot Z_1$ modulo~$\dot\lambda$ (or~$\hat\lambda$), denoted $\displaystyle\ddot Z_1\overset{\scriptscriptstyle\dot\lambda}{\leftrightarrow}\ddot Z_2$ (i.e.\ $\ddot Z_2(\zeta_2)$ is harmonious with~$\ddot Z_1(\zeta_1)$).

Suppose
$$
L_j:=\{\ddot Z:\,\ddot Z\in\ddot{\mathcal A} \text{ is analytically harmonious with some } \ddot S\in L^0_j\} \quad (j\in J),
$$
which are called \textit{level surfaces} or \textit{levels}. Let
$$
\tilde{Z}:=\biguplus_{j\in J} L_j.
$$
Then $\tilde{Z}$ is the analytically up-harmonious class of~$\ddot{\mathcal S}^0$ in~$\ddot{\mathcal A}$. We call $\tilde{Z}$ the \textit{algebraic Riemann surface} (over~$\hat R$) determined by~$\mathtt S$ and $\ddot{\mathcal S}^0$ the \textit{original algebraic Riemann surface} corresponding to~$\tilde Z$, denoted by~$\tilde Z^0$. We call $\mathbf S$ in the above the \textit{natural label set} or the \textit{natural label field} of~$\tilde Z$ or~$\tilde Z^0$, 
denoted by~$\mathsf L(\tilde Z)$ or~$\mathsf L(\tilde Z^0)$.

\smallskip
\begin{remark}\label{remark6}
If $\mathtt S=\{\hat\varphi\}$, where $\hat\varphi$ is an algebraic punctured partial germ at the base point~$r_0$, then instead of a $($traditional$)$ Riemann surface $($determined by~$\hat\varphi)$ we consider the analytically up-harmonious class~$\tilde{Z}$ determined by~$\mathtt S$, which has a geographic element the basic Riemann surface that is determined by~$\hat\varphi$, as our $($algebraic$)$ Riemann surface, which we call the (\textit{algebraic}) \textit{Riemann surface} determined by~$\hat\varphi$.
\end{remark}

\begin{remark}\label{remark7}
Generally, the above definition means that we consider an algebraic Riemann surface as a system of basic algebraic Riemann surfaces organized by the analytically up-harmonious relation with the aid of labels.
\end{remark}

\begin{remark}\label{remark10-1}
Suppose $\tilde Z$ is an algebraic Riemann surfaces over a base surface~$\hat R$ with base point~$r_0$. Then we may consider that all~$\hat Z$, for $\dot Z=(\hat Z,\bar Z)$ and $\ddot Z=(\langle\dot\sigma\rangle_{\hat\varphi};\dot Z)\in\tilde Z$, form a ``coordinate system" in~$\tilde Z$ and all~$\bar Z$ together show the topological and complex structure of~$\tilde Z$, where the base point~$r_0$ may be regarded as an ``origin of coordinates".
\end{remark}
\smallskip

Noticing Remark~\ref{remark4} or Remark~\ref{remark5} we can deduce

\begin{lemma}\label{lemma3.6}
Suppose $\hat\varphi$, $\hat\varphi_1\in\mathbf A$ and $\hat\varphi=\hat P(\hat\varphi_1)$ for some $\hat P(t)\in\mathcal M(\hat R)[t]$. Suppose $\ddot S$ and~$\ddot S_1$ are two original basic algebraic Riemann surfaces over the base surface~$\hat R$ determined by~$\hat\varphi$ and~$\hat\varphi_1$ with labels~$\hat\varphi$ and~$\hat\varphi_1$, respectively. Then there is an up-harmonious mapping $\dot p:\ddot S_1\to\ddot S$, which is defined by $\hat p(\hat\psi_1)=\hat P(\hat\psi_1)$ for $\hat\psi_1\in\hat S_1$ and $\bar p(\hat\psi_1)=\hat p(\hat\psi_1)$ for $\hat\psi_1\in S_1=\hat S_1\cap\bar S_1$ $(\dot S_1=(\hat S_1,\bar S_1))$, such that $\dot F\circ\dot p=\dot P(\dot F_1)$, where $\dot F$ and~$\dot F_1$ are the original basic functions on~$\ddot S$ and~$\ddot S_1$ respectively and $\dot P=(\hat P,\bar P)$ $(\bar P=\hat P)$.
\hfill$\square$
\end{lemma}

The (directly) up-harmonious mapping $\dot p:\ddot S_1\to\ddot S$ in Lemma~\ref{lemma3.6} is said to be \textit{corresponding to} $\hat\varphi=\hat P(\hat\varphi_1)$ or \textit{determined by}~$\hat P(t)$. By Proposition~\ref{proposition3.5} we see that this $\dot p$ is a holomorphic covering map. We can also see that generally an analytically up-harmonious mapping $\dot\lambda=(\hat\lambda,\bar\lambda):\ddot Z_1\to\ddot Z$ (or $\dot Z_1\to\dot Z$) is a holomorphic covering map, which means that $\hat\lambda:\hat Z_1\to\hat Z$ is a covering map and $\bar\lambda:\bar Z_1\to\bar Z$ is a branched holomorphic covering map.

Suppose $\hat X$ and~$\hat Y$ are connected universal topological spaces (i.e.\ $(\hat X,\check{\mathscr T})$ and $(\hat Y,\check{\mathscr T}')$ are connected) and $\hat p:\hat Y\to\hat X$ is a covering map. The covering is called \textit{Galois} if for every pair of points $y_1$, $y_2\in\hat Y$ with $\hat p(y_1)=\hat p(y_2)$ there exists a deck transformation $\hat\sigma:\hat Y\to\hat Y$ such that $\hat\sigma(y_1)=y_2$.
Suppose $\dot X=(\hat X,\bar X)$ and $\dot Y=(\hat Y,\bar Y)$ are universal topological spaces with tied spaces and $\dot p:\dot Y\to\dot X$ is a covering map, which means $\hat p:\hat Y\to\hat X$ is a covering map, where $\dot p=(\hat p,\bar p)$. The covering $\dot p:\dot Y\to\dot X$ is called \textit{Galois} if $\hat p:\hat Y\to\hat X$ is Galois.

Suppose $\hat R$ is a base surface with base point $r_0\in\hat R$. Suppose $\tilde Y$ is an algebraic Riemann surface over~$\hat R$ and $\dot\varLambda$ is the analytically up-harmonious relation in~$\tilde Y$. For $\ddot Y\in\tilde Y$ we denote mapping $\dot f:\ddot Y\to(\hat{\mathbb C},\hat{\mathbb C})$ (i.e.\ $\hat f:\hat Y\to\hat{\mathbb C}$ and $\bar f:\bar Y\to\hat{\mathbb C}$ are mappings) by $\dot f:\ddot Y\to\hat{\mathbb C}$, called a (\textit{complex}) \textit{function} on~$\ddot Y$. Let $\dot f_1:\ddot Y_1\to\hat{\mathbb C}$ and $\dot f_2:\ddot Y_2\to\hat{\mathbb C}$ be functions, where $\ddot Y_1$, $\ddot Y_2\in\tilde Y$. If there exists $\dot\lambda\in\dot\varLambda$ such that $\ddot Y_1\overset{\scriptscriptstyle\,\dot\lambda}{\looparrowright}\ddot Y_2$ and $\dot f_2=\dot f_1\circ\dot\lambda$, then we say that $(\dot f_2,\ddot Y_2)$ is \textit{over}~$(\dot f_1,\ddot Y_1)$. If $(\dot f_2,\ddot Y_2)$ is over~$(\dot f_1,\ddot Y_1)$ or $(\dot f_1,\ddot Y_1)$ is over~$(\dot f_2,\ddot Y_2)$, then we say that $(\dot f_1,\ddot Y_1)$ and $(\dot f_2,\ddot Y_2)$ are \textit{directly compatible}. If there exists a chain of functions $(\dot g_j,\ddot Z_j)$ ($\ddot Z_j\in\tilde Y$, $j=1$, $\dots$, $n$) such that $(\dot g_1,\ddot Z_1)=(\dot f_1,\ddot Y_1)$, $(\dot g_n,\ddot Z_n)=(\dot f_2,\ddot Y_2)$ and $(\dot g_j,\ddot Z_j)$ and $(\dot g_{j+1},\ddot Z_{j+1})$ are directly compatible ($j=1$, $\dots$, $n\!-\!1$), then we say that $(\dot f_1,\ddot Y_1)$ and $(\dot f_2,\ddot Y_2)$ are \textit{compatible}. It is plain that the compatibility relation is an equivalence relation. We call the equivalence class of $(\dot f,\ddot Y)$ a (\textit{complex}) \textit{function} on~$\tilde Y$, denoted~$\tilde f$ or $\tilde f:\tilde Y\to\hat{\mathbb C}$. We call $(\dot f,\ddot Y)$ an \textit{expression element} of~$\tilde f$, where $\dot f$ is called an \textit{expression function} and $\ddot Y$ an \textit{expression domain}. Let $\tilde f|_{\ddot Y}:=\dot f$, called a \textit{restriction} of~$\tilde f$ on~$\ddot Y$. If partial elements of~$\tilde f$ (as a set) are omitted, then we still use it to denote the same function.

For functions $\dot f_1$ on~$\ddot Y_1$ and $\dot f_2$ on~$\ddot Y_2$, where $\ddot Y_1$, $\ddot Y_2\in\tilde Y$, by the following lemma we know that $\dot f_1$ and~$\dot f_2$ are compatible precisely if there exists $\ddot Y_0\in\tilde Y$ such that $\ddot Y_1\overset{\scriptscriptstyle\,\dot\lambda_1}{\looparrowright}\ddot Y_0$, $\ddot Y_2\overset{\scriptscriptstyle\,\dot\lambda_2}{\looparrowright}\ddot Y_0$ for $\dot\lambda_1$, $\dot\lambda_2\in\dot\varLambda$ and $\dot f_1\circ\dot\lambda_1=\dot f_2\circ\dot\lambda_2$.

\begin{lemma}\label{lemma3.7}
Let $\ddot S_1$, $\ddot S_2\in\ddot{\mathcal S}^0$. Then there exists $\ddot S_0\in\ddot{\mathcal S}^0$ and $\dot q_1$, $\dot q_2\in\dot\varLambda^0$ $(\dot\varLambda^0$ is the direct up-harmonious relation in~$\ddot{\mathcal A}^0)$ such that $\ddot S_1\overset{\scriptscriptstyle\,\dot q_1}{\looparrowright}\ddot S_0$ and $\ddot S_2\overset{\scriptscriptstyle\,\dot q_2}{\looparrowright}\ddot S_0$.
\end{lemma}

\begin{proof}
Suppose $\ddot S_j=(\hat\varphi_j,\dot S_j)$ ($j=1$, $2$) and $\mathbf S_0=\mathcal M_{r_0}(\hat\varphi_1,\hat\varphi_2)$ (the field generated by~$\hat\varphi_1$, $\hat\varphi_2$ and~$\mathcal M_{r_0}$), where $r_0$ is the base point in the base surface~$\hat R$. Then there exists $\hat\varphi_0\in\mathbf S_0$ such that $\mathbf S_0=\mathcal M_{r_0}(\hat\varphi_0)$. Hence, there exist polynomials~$\hat Q_1(t)$ and~$\hat Q_2(t)$ in~$\mathcal M(\hat R)[t]$ such that $\hat\varphi_1=\hat Q_1(\hat\varphi_0)$ and $\hat\varphi_2=\hat Q_2(\hat\varphi_0)$. Let $\ddot S_0$ be the original algebraic Riemann surface determined by~$\hat\varphi_0$ with label~$\hat\varphi_0$. Then $\ddot S_1\overset{\scriptscriptstyle\,\dot q_1}{\looparrowright}\ddot S_0$ and $\ddot S_2\overset{\scriptscriptstyle\,\dot q_2}{\looparrowright}\ddot S_0$ by Lemma~\ref{lemma3.6}, where $\dot q_1$ and~$\dot q_2$ are determined by~$\hat Q_1$ and~$\hat Q_2$ respectively.
\end{proof}

Let $\dot f\in\mathcal M(\ddot Y)$, where $\ddot Y\in\tilde Y$. Then the equivalence class~$\tilde f$ of $(\dot f,\ddot Y)$ is called a \textit{meromorphic function} on~$\tilde Y$. If $\dot f\in\mathcal H(\ddot Y)$ then $\tilde f$ is called a \textit{holomorphic function} on~$\tilde Y$. Denote the set of all meromorphic (resp.\ holomorphic) functions on~$\tilde Y$ by $\mathcal M(\tilde Y)$ (resp.\ $\mathcal H(\tilde Y)$). Then $\mathcal M(\tilde Y)$ is a field and $\mathcal H(\tilde Y)$ is a ring by means of the operation defined on representatives. $\mathcal M(\tilde Y)$ is also a vector space over $\mathcal M(\dot R)$ under the scalar multiplication that $\dot h\cdot\tilde f:=\tilde g$ for $\dot h\in\mathcal M(\dot R)$ and $\tilde f\in\mathcal M(\tilde Y)$, where $\tilde f$ is determined by $\dot f\in\mathcal M(\ddot Y)$, $\dot p:\ddot Y\to\dot R$ is the canonical projection and $\tilde g\in\mathcal M(\tilde Y)$ is determined by $\dot g=(\dot h\circ\dot p)\cdot\dot f$. Moreover, we may consider $\mathcal M(\dot R)$ as a subfield of~$\mathcal M(\tilde Y)$ by the monomorphism $\gamma:\mathcal M(\dot R)\to\mathcal M(\tilde Y)$, defined by $\gamma(\dot h):=\tilde{h}$, where $\tilde{h}\in\mathcal M(\tilde Y)$ is determined by $\dot h\in\mathcal M(\dot R)$.

Suppose $\tilde Y$ and $\tilde Z$ are algebraic Riemann surfaces (over base surfaces~$\hat R_1$ and~$\hat R_2$, respectively). Let $\dot\sigma_j:\ddot Y_j\to\ddot Z_j$ be mappings ($\ddot Y_j\in\tilde Y$ and $\ddot Z_j\in\tilde Z$, $j=1$, $2$). If there exist $\dot\lambda$, $\dot\mu\in\dot\varLambda$ such that $\ddot Y_1\overset{\scriptscriptstyle\,\dot\lambda}{\looparrowright}\ddot Y_2$, $\ddot Z_1\overset{\scriptscriptstyle\,\dot\mu}{\looparrowright}\ddot Z_2$ and $\dot\mu\circ\dot\sigma_2=\dot\sigma_1\circ\dot\lambda$ then we say $\dot\sigma_2$ is \textit{over}~$\dot\sigma_1$ or $\dot\sigma_1$ is \textit{under}~$\dot\sigma_2$; if moreover $\dot\sigma_2\circ\dot\lambda^{-1}=\dot\mu^{-1}\circ\dot\sigma_1$ (i.e.\ $\dot\sigma_2(\dot\lambda^{-1}(y_1))=\dot\mu^{-1}(\dot\sigma_1(y_1))$ for any $y_1\in\hat Y_1$), then we say $\dot\sigma_2$ is \textit{exactly over}~$\dot\sigma_1$ or $\dot\sigma_1$ is \textit{exactly under}~$\dot\sigma_2$. If $\dot\sigma_1$ is (resp.\ exactly) over~$\dot\sigma_2$ or $\dot\sigma_2$ is (resp.\ exactly) over~$\dot\sigma_1$, then we say $\dot\sigma_1$ and~$\dot\sigma_2$ are \textit{directly} (resp.\ \textit{directly and exactly}) \textit{compatible}. If in the above $\dot\lambda$ and~$\dot\mu$ are biholomorphic, i.e.\ $\ddot Y_1\overset{\scriptscriptstyle\,\dot\lambda}{\leftrightarrow}\ddot Y_2$ and $\ddot Z_1\overset{\scriptscriptstyle\,\dot\mu}{\leftrightarrow}\ddot Z_2$, then we say $\dot\sigma_1$ and~$\dot\sigma_2$ are \textit{equivalent}, denoted $\dot\sigma_1\thicksim\dot\sigma_2$. This is clearly an equivalence relation.

\smallskip
\begin{remark}\label{remark8}
Evidently, $\dot\sigma_2\circ\dot\lambda^{-1}=\dot\mu^{-1}\circ\dot\sigma_1$ implies $\dot\mu\circ\dot\sigma_2=\dot\sigma_1\circ\dot\lambda$. If $\dot\sigma_1$ is injective and $\dot\sigma_2$ is surjective, then $\dot\sigma_2\circ\dot\lambda^{-1}=\dot\mu^{-1}\circ\dot\sigma_1$ precisely if $\dot\mu\circ\dot\sigma_2=\dot\sigma_1\circ\dot\lambda$.
\end{remark}
\smallskip

For a mapping $\dot\sigma:\ddot Y\to\ddot Z$ we denote its \textit{domain} $\operatorname{dom}(\dot\sigma):=\ddot Y$ and its \textit{codomain} $\operatorname{codom}(\dot\sigma):=\ddot Z$. Suppose $\tilde\sigma$ is a set of some mappings. Suppose $\tilde Y$ and~$\tilde Z$ are two algebraic Riemann surfaces. If for two mappings $\dot\sigma_1$ and~$\dot\sigma_2$, there exists $\dot\sigma_0\in\tilde\sigma$ 
over $\dot\sigma_1$ and~$\dot\sigma_2$ and, moreover, that $\operatorname{dom}(\dot\sigma_2)$ is over (resp.\ under) $\operatorname{dom}(\dot\sigma_1)$ implies that $\dot\sigma_2$ is 
over (resp.\ under)~$\dot\sigma_1$, then we say $\dot\sigma_1$ and~$\dot\sigma_2$ are 
\textit{compatible} in~$\tilde\sigma$. If any two mappings $\dot\sigma_1$, $\dot\sigma_2\in\tilde\sigma$ are 
compatible in~$\tilde\sigma$ then we say $\tilde\sigma$ is 
\textit{compatible} (similarly we have the notion of \textit{exact compatibility} of~$\tilde\sigma$). Suppose $\tilde\sigma$ is compatible and satisfies the following two conditions: 
\medskip\par\noindent
(1) There exists $\dot\sigma_0\in\tilde\sigma$ such that for any $\ddot Y\in\tilde Y$ over $\operatorname{dom}(\dot\sigma_0)$ there exists $\dot\sigma\in\tilde\sigma$ with $\operatorname{dom}(\dot\sigma)\leftrightarrow\ddot Y$;
\smallskip\par\noindent
(2) For any $\ddot Z\in\tilde Z$ , there exists $\dot\sigma'\in\tilde\sigma$ with $\operatorname{codom}(\dot\sigma')$ over $\ddot Z$.  
\medskip\par\noindent
Then we say $\tilde\sigma$ is a \textit{mapping} from $\tilde Y$ to~$\tilde Z$.
We call $(\dot\sigma,\ddot Y)$ an \textit{expression element} of\/~$\tilde\sigma$, where $\dot\sigma$ is called an \textit{expression mapping} and $\ddot Y$ an \textit{expression domain}. Denote $\tilde\sigma|_{\ddot Y}:=\dot\sigma$, called the \textit{restriction} of~$\tilde\sigma$ on~$\ddot Y$. Specially, a function $\tilde f:\tilde Y\to\hat{\mathbb C}$ is also a mapping.

Suppose $\tilde\sigma$ and~$\tilde\tau$ are mappings from $\tilde Y$ to~$\tilde Z$. If every $\dot\sigma\in\tilde\sigma$ over some $\dot\sigma_0\in\tilde\sigma$ is compatible with every $\dot\tau\in\tilde\tau$ over some $\dot\tau_0\in\tilde\tau$ both in~$\tilde\sigma$ and~$\tilde\tau$, then we say $\tilde\sigma$ and~$\tilde\tau$ are \textit{equal}, denoted $\tilde\sigma\overset{\scriptscriptstyle\mathrm{map}}=\tilde\tau$ (it is probable that as sets $\tilde\sigma$ and~$\tilde\tau$ are not equal). This is equivalent to that there exists $\ddot Y_0\in\tilde Y$ such that for every $\dot\sigma\in\tilde\sigma$ with $\operatorname{dom}(\dot\sigma)$ over~$\ddot Y_0$ there exists $\dot\tau\in\tilde\tau$ such that $\dot\tau\thicksim\dot\sigma$ and for every $\dot\tau\in\tilde\tau$ with $\operatorname{dom}(\dot\tau)$ over~$\ddot Y_0$ there exists $\dot\sigma\in\tilde\sigma$ such that $\dot\sigma\thicksim\dot\tau$. In fact, if necessary we may assume $\tilde\sigma$ contains any mapping $\dot\sigma:\ddot Y\to\ddot Z$ that is under any $\dot\sigma_1\in\tilde\sigma$, where $\ddot Y\in\tilde Y$ and $\ddot Z\in\tilde Z$.

Suppose $\tilde Z$ and~$\tilde W$ are algebraic Riemann surfaces over~a base surface~$\hat R$ with base point~$r_0$. If $\tilde Z\subseteq\tilde W$ as sets, then we say $\tilde W$ is \textit{over}~$\tilde Z$ or $\tilde Z$ is \textit{under}~$\tilde W$, denoted $\tilde W\geqslant\tilde Z$ or $\tilde Z\leqslant\tilde W$. Suppose $\tilde\tau:\tilde X\to\tilde Y_1$ and $\tilde\sigma:\tilde Y_2\to\tilde Z$ are mappings, where $\tilde Y_1\leqslant\tilde Y_2$. Let $\tilde\sigma\circ\tilde\tau$ be the set of all mappings $\dot\sigma\circ\dot\lambda\circ\dot\tau$ for all possible $\dot\tau:\ddot X\to\ddot Y_1$ in~$\tilde\tau$ and the corresponding $\dot\sigma:\ddot Y_2\to\ddot Z$ in~$\tilde\sigma$ with $\ddot Y_2\overset{\scriptscriptstyle\dot\lambda}{\leftrightarrow}\ddot Y_1$. If $\tilde\sigma\circ\tilde\tau$ satisfies the condition~(2) of a mapping (when $\tilde Y_1=\tilde Y_2$ this condition is satisfied naturally), then it is a mapping from~$\tilde X$ to $\tilde Z$, which we call the \textit{composition} of $\tilde\sigma$ and~$\tilde\tau$. It is easy to see that $\tilde\sigma_1\overset{\scriptscriptstyle\mathrm{map}}=\tilde\sigma_2$ and $\tilde\tau_1\overset{\scriptscriptstyle\mathrm{map}}=\tilde\tau_2$ imply $\tilde\sigma_1\circ\tilde\tau_1\overset{\scriptscriptstyle\mathrm{map}}=\tilde\sigma_2\circ\tilde\tau_2$. It is also easy to see that the composition satisfies the associative law. Specially, we have a function $\tilde f\circ\tilde\sigma:\tilde X\to\hat{\mathbb C}$ for a mapping $\tilde\sigma:\tilde X\to\tilde Y_1$ and a function $\tilde f:\tilde Y_2\to\hat{\mathbb C}$, where $\tilde Y_1\leqslant\tilde Y_2$ and there is an expression domain~$\ddot Y_2$ of~$\tilde f$ belonging to~$\tilde Y_1$.

\smallskip
\begin{remark}\label{remark9}
If $\dot\sigma_2$ is over~$\dot\sigma_1$, then $\dot\sigma_2$ is surjective implies $\dot\sigma_1$ is surjective. If $\dot\sigma_2$ is exactly over~$\dot\sigma_1$, then $\dot\sigma_2$ is injective implies $\dot\sigma_1$ is injective.
\end{remark}
\smallskip

Suppose $\tilde\sigma:\tilde Y\to\tilde Z$ is a mapping. If every $\dot\sigma\in\tilde\sigma$ over some $\dot\sigma_0\in\tilde\sigma$ is surjective then we say $\tilde\sigma$ is \textit{surjective}. We say $\tilde\sigma$ is \textit{injective} if there exists $\dot\sigma_0\in\tilde\sigma$ such that the following conditions are satisfied: 
\medskip\par\noindent
(1) For any $\ddot Z\in\tilde Z$ over $\operatorname{codom}(\dot\sigma_0)$, there exists $\dot\sigma'\in\tilde\sigma$ with $\operatorname{codom}(\dot\sigma')\leftrightarrow\ddot Z$;
\smallskip\par\noindent
(2) For $\dot\sigma_1$, $\dot\sigma_2\in\tilde\sigma$ over $\dot\sigma_0$, that $\operatorname{codom}(\dot\sigma_2)$ is over $\operatorname{codom}(\dot\sigma_1)$ implies that $\dot\sigma_2$ is over $\dot\sigma_1$;
\smallskip\par\noindent
(3) Every $\dot\sigma\in\tilde\sigma$ over $\dot\sigma_0$ is injective.
\bigskip

We say $\tilde\sigma$ is \textit{bijective} if it is both surjective and injective. Suppose $\tilde\sigma$ is a bijection. Let $\dot\sigma_0\in\tilde\sigma$ be the one in the conditions for $\tilde\sigma$ being bijective. Denote
$$
\tilde\sigma^{-1}:=\{\dot\sigma^{-1}:\,\dot\sigma\in\tilde\sigma \text{ over } \dot\sigma_0\in\tilde\sigma\}.
$$
Then $\tilde\sigma^{-1}$ is a mapping from $\tilde Z$ to~$\tilde Y$, called the \textit{inverse} of~$\tilde\sigma$. Evidently, $\tilde\sigma\overset{\scriptscriptstyle\mathrm{map}}=\tilde\tau$ implies $\tilde\sigma^{-1}\overset{\scriptscriptstyle\mathrm{map}}=\tilde\tau^{-1}$. We also have $\tilde\sigma^{-1}\circ\tilde\sigma\overset{\scriptscriptstyle\mathrm{map}}=\operatorname{id}_{\tilde Y}$ and $\tilde\sigma\circ\tilde\sigma^{-1}\overset{\scriptscriptstyle\mathrm{map}}=\operatorname{id}_{\tilde Z}$, where $\operatorname{id}_{\tilde Y}:=\{\operatorname{id}_{\ddot Y}:\ddot Y\in\tilde Y\}$, called the \textit{identical mapping} on~$\tilde Y$ (clearly, $\operatorname{id}_{\tilde Y}\circ\tilde\tau\overset{\scriptscriptstyle\mathrm{map}}=\tilde\tau$ and $\tilde\sigma\circ\operatorname{id}_{\tilde Y}\overset{\scriptscriptstyle\mathrm{map}}=\tilde\sigma$ for mappings $\tilde\tau:\tilde X\to\tilde Y$ and $\tilde\sigma:\tilde Y\to\tilde Z$). Conversely, suppose $\tilde\sigma:\tilde Y\to\tilde Z$ is a mapping and there exists a mapping $\tilde\tau:\tilde Z\to\tilde Y$ such that $\tilde\tau\circ\tilde\sigma\overset{\scriptscriptstyle\mathrm{map}}=\operatorname{id}_{\tilde Y}$ and $\tilde\sigma\circ\tilde\tau\overset{\scriptscriptstyle\mathrm{map}}=\operatorname{id}_{\tilde Z}$. Then $\tilde\sigma$ is a bijection and $\tilde\sigma^{-1}\overset{\scriptscriptstyle\mathrm{map}}=\tilde\tau$.

Let $\tilde\sigma$ be a mapping from $\tilde Y$ to~$\tilde Z$. If every $\dot\sigma\in\tilde\sigma$ over some $\dot\sigma_0\in\tilde\sigma$ is (resp.\ essentially) continuous (i.e.\ $\hat\sigma$ and~$\bar\sigma$ are (resp.\ essentially) continuous, where $\dot\sigma=(\hat\sigma,\bar\sigma)$) then $\tilde\sigma$ is said to be (\textit{exactly}) (resp.\ \textit{essentially}) \textit{continuous}. If $\tilde\sigma$ is bijective and both $\tilde\sigma$ and~$\tilde\sigma^{-1}$ are (resp.\ essentially) continuous then $\tilde\sigma$ is said to be (\textit{exactly}) (resp.\ \textit{essentially}) \textit{homeomorphic}. If every $\dot\sigma\in\tilde\sigma$ over some $\dot\sigma_0\in\tilde\sigma$ is (resp.\ essentially) holomorphic then $\tilde\sigma$ is said to be (\textit{exactly}) (resp.\ \textit{essentially}) \textit{holomorphic} (\textit{analytic}). If $\tilde\sigma$ is bijective and both $\tilde\sigma$ and~$\tilde\sigma^{-1}$ are (resp.\ essentially) holomorphic then $\tilde\sigma$ is said to be (\textit{exactly}) (resp.\ \textit{essentially}) \textit{biholomorphic}. In fact, it is sufficient for a (resp.\ an essential) homeomorphism~$\tilde\sigma$ being (resp.\ essentially) biholomorphic that $\tilde\sigma$ is (resp.\ essentially) holomorphic.

Suppose $\tilde X$, $\tilde Y$ and~$\tilde Z$ are algebraic Riemann surfaces. Suppose $\tilde p:\tilde Y\to\tilde X$ and $\tilde q:\tilde Z\to\tilde X$ are (resp.\ essentially) continuous maps. A mapping $\tilde\sigma:\tilde Y\to\tilde Z$ is called \textit{fiber-preserving} (\textit{over}~$\tilde X$) if $\tilde p\overset{\scriptscriptstyle\mathrm{map}}=\tilde q\circ\tilde\sigma$. A mapping $\tilde p:\tilde Y\to\tilde X$ is called a (or an \textit{exact}) (resp.\ an \textit{essential}) \textit{covering map} if it is a surjection and every $\dot p\in\tilde p$ over some $\dot p_0\in\tilde p$ is a (resp.\ an essential) covering map.

Suppose $\tilde p:\tilde Y\to\tilde X$ is a (resp.\ an essential) covering map. We call a fiber-preserving (resp.\ essential) homeomorphism $\tilde\sigma:\tilde Y\to\tilde Y$ a (or an \textit{exact}) (resp.\ an \textit{essential}) \textit{covering transformation} or a (or an \textit{exact}) (resp.\ an \textit{essential}) \textit{deck transformation} of~$\tilde p$. Obviously, the set of all deck transformations of~$\tilde p$ forms a group under the compsition of mappings, denoted $\operatorname{Deck}(\tilde Y\!\overset{\scriptscriptstyle{\tilde p}}\rightarrow\!\tilde X)$ or $\operatorname{Deck}(\tilde Y\!/\tilde X)$. The (resp.\ essential) covering $\tilde p:\tilde Y\to\tilde X$ is called \textit{Galois} if for any $\dot p\in\tilde p$ there is $\dot q\in\tilde p$ over~$\dot p$ such that $\dot q$ is Galois. It is easy to see that if the (resp.\ essential) covering map $\tilde p:\tilde Y\to\tilde X$ is (resp.\ essentially) holomorphic then the (resp.\ essential) deck transformations~$\tilde\sigma$ are (resp.\ essentially) biholomorphic.

Suppose $\tilde Z$ and~$\tilde W$ are algebraic Riemann surfaces over~a base surface~$\hat R$ with base point~$r_0$ and $\tilde Z\leqslant\tilde W$. Suppose $\tilde Z$ is determined by a subfield~$\mathbf S$ of~$\mathbf A$ and $\ddot W\in\tilde W$ is an original algebraic Riemann surface with natural label~$\hat\psi$. Let $\mathbf S'=\mathbf S\cap\mathcal M_{r_0}(\hat\psi)$, where $\mathcal M_{r_0}=\{\langle \hat f\rangle_{r_0}:\hat f\in\mathcal M(\hat R)\}$. Then there exists $\hat\varphi\in\mathbf S$ and $\hat P(t)\in\mathcal M_{r_0}[t]$ such that $\mathbf S'=\mathcal M_{r_0}(\hat\varphi)$ and $\hat\varphi=\hat P(\hat\psi)$. Suppose $\ddot Z$ is an original algebraic Riemann surface determined by~$\hat\varphi$ with label~$\hat\varphi$. Then the directly up-harmonious mapping $\dot p:\ddot W\to\ddot Z$ determined by~$\hat P(t)$ (see Lemma~\ref{lemma3.6}) is \textit{maximal}, which means if there exists an analytically up-harmonious mapping $\dot q:\ddot W\to\ddot Z'$ for $\ddot Z'\in\tilde Z$ over~$\ddot Z$ then $\dot q\thicksim\dot p$. Let $\tilde p$ be the set of all~$\dot p$ defined above. Then $\tilde p$ is a holomorphic covering map (see Proposition~\ref{proposition3.5}), called the \textit{natural covering map} from $\tilde W$ to~$\tilde Z$.  Suppose $\tilde X$, $\tilde Y$ and~$\tilde Z$ are algebraic Riemann Surfaces satisfying $\tilde X\leqslant\tilde Z\leqslant\tilde Y$. If $\tilde p:\tilde Y\to\tilde X$, $\tilde q_1:\tilde Y\to\tilde Z$ and $\tilde q_2:\tilde Z\to\tilde X$ are the natural covering maps, then it follows that $\tilde q_2\circ\tilde q_1\overset{\scriptscriptstyle\mathrm{map}}=\tilde p$.


\smallskip
\begin{remark}\label{remark10}
If the concerned algebraic Riemann surfaces have holographic elements, then we may define the mappings between them by means of the holographic elements. In general, we may replace these algebraic Riemann surfaces by their holographic elements, respectively.
\end{remark}

\begin{remark}\label{remark11}
Some of the notions in Subsections~3.3 and~3.4 are based on the consideration of the following example: If we are going to find the sum of two algebraic functions~$\sqrt{z}$ and~$\sqrt[3]{z}$, we may calculate ``\,$\omega_1(\sqrt[6]{z})^3+\omega_2(\sqrt[6]{z})^2$\,'' on the Riemann surface of algebraic function~$\sqrt[6]{z}$ instead, where $\omega_1^2=\omega_2^3=1$.
\end{remark}


\section{Algebraic Riemann surfaces and the Galois correspondence}

Let $\hat R$ be a base surface (a basic Riemann surface with $\hat R=R$) with base point~$r_0\in\hat R$. Suppose $\tilde X$ and~$\tilde Y$ are algebraic Riemann Surfaces over~$\hat R$ and $\tilde Y$ is over~$\tilde X$. Suppose the natural covering map $\tilde\pi:\tilde Y\to\tilde X$ is Galois (then we say $\tilde Y$ is \textit{Galois} over~$\tilde X$ or $\tilde Y/\tilde X$ is \textit{Galois}). In this section $\deck(\tilde Y\!/\tilde X)$ always means $\deck(\tilde Y\!\overset{\!\scriptscriptstyle\tilde\pi}\to\!\tilde X)$. Suppose $\tilde Z$ is an algebraic Riemann Surface satisfying $\tilde X\leqslant\tilde Z\leqslant\tilde Y$, which we call an \textit{intermediate $($algebraic$)$ Riemann Surface} of\/~$\tilde Y/\tilde X$.

Let $D:=\deck(\tilde Y\!/\tilde X)$. Then $E=\deck(\tilde Y\!/\tilde Z)$ is a subgroup of~$D$. Denote
$$
[\tilde Z:\tilde X]:=[\mathcal M(\tilde Z):\mathcal M(\tilde X)]
$$
(the degree of field extension $\mathcal M(\tilde Z)/\mathcal M(\tilde X)$), called the \textit{$($covering$)$ order} of $\tilde Z$ over~$\tilde X$ (or of\/~$\tilde Z/\tilde X$). If $[\tilde Z:\tilde X]$ is finite then we say $\tilde Z$ is \textit{finite} over~$\tilde X$ or $\tilde Z/\tilde X$ is \textit{finite}.

Denote
$$
\operatorname{Int}(\tilde Y\!/\tilde X):=\{\tilde Z:\,\tilde X\leqslant\tilde Z\leqslant\tilde Y\},
$$
$$
\operatorname{FG}(\tilde Y\!/\tilde X):=\{\tilde Z\in\operatorname{Int}(\tilde Y\!/\tilde X):\,\tilde Z/\tilde X \text{ is finite Galois}\}
$$
and
$$
\mathscr E:=\{E=\deck(\tilde Y\!/\tilde Z):\,\tilde Z\in\operatorname{FG}(\tilde Y\!/\tilde X)\}.
$$
We define a topology on~$D$ similar to the Krull topology on a Galois group as follows:
$$
\mathscr K_D:=\{T:\,T=\emptyset \text{ or } T=\bigcup_j\tilde\tau_j E_j \text{ for some } \tilde\tau_j\in D \text{ and } E_j\in\mathscr E\}
$$
(cf.\ \cite[Definition~17.5]{Mo}). By Lemma~\ref{lemma4.5} below (refer to Lemma~\ref{lemma4.7} below) and similar reasoning to that in \cite[Section~17]{Mo} we see that $\mathscr K_D$ is really a topology on~$D$. We also call this topology the \textit{Krull topology} on~$D$. We always assume that $D$ is equipped with the Krull topology later. Denote
$$
\mathfrak C:=\{C:\,C \text{ is a closed subgroup of } D\}.
$$
If $\tilde Y$ is finite over~$\tilde X$, then $\mathfrak C$ is consisting of all subgroups of~$D$ (refer to Lemma~\ref{lemma4.7} below and \cite[Example~17.9]{Mo}).

\begin{theorem}[Galois Correspondence on Algebraic Riemann Surfaces]\label{theorem4.1}
Suppose $\tilde X$ and~$\tilde Y$ are algebraic Riemann Surfaces and $\tilde Y$ is Galois over~$\tilde X$. Let $D=\deck(\tilde Y\!/\tilde X)$. Then the mapping
\begin{equation*}
\begin{split}
\Delta:\,\operatorname{Int}(\tilde Y\!/\tilde X)&\to\mathfrak C \\
\tilde Z&\mapsto\deck(\tilde Y\!/\tilde Z)
\end{split}
\end{equation*}
is a bijection, which gives an inclusion reversing correspondence and whose inverse mapping is\/ $\Gamma$ given in~$(\ref{4.13})$ below. Moreover, letting $E=\Delta(\tilde Z)$, we have
\medskip\par\noindent
$(1)$ The following statements are equivalent:
\medskip\par
$($\rm{i}$)$ $[D:E]$ (the index of~$E$ in~$D$) is finite;
\smallskip\par
$($\rm{ii}$)$ $[\tilde Z:\tilde X]$ is finite;
\smallskip\par
$($\rm{iii}$)$ $E$ is open in~$D$.
\medskip\par\noindent
On this condition it is true that $[D:E]=[\tilde Z:\tilde X]$.
\medskip\par\noindent
$(2)$ $\tilde Z$ is Galois over~$\tilde X$ if and only if $E$ is normal in~$D$. On this condition, there is a group isomorphism
\begin{equation*}
\deck(\tilde Z/\tilde X)\cong D/E, 
\end{equation*}
which is also a homeomorphism as the quotient group~$D/E$ is given the quotient topology.
\end{theorem}

In order to obtain Theorem~\ref{theorem4.1}, we give some preliminary results.

Suppose $\tilde Y$ is an algebraic Riemann surface over a base surface~$\hat R$ with base point~$r_0$ and $\tilde Y^0$ is the corresponding original algebraic Riemann surface. Suppose $\tilde f\in\mathcal M(\tilde Y)$. Then there exists $\ddot S_1=(\hat\varphi_1;\dot S_1)\in\tilde Y^0$ with the original basic function $\dot F_1$ and a polynomial $\dot P_1(t)=(\hat P_1(t),\bar P_1(t))\in\mathcal M(\dot R)[t]$ such that $\tilde f|_{\ddot S_1}=\dot P_1(\dot F_1)$ by Lemma~\ref{lemma3.4}. If there exists another $\ddot S_2=(\hat\varphi_2;\dot S_2)\in\tilde Y^0$ with the original basic function $\dot F_2$ and a polynomial $\dot P_2(t)=(\hat P_2(t),\bar P_2(t))\in\mathcal M(\dot R)[t]$ such that $\tilde f|_{\ddot S_2}=\dot P_2(\dot F_2)$, then by Lemma~\ref{lemma3.7} we can take a common covering $\ddot S_0=(\hat\varphi_0;\dot S_0)\in\tilde Y^0$ of $\ddot S_1$ and~$\ddot S_2$ with $\ddot S_1\overset{\scriptscriptstyle\dot q_1}\looparrowright\ddot S_0$, $\ddot S_2\overset{\scriptscriptstyle\dot q_2}\looparrowright\ddot S_0$, $\hat Q_1(\hat\varphi_0)=\hat\varphi_1$ and $\hat Q_2(\hat\varphi_0)=\hat\varphi_2$, where $\dot Q_j(t)=(\hat Q_j(t),\bar Q_j(t))\in\mathcal M(\dot R)[t]$ and $\dot q_j$ is determined by $\dot Q_j$ ($j=1$, $2$). Thus $\dot F_1\circ\dot q_1=\dot Q_1(\dot F_0)$ and $\dot F_2\circ\dot q_2=\dot Q_2(\dot F_0)$ by Lemma~\ref{lemma3.6}, where $\dot F_0$ is the original basic function on~$\ddot S_0$. Hence
$$
\tilde f|_{\ddot S_0}=\tilde f|_{\ddot S_1}\circ\dot q_1=\dot P_1(\dot F_1)\circ\dot q_1=\dot P_1(\dot F_1\circ\dot q_1)=\dot P_1(\dot Q_1(\dot F_0)),
$$
and similarly
$$
\tilde f|_{\ddot S_0}=\dot P_2(\dot Q_2(\dot F_0)),
$$
which imply $\dot P_1(\dot Q_1(\dot F_0))=\dot P_2(\dot Q_2(\dot F_0))$. Therefore,
$$
\hat P_1(\hat Q_1(\hat\varphi_0))=\hat P_2(\hat Q_2(\hat\varphi_0)),
$$
i.e.\ $\hat P_1(\hat\varphi_1)=\hat P_2(\hat\varphi_2)$. So we can give the following definition:
We call a (resp.\ the original) basic algebraic Riemann surface determined by $\hat P_1(\hat\varphi_1)$ with label (resp.\ with natural label $\hat P_1(\hat\varphi_1)$) a (resp.\ the \textit{original}) \textit{basic domain} of~$\tilde f$. For the basic domain of~$\tilde f$, the label means that corresponding to the natural label $\hat P_1(\hat\varphi_1)$. We denote the original basic domain of~$\tilde f$ by~$\operatorname{obdom}(\tilde f)$. We call the (resp.\ original) level surface determined by $\hat P_1(\hat\varphi_1)$ (i.e.\ containing $\operatorname{obdom}(\tilde f)$) the (resp.\ \textit{original}) \textit{level domain}, denoted by $\mathcal L(\tilde f)$ (resp.\ $\mathcal L^0(\tilde f)$). We call the (resp.\ original) algebraic Riemann surface determined by $\hat P_1(\hat\varphi_1)$ the (resp.\ \textit{original}) \textit{natural domain} of~$\tilde f$, denoted by $\operatorname{Ndom}(\tilde f)$ (resp.\ $\operatorname{oNdom}(\tilde f)$).

By the reasoning and the definition in the above and Lemma~\ref{lemma3.6}, we can deduce

\begin{lemma}\label{lemma4.2}
Suppose $\tilde Y$ is an algebraic Riemann surface and $\tilde f\in\mathcal M(\tilde Y)$. Then $\operatorname{Ndom}(\tilde f)\leqslant\tilde Y$ and every element in~$\mathcal L(\tilde f)$, i.e.\ every basic domain of~$\tilde f$, is a holographic element of~$\,\operatorname{Ndom}(\tilde f)$. For $\ddot S\in\tilde Y^0$ and the original basic function~$\dot F$ on~$\ddot S$ we have $\operatorname{obdom}(\tilde F)=\ddot S$, where $\tilde F\in\mathcal M(\tilde Y)$ is determined by~$\dot F$. If $\ddot S=(\hat\varphi;\dot S)=\operatorname{obdom}(\tilde f)$, $\ddot S_1=(\hat\varphi_1;\dot S_1)\in\tilde Y^0$ is an expression domain of~$\tilde f$ and $\dot F_1$ is the original basic function on~$\ddot S_1$, then $\ddot S$ is an expression domain of~$\tilde f$ and $\dot F:=\tilde f|_{\ddot S}$ is the original basic function on~$\ddot S$, and there exists a directly up-harmonious mapping $\dot\pi_1:\ddot S_1\to\ddot S$ $($i.e.\ $\ddot S\overset{\scriptscriptstyle\dot\pi_1}\looparrowright\ddot S_1)$ and a polynomial $\dot P_1(t)=(\hat P_1(t),\bar P_1(t))\in\mathcal M(\dot R)[t]$ such that $\dot F\circ\dot\pi_1=\dot P_1(\dot F_1)=\tilde f|_{\ddot S_1}$ and $\hat\varphi=\hat P_1(\hat\varphi_1)$.
\hfill $\square$
\end{lemma}

Suppose $\tilde Y$ is an algebraic Riemann surface over the base surface~$\hat R$. We will consider algebraic Riemann surfaces under~$\tilde Y$. Suppose $\tilde Z\leqslant\tilde Y$. Let
$$
\mathcal M_{\tilde Y}(\tilde Z):=\{\tilde f\in\mathcal M(\tilde Y):\,\text{there exists } \ddot Z\in\tilde Z \text{ and } \dot g\in\mathcal M(\ddot Z) \text{ such that } \dot g\in\tilde f\}.
$$
If there is no confusion, we will write $\mathcal M(\tilde Z)$ instead of~$\mathcal M_{\tilde Y}(\tilde Z)$. Actually, we may consider $\mathcal M(\tilde Z)$ in the usual sense as $\mathcal M_{\tilde Y}(\tilde Z)$.

\begin{lemma}\label{lemma4.3}
Suppose $\tilde X\leqslant\tilde Y$ and\/ $\operatorname{Int}(K/N)$ denotes the set of all intermediate fields of\/~$K/N$, where $N=\mathcal M(\tilde X)$ and $K=\mathcal M(\tilde Y)$. Then the mapping
\begin{equation*}
\begin{split}
\mathcal M:\,\operatorname{Int}(\tilde Y\!/\tilde X)&\to\operatorname{Int}(K/N) \\
\tilde Z&\mapsto L=\mathcal M(\tilde Z)
\end{split}
\end{equation*}
is a partial order preserving bijection, whose inverse mapping is $\mathcal R$ given in~$(\ref{4.4})$ below.
\end{lemma}

\begin{proof}
Suppose $L\in\operatorname{Int}(K\!/\!N)$. Define
$$
\mathcal R(L):=\bigcup\{\mathcal L(\tilde f):\,\tilde f\in L\}.
$$
We see that $\mathcal R(L)\in\operatorname{Int}(\tilde Y\!/\!\tilde X)$, since it is determined by
$$
\mathsf L(L):=\{\hat\varphi\in\mathbf A:\,\hat\varphi \text{ is the natural label of } \operatorname{obdom}(\tilde f) \text{ for } \tilde f\in L\},
$$
which is an intermediate field of\/~$\mathsf L(\tilde Y)/\mathsf L(\tilde X)$ ($\mathsf L(\tilde X)$ and~$\mathsf L(\tilde Y)$ are the natural label fields of~$\tilde X$ and~$\tilde Y$ respectively) by Lemma~\ref{lemma4.2}. It is easy to see that the mapping
\begin{equation}
\begin{split}\label{4.4}
\mathcal R:\,\operatorname{Int}(K/N)&\to\operatorname{Int}(\tilde Y\!/\tilde X) \\
L&\mapsto\tilde Z=\mathcal R(L)
\end{split}
\end{equation}
is the inverse of~$\mathcal M$ by Lemma~\ref{lemma4.2}. The preserving of partial order by~$\mathcal M$ is obvious.
\end{proof}

Suppose $\tilde Y^0$ is the original algebraic Riemann surface corresponding to~$\tilde Y$. Suppose $\ddot X$, $\ddot Z\in\tilde Y^0$ and $\ddot Z$ is over~$\ddot X$. Let $\tilde{\ddot X}$ and~$\tilde{\ddot Z}$ denote the level surfaces containing $\ddot X$ and~$\ddot Z$, respectively. Let $N_0=\mathcal M(\ddot X)$, $L_0=\mathcal M(\ddot Z)$, $N=\mathcal M(\tilde{\ddot X}):=\mathcal M(\tilde X)$, where $\tilde X$ is the algebraic Riemann surface with holographic element~$\ddot X$, and $L=\mathcal M(\tilde{\ddot Z})$. Let $\dot\pi:\ddot Z\to\ddot X$ be the directly up-harmonious mapping, which is a holomorphic covering map (called the \textit{natural covering map}). Then $\dot\pi^*:N_0\to L_0$ is a monomorphism of fields and there exists an isomorphism $\gamma:L_0\to L$, i.e.\ $L_0\overset{\scriptscriptstyle\gamma}\cong L$, which can be defined by $\dot f\mapsto\tilde f$, where $\tilde f\in\mathcal M(\tilde Z)$ is determined by $\dot f\in\mathcal M(\ddot Z)$ ($\tilde f$ may be written~$(\dot f)^{\sim}$), such that $\dot\pi^*(N_0)\overset{\scriptscriptstyle\gamma_1}\cong N$ ($\gamma_1=\gamma|_{\dot\pi^*(N_0)}$). Let $\dot F$ be the original basic function on~$\ddot Z$ and let $\dot P_0(t)\in N_0[t]$ be the minimal polynomial of~$\dot F$ over~$\dot\pi^*(N_0)$, i.e.\ the monic irreducible polynomial in~$N_0[t]$ satisfying $(\dot\pi^*\dot P_0)(\dot F)=0$. We call $\dot P_0(t)$ the \textit{minimal polynomial} of~$\ddot Z$ over~$\ddot X$.

For the $n$-sheeted natural covering map $\dot\pi:\ddot Z\to\ddot X$, suppose $B$ is the set of branch points of $\bar\pi:\bar Z\to\bar X$ ($\dot\pi=(\hat\pi,\bar\pi)$) and $U\subseteq\bar X\!\setminus\! B$ ($\dot X=(\hat X,\bar X)$) is a non-empty open set such that $\bar\pi^{-1}(U)$ is the disjoint union of open sets $V_1$, $\dots$, $V_n$ and $\bar\pi|_{V_j}\!:V_j\to U$ is biholomorphic ($j=1$, $\dots$, $n$). Let $\bar\tau_j=\hat\tau_j:U\to V_j$ be the inverse mapping of~$\bar\pi|_{V_j}$. Let $\dot f_j=\dot f\circ\dot\tau_j$, where $\dot f\in\mathcal M(\ddot Z)$ and $\dot\tau_j=(\hat\tau_j,\bar\tau_j)$. We consider the polynomial
$$
\dot Q_{\dot f}(t)=\prod_{j=1}^n(t-\dot f_j)=t^n+\dot c_1 t^{n-1}+\cdots+\dot c_n,
$$
where $\dot c_j=(-1)^j\mathbf s_j(\dot f_1,\dots,\dot f_n)$ ($\mathbf s_j$ denotes the $j$-th elementary symmetric function in $n$~variables, $j=1$, $\dots$, $n$). Similarly to \cite[\S8.1, \S8.2 and~\S8.3]{Fo} we can deduce that $\dot Q_{\dot F}(t)$ is just the minimal polynomial $\dot P_0(t)$ of~$\dot F$ over $\dot\pi^*(N_0)$. So we get
\begin{equation}\label{4.5}
\deg\dot P_0=n,
\end{equation}
where $\deg\dot P_0$ denotes the degree of $\dot P_0(t)$.

Let $\mathsf G_{\dot\sigma}(\dot f):=\dot f\circ\dot\sigma^{-1}$ for $\dot\sigma\in\deck(\ddot Z/\ddot X)$ and $\dot f\in L_0=\mathcal M(\ddot Z)$. Then $\mathsf G_{\dot\sigma}\in\gal(L_0/\dot\pi^*(N_0))$. Define
$$
\mathscr G(\dot\sigma):=\mathsf G_{\dot\sigma}
$$
for $\dot\sigma\in\deck(\ddot Z/\ddot X)$. Let $\tilde\alpha(\tilde f):=\gamma(\alpha(\dot f))$ for $\alpha\in\gal(L_0/\dot\pi^*(N_0))$ and $\dot f\in L_0$, where $\tilde f=\gamma(\dot f)$. Then $\tilde\alpha\in\gal(L/N)$. Define
$$
\beta(\alpha):=\tilde\alpha
$$
for $\alpha\in\gal(L_0/\dot\pi^*(N_0))$. Similarly to \cite[Theorem~(8.12)]{Fo} we have

\begin{lemma}\label{lemma4.4}
Suppose $\tilde Y^0$ is an original algebraic Riemann surface over a base surface~$\hat R$, $\ddot X$, $\ddot Z\in\tilde Y^0$ and $\ddot Z$ is over~$\ddot X$. Let $N_0=\mathcal M(\ddot X)$, $L_0=\mathcal M(\ddot Z)$, $N=\mathcal M(\tilde{\ddot X}):=\mathcal M(\tilde X)$ and $L=\mathcal M(\tilde{\ddot Z})$. Suppose $\dot\pi:\ddot Z\to\ddot X$ is the $n$-sheeted natural covering map and $\dot P_0(t)\in N_0[t]$ is the minimal polynomial of~$\ddot Z$ over~$\ddot X$. Then
\smallskip\par\noindent
$(1)$ $[L:N]=[L_0:\dot\pi^*(N_0)]=\operatorname{deg}\dot P_0=n$ and $\displaystyle L\cong L_0\cong N_0[t]/(\dot P_0(t));$
\par\noindent
$(2)$ $\displaystyle\deck(\ddot Z/\ddot X)\overset{\scriptscriptstyle\mathscr G}\cong\gal(L_0/\dot\pi^*(N_0))\overset{\scriptscriptstyle\beta}\cong\gal(L/N);$
\smallskip\par\noindent
$(3)$ The natural covering $\dot\pi:\ddot Z\to\ddot X$ is Galois precisely if the field extension $L/N$ $($or $L_0/\dot\pi^*(N_0))$ is Galois.
\end{lemma}

\begin{proof}
Noticing~(\ref{4.5}), by Lemma~\ref{lemma3.4} we can easily see that (1) is true. For (2), we show that $\mathscr G$ is surjective as follows.

Suppose $\alpha\in\gal(L_0/\dot\pi^*(N_0))$ and $\dot F$ is the original basic function on~$\ddot Z$. Then $\dot P(\alpha(\dot F))=\alpha(\dot P(\dot F))=0$, where $\dot P(t)\in\mathcal M(\dot R)[t]$ is the minimal polynomial of~$\dot F$ over~$\mathcal M(\dot R)$. By Theorem~\ref{theorem3.3}, there exists a base-preserving biholomorphic mapping $\dot\sigma:\ddot Z\to\ddot Z$ such that $\mathsf G_{\dot\sigma}(\dot F)=\dot F\circ\dot\sigma^{-1}=\alpha(\dot F)$. For $\dot f\in\mathcal M(\ddot Z)$ there exists $\dot Q(t)\in\mathcal M(\dot R)[t]$ such that $\dot f=\dot Q(\dot F)$ by Lemma~\ref{lemma3.4}. Hence,
$$
\mathsf G_{\dot\sigma}(\dot f)=\dot f\circ\dot\sigma^{-1}=\dot Q(\dot F)\circ\dot\sigma^{-1}=\dot Q(\dot F\circ\dot\sigma^{-1})=\dot Q(\alpha(\dot F))=\alpha(\dot Q(\dot F))=\alpha(\dot f).
$$
Specially for the original basic function~$\dot F_0$ on~$\ddot X$, since $\dot F_0\circ\dot\pi\in\dot\pi^*(N_0)$, we have
$$
\dot F_0\circ\dot\pi\circ\dot\sigma^{-1}=\alpha(\dot F_0\circ\dot\pi)=\dot F_0\circ\dot\pi.
$$
So it follows $\dot\pi\circ\dot\sigma^{-1}=\dot\pi$, which means $\dot\sigma\in\deck(\ddot Z/\ddot X)$.

Assertion~(3) follows from assertions~(1) and~(2) and that $L/N$ is Galois precisely if $|\gal(L/N)|=[L:N]$ (see \cite[Corollary~2.16]{Mo}).
\end{proof}

\begin{remark}\label{remark12}
Because of the isomorphism~$\beta$ we may use $\alpha$ instead of~$\tilde\alpha$.
\end{remark}
\smallskip

In the following we will write ``$=$" instead of ``$\overset{\scriptscriptstyle\mathrm{map}}=$" for the equality of mappings on algebraic Riemann surfaces.

\begin{lemma}\label{lemma4.5}
Suppose $\tilde Z\leqslant\tilde W$ $(\tilde Z$, $\tilde W\in\operatorname{Int}(\tilde Y\!/\tilde X))$. Let $L=\mathcal M(\tilde Z)$ and $M=\mathcal M(\tilde W)$. For $\tilde\sigma\in\deck(\tilde W\!/\tilde Z)$ define
$$
\mathsf G_{\tilde\sigma}(\tilde f):=\tilde f\circ\tilde\sigma^{-1},
$$
where $\tilde f\in M$. Then $\mathsf G_{\tilde\sigma}\in\gal(M/L)$ and the mapping
\begin{equation*}
\begin{split}
\mathcal G=\mathcal G_{\tilde W\!/\tilde Z}:\,\deck(\tilde W\!/\tilde Z)&\to\gal(M/L) \\
\tilde\sigma&\mapsto \mathsf G_{\tilde\sigma}
\end{split}
\end{equation*}
is a group isomorphism. Moreover, $\tilde W\!/\tilde Z$ is Galois if and only if $M/L$ is Galois.
\end{lemma}

\begin{proof}
At first, $\mathsf G_{\tilde\sigma}\in\gal(M/L)$ since $\tilde f\circ\tilde\pi=\tilde f$\, for $\tilde f\in L=\mathcal M(\tilde Z)$, where $\tilde\pi:\tilde W\to\tilde Z$ is the natural covering map. For $\tilde\sigma$, $\tilde\tau\in\deck(\tilde W\!/\tilde Z)$ we have
$$
\mathcal G(\tilde\sigma\circ\tilde\tau)=\mathsf G_{\tilde\sigma\circ\tilde\tau}=\mathsf G_{\tilde\sigma}\circ \mathsf G_{\tilde\tau}=\mathcal G(\tilde\sigma)\circ\mathcal G(\tilde\tau),
$$
since
$$
\mathsf G_{\tilde\sigma\circ\tilde\tau}(\tilde f)=\tilde f\circ(\tilde\sigma\circ\tilde\tau)^{-1}=\tilde f\circ\tilde\tau^{-1}\circ\tilde\sigma^{-1}=(\mathsf G_{\tilde\sigma}\circ \mathsf G_{\tilde\tau})(\tilde f)
$$
for $\tilde f\in M$. Hence, $\mathcal G$ is a group homomorphism.

(i) Suppose $\mathcal G(\tilde\sigma)=\mathsf G_{\tilde\sigma}=\id_M$ for $\tilde\sigma\in\deck(\tilde W\!/\tilde Z)$. Then for every $\tilde f\in M$ we have $\mathsf G_{\tilde\sigma}(\tilde f)=\tilde f$, i.e.\ $\tilde f\circ\tilde\sigma^{-1}=\tilde f$. For $\ddot S=(\hat\varphi;\dot S)\in\tilde W^0$ (the original algebraic Riemann surface corresponding to~$\tilde W$) over some $\ddot S_0\in\tilde W^0$, let $\tilde F\in M$ be determined by the original basic function~$\dot F$ on~$\ddot S$. Then $\ddot S=\operatorname{obdom}(\tilde F)$ by Lemma~\ref{lemma4.2}. Since $\tilde F\circ\tilde\sigma^{-1}=\tilde F$, i.e.\ \,$\tilde F\circ\tilde\sigma=\tilde F$, there exists $\dot\sigma:\ddot W_1\to\ddot W_2$ in~$\tilde\sigma$, where $\ddot S\overset{\scriptscriptstyle\dot\mu}\leftrightarrow\ddot W_2$ and $\ddot W_1$ is also an expression domain of~$\tilde F$ in~$\tilde W$, such that
\begin{equation}\label{4.6}
\tilde F|_{\ddot W_2}\circ\dot\sigma=\tilde F|_{\ddot W_1}.
\end{equation}
By Lemma~\ref{lemma4.2}, there exists an analytically up-harmonious mapping $\dot\pi_1:\ddot W_1\to\ddot S$ such that
\begin{equation}\label{4.7}
\tilde F|_{\ddot W_1}=\dot F\circ\dot\pi_1.
\end{equation}
By (\ref{4.6}) and~(\ref{4.7}) it follows $\dot F\circ\dot\mu\circ\dot\sigma=\dot F\circ\dot\pi_1$. So we get $\dot\mu\circ\dot\sigma=\dot\pi_1$, i.e.\ $\dot\mu\circ\dot\sigma=\id_{\ddot S}\circ\dot\pi_1$, which mean $\ddot S\overset{\scriptscriptstyle\dot\pi_1}\leftrightarrow\ddot W_1$ (since $\dot\mu$ and~$\dot\sigma$ are biholomorphic) and $\dot\sigma\thicksim\id_{\ddot S}$. Therefore, $\tilde\sigma=\id_{\tilde W}$ and then $\mathcal G$ is injective.

(ii) Now we show that $\mathcal G$ is surjective. Suppose $\alpha\in\gal(M\!/\!L)$. Suppose $\ddot S=(\hat\varphi;\dot S)\in\tilde W^0$ and $\dot F$ is the original basic function on~$\ddot S$. Let $\tilde F\in M$ be determined by~$\dot F$. Suppose $\ddot S'=(\hat\varphi';\dot S')=\operatorname{obdom}(\alpha(\tilde F))$. Then $\ddot S'\in\tilde W^0$ by Lemma~\ref{lemma4.2}. Since
$$
\dot P(\alpha(\tilde F)|_{\ddot S'})=\dot P(\alpha(\tilde F))|_{\ddot S'}=\alpha(\dot P(\tilde F))|_{\ddot S'}=0,
$$
where $\dot P(t)\in\mathcal M(\dot R)[t]$ is the minimal polynomial of~$\dot F$ over~$\mathcal M(\dot R)$, there exists a base-preserving biholomorphic mapping $\dot\sigma=\dot\sigma(\alpha,\ddot S):\ddot S\to\ddot S'$ such that
\begin{equation}\label{4.8}
\dot F'\circ\dot\sigma=\dot F,
\end{equation}
where $\dot F'=\alpha(\tilde F)|_{\ddot S'}$, by Theorem~\ref{theorem3.3} (cf.\ \cite[Theorem~(8.9)]{Fo}).
Let
$$
\tilde\sigma=\tilde\sigma_{\alpha}:=\{\dot\sigma(\alpha,\ddot S):\,\ddot S\in\tilde W^0\}.
$$

Given $\ddot S'\in\tilde W^0$, suppose $\dot F'$ is the original basic function on~$\ddot S'$ and $\tilde F'\in M$ is determined by~$\dot F'$, i.e.\ $\tilde F'|_{\ddot S'}=\dot F'$. Let $\tilde F=\alpha^{-1}(\tilde F')$, $\ddot S=\operatorname{obdom}(\tilde F)$ and $\dot F=\tilde F|_{\ddot S}$. Then similarly to the above we can deduce that there is a unique mapping $\dot\sigma(\alpha,\ddot S):\ddot S\to\ddot S'$ in~$\tilde\sigma$ (satisfying~(\ref{4.8})) corresponding to~$\ddot S'$ and~$\alpha$. Moreover, we can deduce that $\dot\sigma(\alpha^{-1},\ddot S'):\ddot S'\to\ddot S$ is just the inverse of~$\dot\sigma(\alpha,\ddot S)$.

\smallskip
\begin{remark}\label{remark13}
In the above definition of~$\tilde\sigma$, we may assume the expression domain~$\ddot S'$ of $\alpha(\tilde F)$ is replaced by~$\ddot S$ if necessary provided that $\ddot S'\overset{\scriptscriptstyle\dot\lambda}\leftrightarrow\ddot S$ for $\dot\lambda\in\bar\Lambda$ $($the analytically harmonious relation$)$. To see that this assumption is reasonable, we consider different $\ddot S$ and~$\ddot S'$, which are analytically harmonious modulo $\dot\lambda:\ddot S\to\ddot S'$. Then by Theorem~\ref{theorem3.3} there are base-preserving biholomorphic mappings $\dot\sigma:\ddot S\to\ddot S'$ and $\dot\sigma_0:\ddot S\to\ddot S$ such that $\alpha(\tilde F)|_{\ddot S'}=\dot F\circ\dot\sigma^{-1}$ and $\alpha(\tilde F)|_{\ddot S}=\dot F\circ\dot\sigma_0^{-1}$, where $\dot F$ is the original basic function on~$\ddot S$ and $\tilde F|_{\ddot S}=\dot F$. Therefore,
$$
\dot F\circ\dot\sigma_0^{-1}=\alpha(\tilde F)|_{\ddot S}=\alpha(\tilde F)|_{\ddot S'}\circ\dot\lambda=\dot F\circ\dot\sigma^{-1}\circ\dot\lambda.
$$
We deduce $\dot\sigma_0=\dot\lambda^{-1}\circ\dot\sigma$, i.e.\ $\dot\sigma_0\thicksim\dot\sigma$, since $\dot F$ is the original basic function on~$\ddot S$. Similarly, $\ddot S'$ may also be replaced by another~$\ddot S'_1$, provided that $\ddot S'\leftrightarrow\ddot S'_1$.
\end{remark}

\begin{remark}\label{remark14}
In fact, by~$(\ref{4.8})$ we can deduce that $\dot S=\dot S'$ $(\dot F'=\alpha(\tilde F)|_{\ddot S'}$ is the original basic function on~$\ddot S'$ by Lemma~$\ref{lemma4.2}$ since $\ddot S'=\operatorname{obdom}(\alpha(\tilde F)))$ and $\dot\sigma(\alpha,\dot S):\dot S\to\dot S'$ is an identical mapping. But generally $\ddot S\ne\ddot S'$ (for $\alpha\ne\id$) and so $\dot\sigma(\alpha,\ddot S):\ddot S\to\ddot S'$ is not any identical mapping essentially.
\end{remark}
\smallskip

Suppose $\ddot S_1$, $\ddot S_2\in\tilde W^0$. Suppose $\dot F_1$ and~$\dot F_2$ are the original basic functions on $\ddot S_1$ and~$\ddot S_2$, respectively. If there is a direct up-harmonious mapping $\dot p:\ddot S_2\to\ddot S_1$, then by Lemma~\ref{lemma3.4} there exists $\dot P(t)\in\mathcal M(\dot R)[t]$ such that
\begin{equation}\label{4.9}
\dot F_1\circ\dot p=\dot P(\dot F_2).
\end{equation}
Suppose $\tilde F_1$ and~$\tilde F_2$ in~$M$ are determined by $\dot F_1$ and~$\dot F_2$ respectively, i.e.\ $\tilde F_1|_{\ddot S_1}=\dot F_1$ and $\tilde F_2|_{\ddot S_2}=\dot F_2$. Then
$$
\dot P(\tilde F_2)|_{\ddot S_2}=\dot P(\tilde F_2|_{\ddot S_2})=\dot P(\dot F_2)=\tilde F_1|_{\ddot S_1}\circ\dot p
$$
by~(\ref{4.9}), so that $\dot P(\tilde F_2)=\tilde F_1$. Hence $\alpha(\tilde F_1)=\alpha(\dot P(\tilde F_2))=\dot P(\alpha(\tilde F_2))$. Let $\tilde F'_j=\alpha(\tilde F_j)$ ($j=1$, $2$). Then
\begin{equation}\label{4.10}
\tilde F'_1=\dot P(\tilde F'_2).
\end{equation}
Let $\dot\sigma_j:\ddot S_j\to\ddot S'_j$ ($j=1$, $2$) be base-preserving biholomorphic mappings such that
\begin{equation}\label{4.11}
\tilde F'_j|_{\ddot S'_j}=\dot F_j\circ\dot\sigma_j^{-1},
\end{equation}
where $\ddot S'_j=(\hat\varphi'_j;\dot S'_j)=\operatorname{obdom}(\tilde F'_j)$.

Suppose $\ddot S'_0=(\hat\varphi'_0;\dot S'_0)\in\tilde W$ is an original algebraic Riemann surface over $\ddot S'_1$ and~$\ddot S'_2$, and $\dot F'_0$ is the original basic function on~$\ddot S'_0$ (refer to Lemma~\ref{lemma3.7}). Then by Lemma~\ref{lemma3.4} there exists $\dot Q_j(t)\in\mathcal M(\dot R)[t]$ such that $\tilde F'_j|_{\ddot S'_j}\circ\dot\pi_j=\tilde F'_j|_{\ddot S'_0}=\dot Q_j(\dot F'_0)$ ($j=1$, $2$), where $\dot\pi_j:\ddot S'_0\to\ddot S'_j$ ($j=1$, $2$) are directly up-harmonious mappings. Therefore,
$$
\dot Q_1(\dot F'_0)=\tilde F'_1|_{\ddot S'_0}=\dot P(\tilde F'_2)|_{\ddot S'_0}=\dot P(\tilde F'_2|_{\ddot S'_0})=\dot P(\dot Q_2(\dot F'_0))
$$
by~(\ref{4.10}) and $\hat\varphi'_j=\hat Q_j(\hat\varphi'_0)$ ($j=1$, $2$). Consequently $\hat Q_1(\hat\varphi'_0))=\hat P(\hat Q_2(\hat\varphi'_0))$, i.e.
$$
\hat\varphi'_1=\hat P(\hat\varphi'_2).
$$
Suppose $\dot p':\ddot S'_2\to\ddot S'_1$ is the corresponding directly up-harmonious mapping (see Lemma~\ref{lemma3.6}). Then
\begin{equation*}
\begin{split}
\tilde F'_1|_{\ddot S'_1}\circ\dot p'&=\tilde F'_1|_{\ddot S'_2}=\dot P(\tilde F'_2)|_{\ddot S'_2} \\
&=\dot P(\tilde F'_2|_{\ddot S'_2})=\dot P(\dot F_2\circ\dot\sigma_2^{-1})=\dot P(\dot F_2)\circ\dot\sigma_2^{-1}
\end{split}
\end{equation*}
by (\ref{4.10}) and~(\ref{4.11}), i.e.\ $\dot F_1\circ\dot\sigma_1^{-1}\circ\dot p'=\dot F_1\circ\dot p\circ\dot\sigma_2^{-1}$ by (\ref{4.9}) and~(\ref{4.11}). Since $\dot F_1$ is the original basic function on~$\ddot S_1$, we get
\begin{equation}\label{4.12}
\dot\sigma_1\circ\dot p=\dot p'\circ\dot\sigma_2.
\end{equation}
Therefore, $\dot\sigma_1$ and~$\dot\sigma_2$ are directly compatible. In fact, $\dot\sigma_1$ and~$\dot\sigma_2$ are exactly and directly compatible since from~(\ref{4.12}) we can deduce
$
(\dot p')^{-1}\circ\dot\sigma_1=\dot\sigma_2\circ\dot p^{-1}
$
(see Remark~\ref{remark8}). By the reasoning above we can see that $\tilde\sigma$ is compatible and satisfies the two conditions of a mapping. Hence $\tilde\sigma$ is a mapping from $\tilde W$ to~$\tilde W$. We can also see that if $\operatorname{codom}(\dot\sigma_2)$ is over $\operatorname{codom}(\dot\sigma_1)$ then $\dot\sigma_2$ is exactly over~$\dot\sigma_1$. Therefore, it is easy to know that $\tilde\sigma$ is a biholomorphic transformation on~$\tilde W$. For $\tilde f\in M$, by~(\ref{4.8}) we have
$$
\mathsf G_{\tilde\sigma}(\tilde f)=\tilde f\circ\tilde\sigma^{-1}=(\tilde f|_{\ddot S}\circ\dot\sigma^{-1})^{\sim}=(\alpha(\tilde f)|_{\ddot S'})^{\sim}=\alpha(\tilde f),
$$
where $\ddot S=\operatorname{obdom}(\tilde f)$, $\dot\sigma=\dot\sigma(\alpha,\ddot S)\in\tilde\sigma$, $\ddot S'=\operatorname{obdom}(\alpha(\tilde f))$ and $(\dot g)^{\sim}$ denotes~$\tilde g=\gamma(\dot g)$, which is determined by~$\dot g$.

Suppose $\tilde\pi:\tilde W\to\tilde Z$ is the natural covering map and $\dot\pi\in\tilde\pi$ is the natural covering map from $\ddot W\in\tilde W^0$ to $\ddot Z\in\tilde Z^0$ (thus $\ddot Z\in\tilde W^0$). In the above reasoning, let $\ddot S_1=\ddot Z$ and $\ddot S_2=\ddot W$. Since $\alpha|_L=\id_L$, letting $\tilde F_1\in L$ be determined by the original basic function~$\dot F_1$ on~$\ddot S_1$, then $\alpha(\tilde F_1)=\tilde F_1$ and $\ddot S'_1=\operatorname{obdom}(\alpha(\tilde F_1))=\ddot S_1$ by Lemma~\ref{lemma4.2}. By~(\ref{4.8}), letting $\dot F=\dot F_1$ and $\dot F'=\alpha(\tilde F_1)|_{\ddot S'_1}=\dot F_1$, we have $\dot\sigma_1=\id_{\ddot S_1}$. Then we obtain a base-preserving biholomorphic mapping $\dot\sigma=\dot\sigma_2:\ddot W\to\ddot W'$ ($\ddot W'=\ddot S'_2$) such that $\dot\pi'\circ\dot\sigma=\dot\pi$ (see~(\ref{4.12})), where $\dot\pi':\ddot W'\to\ddot Z$ is a maximal natural covering map in~$\tilde\pi$, by the preceding reasoning in the proof (in fact, $\dot\sigma(\alpha,\dot W):\dot W\to\dot W'=\dot W$ is an identical mapping by Remark~\ref{remark14}). Easily we see that $\tilde\pi\circ\tilde\sigma=\tilde\pi$, which means $\tilde\sigma\in\deck(\tilde W\!/\tilde Z)$. Therefore, $\mathcal G$ is surjective.

(iii) In the following we prove that $\tilde W\!/\tilde Z$ is Galois if and only if $M/L$ is Galois.

Now we assume the natural covering map $\tilde\pi:\tilde W\to\tilde Z$ is Galois. For $\tilde f\in M$ we will prove that the minimal polynomial $\min(L,\tilde f)$ of~$\tilde f$ over~$L$ splits in~$M$. Suppose $\min(L,\tilde f)=t^n+\tilde c_1t^{n-1}+\cdots+\tilde c_n$. Suppose $L_0=\mathcal M(\dot R)(\tilde c_1,\dots\tilde c_n)$ and $L_0\leqslant L_1\leqslant L$ (for fields $L_1$ and~$L_2$, by $L_1\leqslant L_2$ we mean $L_1$ is a subfield of~$L_2$). Then $\min(L_1,\tilde f)=\min(L,\tilde f)$ and there exists $\tilde g\in L_0$ such that $L_0=\mathcal M(\dot R)(\tilde g)$. Suppose $M_0=L_0(\tilde f)$. Then there exists $\tilde h\in M_0$ such that $M_0=\mathcal M(\dot R)(\tilde h)$. Let $\ddot W_0=\operatorname{obdom}(\tilde h)$ and $\ddot Z_0=\operatorname{obdom}(\tilde g)$. Then $\ddot Z_0\in\tilde Z$ and $\ddot W_0\in\tilde W$ by Lemma~\ref{lemma4.2}. Since $\tilde\pi:\tilde W\to\tilde Z$ is Galois, there exists a Galois covering map $\dot\pi_1:\ddot W_1\to\ddot Z_1$ in~$\tilde\pi$ with $\ddot Z_0\leqslant\ddot Z_1$ and $\ddot W_0\leqslant\ddot W_1$. Therefore $M_1/L_1$, where $M_1=\mathcal M(\tilde{\ddot W}_1)$ and $L_1=\mathcal M(\tilde{\ddot Z}_1)$, is Galois by Lemma~\ref{lemma4.4}(3), which means $\min(L,\tilde f)=\min(L_1,\tilde f)$ splits in~$M_1$, hence in~$M$.

At last, assume $M/L$ is Galois. For a natural covering map $\dot\pi_0:\ddot W_0\to\ddot Z_0$ in~$\tilde\pi$, where $\ddot W_0\in\tilde W$ and $\ddot Z_0\in\tilde Z$, let $M_0=\mathcal M(\tilde{\ddot W}_0)$ and $L_0=\mathcal M(\tilde{\ddot Z}_0)$. Then there exists $M_1\leqslant M$ such that $M_1/M_0$ is a finite extension and $M_1/L'_0$ is a Galois extension, where $M_0=\mathcal M(\dot R)(\tilde f_1)$, $\min(L,\tilde f_1)=t^n+\tilde c_1t^{n-1}+\cdots+\tilde c_n$ and $L'_0=L_0(\tilde c_1,\dots\tilde c_n)$, since $[M_0:\mathcal M(\dot R)]$ is finite and $M/L$ is Galois. In fact, we may let $M_1=L'_0(\tilde f_1,\dots,\tilde f_n)$, where $\tilde f_1$, $\dots$, $\tilde f_n$ are all the roots of $\min(L,\tilde f_1)$ in~$M$. Suppose $M_1=\mathcal M(\dot R)(\tilde g)$ for some $\tilde g\in M_1$ and $\ddot W_1=\operatorname{obdom}(\tilde g)$. Take a natural covering map $\dot\pi_1:\ddot W_1\to\ddot Z_1$ in~$\tilde\pi$ over $\dot\pi_0:\ddot W_0\to\ddot Z_0$, where $\ddot Z_1\in\tilde Z$, and let $L_1=\mathcal M(\tilde{\ddot Z}_1)$. Then by the definition of natural covering maps of algebraic Riemann surfaces (the definition before Remark~\ref{remark10}) and by Lemma~\ref{lemma4.3} we have $L'_0\leqslant L_1\leqslant M_1$ and so $M_1/L_1$ is Galois. Therefore, $\dot\pi_1:\ddot W_1\to\ddot Z_1$ is Galois by Lemma~\ref{lemma4.4}(3).
\end{proof}

Suppose $\tilde X$ and~$\tilde Y$ are algebraic Riemann surfaces and $\tilde Y$ is Galois over~$\tilde X$. Suppose $\tilde Z\in\operatorname{Int}(\tilde Y\!/\tilde X)$ and $\tilde\sigma:\tilde Y\to\tilde Y'$ is a mapping, where $\tilde Y'$ is some algebraic Riemann surface. Define
$$
\tilde\sigma|_{\tilde Z}:=\{\tilde\sigma|_{\ddot Z}:\ddot Z\in\tilde Z\}.
$$
If $\tilde\sigma|_{\tilde Z}:\tilde Z\to\tilde Z'$ is still a mapping, where $\tilde Z'$ is some algebraic Riemann surface under~$\tilde Y'$, then we call $\tilde\sigma|_{\tilde Z}$ a \textit{restriction} of~$\tilde\sigma$ to~$\tilde Z$.

By Lemma~\ref{lemma4.5} and \cite[Theorem~3.28]{Mo} and by the reasoning in part~(ii) of the proof of Lemma~\ref{lemma4.5} we deduce

\begin{lemma}\label{lemma4.6}
Suppose $\tilde Y$ is Galois over~$\tilde X$ and $\tilde Z\in\operatorname{Int}(\tilde Y\!/\tilde X)$. If $\tilde Z/\tilde X$ is Galois and $\tilde\sigma\in\deck(\tilde Y\!/\tilde X)$, then $\tilde\sigma|_{\tilde Z}\in\deck(\tilde Z/\tilde X)$ and for $\tilde\tau\in\deck(\tilde Z/\tilde X)$ there is a $\tilde\sigma\in\deck(\tilde Y\!/\tilde X)$ with $\tilde\sigma|_{\tilde Z}=\tilde\tau$. \hfill $\square$
\end{lemma}

We assume $G=\gal(\mathcal M(\tilde Y)/\mathcal M(\tilde X))$ possesses the Krull topology (see \cite[Definition~17.5]{Mo}). Recall that $D=\deck(\tilde Y\!/\tilde X)$ has been given a similar topological structure in the beginning of this section. By Lemma~\ref{lemma4.5} we see

\begin{lemma}\label{lemma4.7}
Suppose $\tilde X\leqslant\tilde Z\leqslant\tilde Y$. Let $K=\mathcal M(\tilde Y)$, $N=\mathcal M(\tilde X)$ and $L=\mathcal M(\tilde Z)$. Then $\mathcal G=\mathcal G_{\tilde Y\!/\tilde X}$ is an isomorphism and a homeomorphism from $D=\deck(\tilde Y\!/\tilde X)$ to $G=\gal(K/N)$ and $\mathcal G|_E$ is an isomorphism from $E=\deck(\tilde Y\!/\tilde Z)$ to $H=\gal(K/L)$. \hfill $\square$
\end{lemma}

By Lemma~\ref{lemma4.7} and \cite[Theorem~17.6]{Mo} we have

\begin{proposition}\label{proposition4.8}
Suppose $D=\deck(\tilde Y\!/\tilde X)$ and $\mathscr K_D$ is its Krull topology. Then the topological space $(D,\mathscr K_D)$ is Hausdorff, compact and totally disconnected. \hfill $\square$
\end{proposition}

For $H\leqslant G$ let $\mathscr F(H)$ denote the fixed field of~$H$. Define
\begin{equation}\label{4.13}
\Gamma:=\mathcal R\circ\mathscr F\circ\mathcal G,
\end{equation}
where $\mathcal R$ and~$\mathcal G=\mathcal G_{\tilde Y\!/\tilde X}$ are defined in Lemmas~\ref{lemma4.3} and~\ref{lemma4.5} respectively.

\begin{lemma}\label{lemma4.9}
For $E\leqslant D=\deck(\tilde Y\!/\tilde X)$ we have the closure
$$
\overline{\!E}=\deck(\tilde Y\!/\Gamma(E)).
$$
Thus $E$ is closed precisely if $E=\deck(\tilde Y\!/\Gamma(E))$.
\end{lemma}

\begin{proof}
Denote $H=\mathcal G(E)$. Suppose $L=\mathscr F(H)$ and $\tilde Z=\mathcal R(L)$. Then $\mathcal M(\tilde Z)=L$ by Lemma~\ref{lemma4.3} and
$$
\tilde Z=(\mathcal R\circ\mathscr F)(H)=(\mathcal R\circ\mathscr F\circ\mathcal G)(E)=\Gamma(E).
$$
Therefore, by Lemma~\ref{lemma4.7} and \cite[Theorem~17.7]{Mo} it follows that
\begin{equation*}
\begin{split}
\overline{\!E}
=\mathcal G^{-1}(\,\overline{\!H})&=\mathcal G^{-1}(\gal(K/\mathscr F(H)) \\
&=\mathcal G^{-1}(\gal(K/L))=\deck(\tilde Y\!/\tilde Z)=\deck(\tilde Y\!/\Gamma(E)),
\end{split}
\end{equation*}
where $K=\mathcal M(\tilde Y)$.
\end{proof}

\par\smallskip
\noindent
\textit{Proof of Theorem~$\ref{theorem4.1}$.} \,Suppose $\tilde Z\in\operatorname{Int}(\tilde Y\!/\tilde X)$ and $E=\deck(\tilde Y\!/\tilde Z)$. Let $K=\mathcal M(\tilde Y)$, $N=\mathcal M(\tilde X)$ and $L=\mathcal M(\tilde Z)$. Then $\tilde Z=\mathcal R(L)$ by Lemma~\ref{lemma4.3} and $K/N$ is Galois by Lemma~\ref{lemma4.5}. Hence, $K/L$ is Galois. Let $H=\mathcal G(E)$. Then we have $H=\gal(K/L)$ by Lemma~\ref{lemma4.5}. Hence by \cite[Lemma~2.9(6)]{Mo} it follows $H=\gal(K/\mathscr F(H))$, which means that $H$ is closed by \cite[Theorem~17.7]{Mo} and so is~$E$ by Lemma~\ref{lemma4.7}. By Lemma~\ref{lemma4.5} and the fundamental theorem of infinite Galois theory (Krull's, see \cite[Theorem~17.8]{Mo}) (or \cite[Definition~2.15]{Mo}, since $K/L$ is Galois) we have
\begin{equation*}
\begin{split}
(\Gamma\circ\Delta)(\tilde Z)&=(\mathcal R\circ\mathscr F\circ\mathcal G)(\deck(\tilde Y\!/\tilde Z)) \\
&=(\mathcal R\circ\mathscr F)(\gal(K/L))=\mathcal R(L)=\tilde Z.
\end{split}
\end{equation*}

On the other hand, for a closed subgroup~$E$ of $D=\deck(\tilde Y\!/\tilde X)$, let $\tilde Z=\Gamma(E)$. Then $\tilde Z\in\operatorname{Int}(\tilde Y\!/\tilde X)$ by Lemmas~\ref{lemma4.3} and~\ref{lemma4.7} and \cite[Theorem~17.8]{Mo} and by Lemma~\ref{lemma4.9} we have $E=\deck(\tilde Y\!/\tilde Z)=\Delta(\tilde Z)$. Therefore,
$$
(\Delta\circ\Gamma)(E)=\Delta(\tilde Z)=E.
$$

By Lemma~\ref{lemma4.7} we know that $[D:E]=[G:H]$ and that $E$ is open in~$D$ if and only if $H$ is open in~$G$. Thus, by \cite[Theorem~17.8]{Mo}, assertion~(1) in Theorem~\ref{theorem4.1} follows since $[\tilde Z:\tilde X]=[L:N]$, where $L=\mathcal M(\tilde Z)$ and $N=\mathcal M(\tilde X)$.

By Lemmas~\ref{lemma4.5} and~\ref{lemma4.7} we have the following isomorphisms
$$
\deck(\tilde Z/\tilde X)\cong\gal(L/N) \quad\text{and}\quad D/E\cong G/H
$$
when $E$ is normal in~$D$, which are also homeomorphisms. Hence, assertion~(2) in Theorem~\ref{theorem4.1} follows by Lemmas~\ref{lemma4.5} and~\ref{lemma4.7} and \cite[Theorem~17.8]{Mo}.

According to Lemma~\ref{lemma4.6} and noticing that when $\tilde\sigma\in\deck(\tilde Y\!/\tilde X)$ and $\tilde Z/\tilde X$ is Galois we have $\tilde\sigma|_{\tilde Z}=\id_{\tilde Z}$ if and only if $\tilde\pi\circ\tilde\sigma=\tilde\pi$ (refer to part~(ii) of the proof of Lemma~\ref{lemma4.5}), where $\tilde\pi:\tilde Y\to\tilde Z$ is the natural covering map, we can also deduce that if $\tilde Z/\tilde X$ is Galois then the mapping from $D/E$ to $\deck(\tilde Z/\tilde X)$ defined by $\tilde\sigma E\mapsto\tilde\sigma|_{\tilde Z}$ is both an isomorphism and a homeomorphism by similar reasoning to the proof of \cite[Theorem~17.8]{Mo} and by Proposition~\ref{proposition4.8}. 
\hfill $\square$

\par\medskip


For the finite case, as a corollary of Theorem~\ref{theorem4.1}, we have
\begin{theorem}[Galois Correspondence on Algebraic Riemann Surfaces in the Finite Case]\label{theorem4.10}
Suppose $\tilde X$ and~$\tilde Y$ are algebraic Riemann surfaces and $\tilde Y$ is a finite Galois covering of~$\tilde X$. Let $D=\deck(\tilde Y\!/\tilde X)$. Then the mapping
\begin{equation*}
\begin{split}
\Delta:\,\operatorname{Int}(\tilde Y\!/\tilde X)&\to\mathfrak C \\
\tilde Z&\mapsto\deck(\tilde Y\!/\tilde Z)
\end{split}
\end{equation*}
is a bijection, which gives an inclusion reversing correspondence and whose inverse mapping is\/ $\Gamma$ given in~$(\ref{4.13})$. Moreover, if $E=\Delta(\tilde Z)$ then
\medskip\par\noindent
$(1)$ $[\tilde Y:\tilde Z]=|E|$ $($the order of~$E)$ and $[\tilde Z:\tilde X]=[D:E];$
\medskip\par\noindent
$(2)$ $\tilde Z$ is Galois over~$\tilde X$ if and only if $E$ is normal in~$D$, and on this condition we have a group isomorphism
$$
\deck(\tilde Z/\tilde X)\cong D/E. \eqno \square
$$
\end{theorem}

\vskip3mm

School of Mathematical Sciences,
Shanghai Jiao Tong University,

Shanghai 200240, China

\textit{E-mail}: jyyu@sjtu.edu.cn

\end{document}